\newif\ifisarxiv
\isarxivtrue

\ifisarxiv
\documentclass[11pt]{article}
\usepackage{fullpage}
\else
\documentclass{article}
\usepackage{neurips_2021}
\fi

\usepackage[utf8]{inputenc} 
\usepackage[T1]{fontenc}    
\usepackage{hyperref}       
\usepackage{url}            
\usepackage{booktabs}       
\usepackage{amsfonts}       
\usepackage{nicefrac}       
\usepackage{microtype}      
\usepackage{xcolor}         
\usepackage{wrapfig}
\usepackage{enumitem}
\usepackage{graphicx}
\usepackage[linesnumbered,ruled]{algorithm2e}
\usepackage{algorithmic, float}
\usepackage{subfigure}
\usepackage{caption}
\usepackage{tikz,pgfplots}
\hypersetup{
	colorlinks,
	linkcolor={red!40!gray},
	citecolor={blue!40!gray},
	urlcolor={blue!70!gray}
}

\usepackage{amsmath, amssymb, amsthm}

\newcommand{\Af}[2]{{\A_{#1}(#2)}} 
\def\nnz{{\mathrm{nnz}}}

\def\Qbar{{\bar{\mathbf Q}}}
\newcommand{\deff}{{d_{\textnormal{eff}}}}
\newcommand{\tdeff}{{\tilde d_{\textnormal{eff}}}}

\def\xib{\boldsymbol\xi}
\def\Sigmab{\mathbf{\Sigma}}

\def\gammat{{\tilde\gamma}}

\def\xbt{\widetilde{\x}}

\def\s {\mathbf{s}}
\def\g {\mathbf{g}}

\def\Ec{\mathcal{E}}
\def\p{\mathbf p}

\def\R{\mathbf R}
\def\H{\mathbf H}

\ifx\BlackBox\undefined
\newcommand{\BlackBox}{\rule{1.5ex}{1.5ex}}  
\fi

\DeclareMathOperator*{\argmin}{\mathop{\mathrm{argmin}}}

\DeclareMathOperator*{\diag}{\mathop{\mathrm{diag}}}
\def\x{\mathbf x}

\def\b{\mathbf b}

\def\v{\mathbf v}

\def\p{\mathbf{p}}

\def\zero{\mathbf 0}

\def\u{\mathbf u}

\def\X{\mathbf X}

\def\B{\mathbf B}
\def\A{\mathbf A}
\def\C{\mathbf C}
\def\U{\mathbf U}

\def\Ubbar{\bar{\mathbf U}}

\def\D{\mathbf D}

\def\Deltat{{\widetilde\Delta}}

\def\V{\mathbf V}
\def\M{\mathbf M}

\def\T{\mathbf T}
\def\S{\mathbf S}

\def\Sbt{\widetilde{\mathbf{S}}}

\def\I{\mathbf I}

\def\A{\mathbf A}

\def\Q{\mathbf Q}
\def\Qbt{\widetilde{\mathbf Q}}

\def\E{\mathbb E}

\def\R{\mathbb R} 
\def\tr{\mathrm{tr}}

\def\Var{\mathrm{Var}}

\let\origtop\top
\renewcommand\top{{\scriptscriptstyle{\origtop}}} 

\definecolor{silver}{cmyk}{0,0,0,0.3}
\definecolor{yellow}{cmyk}{0,0,0.9,0.0}
\definecolor{reddishyellow}{cmyk}{0,0.22,1.0,0.0}
\definecolor{black}{cmyk}{0,0,0.0,1.0}
\definecolor{darkYellow}{cmyk}{0.2,0.4,1.0,0}
\definecolor{darkSilver}{cmyk}{0,0,0,0.1}
\definecolor{grey}{cmyk}{0,0,0,0.5}
\definecolor{darkgreen}{cmyk}{0.6,0,0.8,0}

\newcommand{\Green}[1]{{\color{darkgreen}  {#1}}}
\newcommand{\Blue}[1]{\color{blue}{#1}\color{black}}
\newcommand{\Brown}[1]{{\color{brown}{#1}\color{black}}}

\ifx\proof\undefined
\newenvironment{proof}{\par\noindent{\bf Proof\ }}{\hfill\BlackBox\\[2mm]}
\fi

\ifx\theorem\undefined
\newtheorem{theorem}{Theorem}
\fi

\ifx\example\undefined
\newtheorem{example}{Example}
\fi

\ifx\property\undefined
\newtheorem{property}{Property}
\fi

\ifx\lemma\undefined
\newtheorem{lemma}[theorem]{Lemma}
\fi
\ifx\condition\undefined
\newtheorem{condition}{Condition}
\fi

\ifx\proposition\undefined
\newtheorem{proposition}[theorem]{Proposition}
\fi

\ifx\remark\undefined
\newtheorem{remark}[theorem]{Remark}
\fi

\ifx\corollary\undefined
\newtheorem{corollary}[theorem]{Corollary}
\fi

\ifx\definition\undefined
\newtheorem{definition}{Definition}
\fi

\ifx\conjecture\undefined
\newtheorem{conjecture}[theorem]{Conjecture}
\fi

\ifx\axiom\undefined
\newtheorem{axiom}[theorem]{Axiom}
\fi

\ifx\claim\undefined
\newtheorem{claim}[theorem]{Claim}
\fi

\ifx\assumption\undefined
\newtheorem{assumption}[theorem]{Assumption}
\fi

\newcommand{\real}{\mathbb{R}}

\title{Newton-LESS: Sparsification without Trade-offs \\
  for the Sketched Newton Update}

\ifisarxiv
\author{
  Micha{\l} Derezi\'nski\thanks{Department of Statistics, University
    of California, Berkeley (\texttt{mderezin@berkeley.edu})}
  \and
  Jonathan Lacotte\thanks{Department of Electrical Engineering,
    Stanford University (\texttt{lacotte@stanford.edu})}
  \and
  Mert Pilanci\thanks{Department of Electrical Engineering,
    Stanford University (\texttt{pilanci@stanford.edu})}
  \and
  Michael W. Mahoney\thanks{ICSI and Department of Statistics,
  University of California, Berkeley (\texttt{mmahoney@stat.berkeley.edu})}
  }
\else
\author{%
  \textbf{Micha{\l} Derezi\'nski}\\
\small  Department of Statistics\\
\small  University of California, Berkeley\\
\small  \texttt{mderezin@berkeley.edu}\\
  \And
  \textbf{Jonathan Lacotte}\\
  Department of Electrical Engineering\\
  Stanford University\\
  \texttt{}
  \And
  \textbf{Mert Pilanci}\\
  Department of Electrical Engineering \\
  Stanford University \\
  \And
  \textbf{Michael W. Mahoney}\\
  ICSI and Department of Statistics \\
  University of California, Berkeley
}
\fi

\begin{document}

\maketitle

\begin{abstract}
In second-order optimization, a potential bottleneck can be computing the
Hessian matrix of the optimized function at every iteration. Randomized
sketching has emerged as a powerful technique for constructing
estimates of the Hessian which can be used to perform approximate
Newton steps. This involves multiplication by a
random sketching matrix, which introduces a trade-off between the
computational cost of sketching and the convergence rate of the
optimization algorithm. A theoretically desirable but practically much too
expensive choice is to use a dense Gaussian sketching matrix, which
produces unbiased estimates of the exact Newton step and which offers strong
problem-independent convergence guarantees. We show that the Gaussian
sketching matrix can be drastically sparsified, significantly reducing the
computational cost of sketching, without substantially affecting its convergence
properties. This approach, called Newton-LESS, is based on a recently introduced
sketching technique: LEverage Score Sparsified (LESS)
embeddings. We prove that Newton-LESS enjoys nearly the same
problem-independent local convergence rate as Gaussian embeddings, not
just up to constant factors  but even down to lower order terms, for
a large class of optimization tasks. In 
particular, this leads to a new state-of-the-art convergence result for
an iterative least squares solver. Finally,
we extend LESS embeddings to include uniformly sparsified random sign matrices
which can be implemented efficiently and which perform well in numerical experiments. 
\end{abstract}

\section{Introduction}

Consider the task of minimizing a twice-differentiable convex function
$f:\R^d\rightarrow \R$:
\vspace{-.5mm}
\begin{align*}
  \text{find}\quad \x^*=\argmin_{\x\in\R^d} f(\x).
\end{align*}
%
One of the most classical iterative algorithms for solving this task is the
Newton's method, which takes steps of the form $\x_{t+1}=\x_t-\mu_t\nabla^2 f(\x_t)^{-1}\nabla
f(\x_t)$, and which leverages second-order information in the $d\times
d$ Hessian matrix $\nabla^2 f(\x_t)$ to achieve rapid
convergence, especially locally as it approaches the optimum
$\x^*$. However, in many settings, the cost of forming the exact Hessian is
prohibitively expensive, particularly when the function $f$ is given
as a sum of $n\gg d$ components, i.e., $f(\x)=\sum_{i=1}^nf_i(\x)$. This
commonly arises in machine learning when $f$ represents the training
loss over a dataset of $n$ elements, as well as in solving
semi-definite programs, portfolio optimization, and other tasks. In
these contexts, we can represent the Hessian via a decomposition
$\nabla^2 f(\x)=\Af{f}{\x}^{\top}\Af{f}{\x}$,
where $\Af{f}{\x}$ is a tall $n\times d$ matrix,
which can be easily formed, and the main bottleneck is the 
matrix multiplication which takes $O(nd^2)$ arithmetic operations. To
avoid this bottleneck, many randomized second-order methods have been proposed
which use a Hessian estimate in place of the exact Hessian (e.g., \cite{byrd2011use,erdogdu2015convergence,naman17,subsampled-newton-math-prog}). This
naturally leads to a trade-off between the per-iteration
cost of the method and the number of iterations needed to reach convergence. We
develop Newton-LESS, a randomized second-order method which eliminates the computational
bottleneck while minimizing the convergence trade-offs.

An important family of approximate second-order methods is known as
the Newton Sketch~\cite{pilanci2017newton}:
\begin{align}
  \xbt_{t+1} = \xbt_t - \mu_t\big(\Af{f}{\xbt_t}^\top\S_t^\top\S_t \Af{f}{\xbt_t}\big)^{-1}\nabla f(\xbt_t),\label{eq:newton-sketch}
\end{align}
where $\mu_t$ is the step size, and $\S_t$ is a random $m\times n$ sketching matrix, with $m\ll n$,
that is used to reduce $\Af{f}{\xbt_t}$ to a small 
$m\times d$ sketch $\S_t \Af{f}{\xbt_t}$. This brings the complexity
of forming the Hessian down to $O(md^2)$ time plus the cost of forming the
sketch.

Naturally, sketching methods vary in their computational cost and they
can affect the convergence rate, so the right choice of $\S_t$ depends
on the computation-convergence trade-off.
On
one end of this spectrum are the so-called Sub-Sampled Newton 
methods~\cite{subsampled-newton-math-prog,XRM17_theory_TR,YXRM18_TR}, where
$\S_t$ simply selects a random sample of $m$ rows of $\Af{f}{\x}$
(e.g., a sample of data points in a training set) to 
form the sketch. Here the sketching cost is negligible, since $\S_t$
is extremely sparse, but the
convergence rate can be highly variable and problem-dependent. On the
other end, we have what we will refer to as the Gaussian Newton
Sketch, where $\S_t$ is a dense matrix with i.i.d.\ scaled Gaussian
entries (a.k.a.~a Gaussian embedding). While the $O(mnd)$
cost of performing this sketch limits its practical appeal, Gaussian
Newton Sketch has a number of unique and desirable properties \cite{lacotte2019faster}: it
enjoys strong problem-independent convergence rates; it produces unbiased
estimates of the exact Newton update (useful in distributed settings);
and it admits analytic expressions for the optimal step~size.

A natural way to interpolate between these two extremes is to
vary the sparsity $s$ of the sketching matrix $\S_t$, from $s=1$ non-zero
element per row (Sub-Sampling) to $s=n$ non-zero elements (Gaussian embedding),
with the sketching complexity $O(mds)$.%
\footnote{For a more detailed discussion of other sketching techniques that
  may not fit this taxonomy, such as the Subsampled Randomized
  Hadamard Transform and the CountSketch, see Section \ref{s:related-work}.}
Motivated by this, we ask:
\begin{quote}
  Can we sparsify the Gaussian embedding, making its sparsity closer
  to that of Sub-Sampling, without suffering \emph{any} convergence trade-offs?
\end{quote}

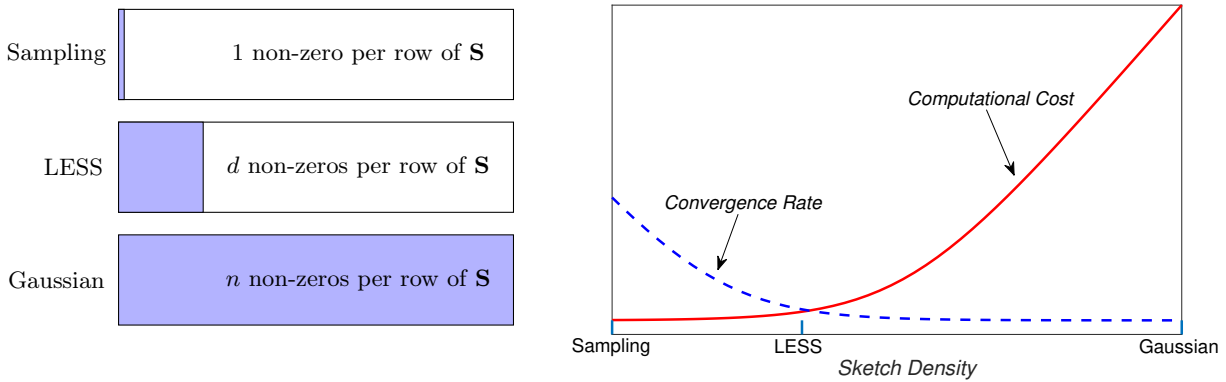
\begin{figure}
\begin{center}
  \begin{tikzpicture}[scale=0.75]
    \draw[white] (0,1.25) rectangle (7,1.25);
    
    \draw (0,6.2) rectangle (7,7.8);
    \draw[fill=blue!30] (0,6.2) rectangle (.1,7.8);
    \draw (4.25,7) node {\mbox{\footnotesize $1$ non-zero per row of $\S$}};
    \draw (-1.1,7) node {\mbox{\footnotesize Sampling}};
    \draw (0,4.2) rectangle (7,5.8);
    \draw[fill=blue!30] (0,4.2) rectangle (1.5,5.8);
    \draw (4.25,5) node {\mbox{\footnotesize $d$ non-zeros per row of $\S$}};
    \draw (-.8,5) node {\mbox{\footnotesize LESS}}; 
    \draw[fill=blue!30] (0,2.2) rectangle (7,3.8);
    \draw (4.25,3) node {\mbox{\footnotesize $n$ non-zeros per row of $\S$}};
    \draw (-1.1,3) node {\mbox{\footnotesize Gaussian}};
  \end{tikzpicture}\hspace{5mm}%
  \includegraphics[width=.55\textwidth]{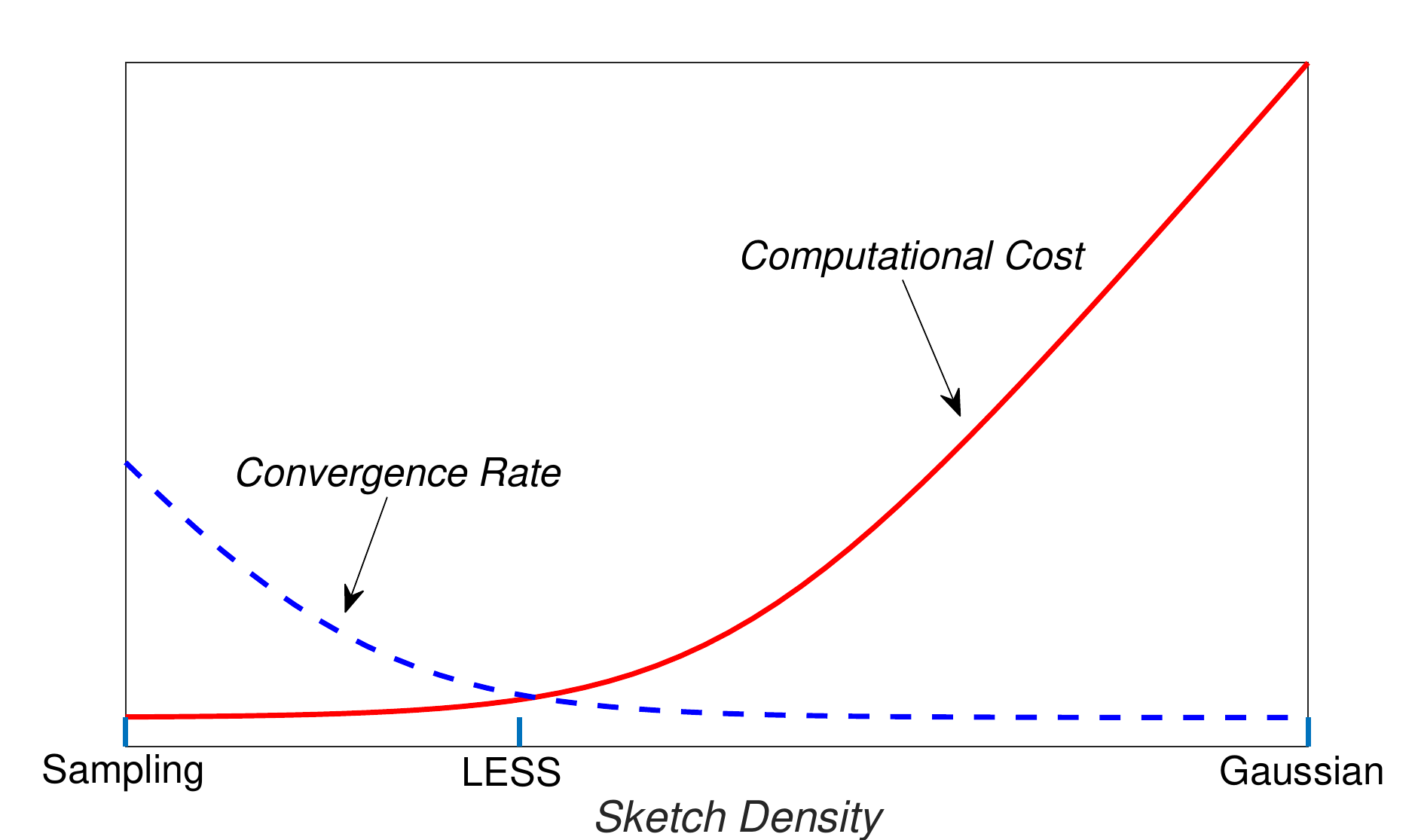}
  \captionof{figure}{The effect of the density of $m\times n$
    sketching matrix $\S$ applied to an $n\times d$ matrix $\A$
    (with $d,m\ll n$) on the convergence rate of Newton Sketch
    and the computational cost of constructing the Hessian estimate. LESS embeddings 
    ``interpolate'' between Sub-Sampled Newton methods and 
    Gaussian Newton Sketches, achieving a ``sweet spot'' in 
    the computation-per-iteration versus number-of-iterations
    tradeoff.} 
\label{fig:sparsity-intro}
\end{center}
\vspace{-3mm}
\end{figure}

In this paper, we provide an affirmative answer to this question.
We show that it is possible to drastically sparsify the Gaussian
embedding so that 
two key statistics of the sketches, namely first and second
inverse moments of the sketched Hessian, are nearly preserved in a
very strong sense. Namely, the two inverse moments of the sparse sketches can be upper and
lower bounded by the corresponding quantities for the dense
Gaussian embeddings, where the upper/lower bounds are matching not just up to
constant factors, but down to lower order terms (see Theorem
\ref{t:structural}). 
We use this to show that the Gaussian Newton Sketch can be  
sparsified to the point where the cost of sketching is proportional to
the cost of other operations, while nearly preserving the convergence rate
(again, down to lower order terms; see Theorem \ref{t:main-simple}). 
This is illustrated conceptually in Figure~\ref{fig:sparsity-intro}%
, showing how
the sparsity of the sketch affects the per-iteration convergence rate as well as the
computational cost of the Newton Sketch. We observe that while the convergence rate
improves as we increase the density, it eventually flattens out. On
the other hand, the computational cost stays largely flat until some
point when it starts increasing at a linear rate. As a result, there
is a sparsity regime where we achieve the best of 
both worlds:  the convergence rate and the computational cost are both nearly at
their optimal values, thereby avoiding any trade-off.

To establish our results, we build upon a recently introduced
sketching technique called LEverage Score Sparsified (LESS) embeddings
\cite{less-embeddings}. LESS embeddings use leverage score
techniques~\cite{fast-leverage-scores} to provide a  
carefully-constructed (random) sparsification pattern. This is used to
produce a sub-Gaussian embedding with $d$ 
non-zeros per row of $\S_t$ (as opposed to $n$ for a dense matrix), so
that the cost of forming the sketch $\S_t\Af{f}{\xbt_t}$ matches the
cost of constructing the 
Hessian estimate, i.e., $O(md^2)$ (see Section \ref{s:preliminaries}). 
\cite{less-embeddings} analyzed the first inverse moment of the sketch
to show that LESS embeddings retain certain
unbiasedness properties of Gaussian embeddings. In our setting, this
captures the bias of the Newton 
Sketch, but it does not capture the variance, which is needed to control the convergence rate.

\paragraph{Contributions.}
In this paper, we analyze both the bias and the variance of Newton
Sketch with LESS embeddings (Newton-LESS; see Definition
\ref{d:less} and Lemma \ref{l:structural-less}), resulting in a
comprehensive convergence analysis. The following are our key
contributions:
\begin{enumerate}
\item Characterization of the second inverse moment of the
  sketched Hessian for a class of sketches
  including sub-Gaussian matrices and LESS embeddings;
\item
Precise problem-independent local
convergence rates for Newton-LESS, matching the Gaussian Newton Sketch
down to lower order terms;
\item Extension of Newton-LESS to \emph{regularized} minimization
  tasks, with improved dimension-independent guarantees for the sketch sparsity and convergence rate;
\item \emph{Notable corollary:} Best known
 global convergence rate for an iterative least squares solver,
 which translates to state-of-the-art numerical performance. 
\end{enumerate}

\subsection{Main results}
As our main contribution, we show
that, under standard assumptions on the function $f(\x)$, Newton-LESS
achieves the same problem-independent local convergence rate as the
Gaussian Newton Sketch, despite drastically smaller per-iteration
cost.
\begin{theorem}\label{t:main-simple}
Assume that $f(\x)$ is (a) self-concordant, or
  (b) has a Lipschitz
  continuous Hessian. Also, let $\H=\nabla^2 f(\x^*)$ be
positive definite. There is a neighborhood $U$ 
  containing $\x^*$ such that if $\xbt_0\in U$, then Newton-LESS
with sketch size $m\geq Cd\log(dT/\delta)$ and step size
$\mu_t=1-\frac dm$ satisfies:
\vspace{-1mm}
  \begin{align*}
    \bigg(\E_\delta\,\frac{\|\xbt_T-\x^*\|_{\H}^2}{\|\xbt_0-\x^*\|_{\H}^2}\bigg)^{1/T}
    \approx_\epsilon\ \frac dm\qquad\text{for}\quad \epsilon =
    O\Big(\frac1{\sqrt d}\Big),
  \end{align*}
  where $\E_\delta\,X$ is expectation conditioned on an
  event that holds with a $1-\delta$ probability, $\|\v\|_{\M}=\sqrt{\v^\top\M\v}$, and $a\approx_\epsilon b$
  means that $|a-b|\leq\epsilon b$.
\end{theorem}
\begin{remark}
The same guarantee holds for the Gaussian Newton Sketch, but it is not
known for any fast sketching method other than Newton-LESS (see Section \ref{s:related-work}).  The
alternative assumptions of self-concordance and Lipschitz 
continuous Hessian are standard in the local convergence analysis of the
classical Newton's method, and they only affect the size of the
neighborhood $U$ (see Section~\ref{s:convergence}). Global convergence
of Newton-LESS follows from existing analysis of the Newton
Sketch \cite{pilanci2017newton}.
\end{remark}
The notion of expectation $\E_\delta$ allows us to accurately capture the
average behavior of a randomized algorithm over a moderate (i.e., polynomial in $d$) number of
trials even when the true expectation is not well behaved. Here, this
guards against the (very unlikely, but non-zero) possibility that the
Hessian estimate produced by a sparse sketch will be ill-conditioned.

To illustrate this result in a special case (of obvious independent
interest), we provide a simple corollary for the least squares
regression task, i.e., $f(\x)=\frac12\|\A\x-\b\|^2$. Importantly,
here the convergence rate of $(\frac dm)^T$ holds
\emph{globally}. Also, for this task we have
$\frac12\|\x-\x^*\|_{\H}^2=f(\x)-f(\x^*)$, so the convergence can be stated in
terms of the excess function value. To our knowledge, this is the best
known convergence guarantee for a fast iterative least squares solver.
\begin{corollary}
  Let $f(\x)=\frac12\|\A\x-\b\|^2$ for $\A\in\R^{n\times d}$ and
  $\b\in\R^n$. Then, given any $\xbt_0\in\R^d$, Newton-LESS with sketch
  size $m\geq Cd\log(dT/\delta)$ and step size $\mu_t=1-\frac dm$
  satisfies:
  \begin{align*}
    \bigg(\E_\delta\,\frac{f(\xbt_T)-f(\x^*)}{f(\xbt_0)-f(\x^*)}\bigg)^{1/T}
    \approx_\epsilon\ \frac dm\qquad\text{for}\quad \epsilon =
    O\Big(\frac1{\sqrt d}\Big).
  \end{align*}
\end{corollary}
Prior to this work, a convergence rate of $(\frac dm)^T$ was
known only for dense Gaussian embeddings, and only for the least squares
task \cite{lacotte2019faster}. On the other hand, our results apply as generally as the standard
local convergence analysis of the Newton's method, and they include a
broad class of sketches. In Section~\ref{s:equivalence}, we provide
general structural conditions on a randomized sketching matrix that are
needed to enable our analysis. These conditions are satisfied
by a wide range of sketching methods, including all sub-Gaussian
embeddings (e.g., using random sign entries instead of Gaussians),
the original LESS embeddings, and other choices of sparse random
matrices (see Lemma~\ref{l:structural-less}). Moreover, we
develop an improved local convergence analysis of the Newton Sketch,
which allows us to recover the precise convergence rate and derive the
optimal step size. In
Appendix~\ref{a:distributed}, we also discuss a 
distributed variant of Newton-LESS, which takes advantage of the
near-unbiasedness properties of LESS embeddings, extending the results
of \cite{less-embeddings}.

The performance of Newton-LESS can be further improved for regularized
minimization tasks. Namely, suppose that function $f$ can be
decomposed as follows: $f(\x) = f_0(\x)+g(\x)$, where $g(\x)$ has a
Hessian that is easy to evaluate (e.g., $l_2$-regularization,
$g(\x)=\frac\lambda2\|\x\|^2$). In this case, a modified variant of the
Newton Sketch has been considered, where only the $f_0$ component
is sketched: 
\begin{align}
  \xbt_{t+1} = \xbt_t - \mu_t\big(\Af{f_0}{\xbt_t}^\top\S_t^\top\S_t
  \Af{f_0}{\xbt_t} + \nabla^2g(\xbt_t)\big)^{-1}\nabla
  f(\xbt_t),\label{eq:regularized-sketch} 
\end{align}
where, again, we let $\Af{f_0}{\x}$ be an $n\times d$ matrix
that encodes the second-order information in $f_0$ at $\x$. For example, in
the case of regularized least squares,
$f(\x)=\frac12\|\A\x-\b\|^2+\frac\lambda2\|\x\|^2$, we have
$\Af{f_0}{\x}=\A$ and $\nabla^2g(\x)=\lambda\I$ for all
$\x$. We show that the convergence
rate of both Newton-LESS and the Gaussian Newton Sketch can be
improved in the presence of regularization, by replacing the dimension
$d$ with an \emph{effective} dimension $\deff$.  This can be
significantly smaller than $d$ when the Hessian of $f_0$ at the optimum
exhibits rapid spectral decay or is approximately low-rank:
\begin{align*}
\deff = \tr\big(\nabla^2 f_0(\x^*)
  \,\nabla^2 f(\x^*)^{-1}\big) \leq d.
\end{align*}
\begin{theorem}\label{t:regularized-simple}
  Assume that $f_0$ and $f$ are (a) self-concordant, or
  (b) have a Lipschitz continuous Hessian. Also, let $\nabla^2
  f_0(\x^*)$ be positive definite and let $\nabla^2 g(\x^*)$ be positive semidefinite, with
  $\H=\nabla^2f(\x^*)$. There is a neighborhood $U$  
  containing $\x^*$ such that if $\xbt_0\in U$, then Regularized
  Newton-LESS \eqref{eq:regularized-sketch},
with sketch size $m\geq C\deff\log(\deff T/\delta)$ and step size
$\mu_t=1-\frac \deff m$, satisfies:
\vspace{-2mm}
  \begin{align*}
\bigg(\E_\delta\,\frac{\|\xbt_T-\x^*\|_{\H}^2}{\|\xbt_0-\x^*\|_{\H}^2}\bigg)^{1/T}
\leq\ \frac \deff m\cdot
    (1+\epsilon)\qquad\text{for}\quad\epsilon=O\Big(\frac1{\sqrt\deff}\Big).
  \end{align*}
\end{theorem}
\begin{remark}
The same guarantee holds for the Gaussian Newton Sketch. Unlike in
Theorem \ref{t:main-simple}, here we can only obtain an upper-bound 
on the local convergence rate, because the exact rate may depend on the
starting point $\xbt_0$ (see Section \ref{s:convergence}). For regularized least squares,
$f(\x)=\frac12\|\A\x-\b\|^2+\frac\lambda2\|\x\|^2$, the above convergence
guarantee holds globally, i.e., $U=\R^d$. Note that $\deff$ can be
efficiently estimated using sketching-based trace estimators
\cite{avron2016sharper,cohen2015dimensionality}.
\end{remark}

\vspace{-1mm}
Finally, our numerical results show that Newton-LESS can be
implemented very efficiently on modern hardware platforms, improving
on the optimization cost over not only dense Gaussian embeddings, but
also state-of-the-art sketching methods such as the Subsampled Randomized
Hadamard Transform, as well as other first-order and second-order
methods. Moreover, we demonstrate that our theoretical predictions for 
the optimal sparsity level and convergence rate are extremely accurate in
practice.

\subsection{Related work}
\label{s:related-work}
LEverage Score Sparsified (LESS) embeddings were proposed by \cite{less-embeddings} as a
way of addressing the phenomenon of inversion bias, which arises in
distributed second-order methods \cite{distributed-newton,determinantal-averaging,debiasing-second-order,gupta2021localnewton}. Their 
results only establish the \emph{near-unbiasedness} of Newton-LESS iterates
(i.e., that $\E[\xbt_{t+1}]\approx \x_{t+1}$), but they did not provide
any improved guarantees on the convergence rate. Also, their notion of LESS
embeddings is much narrower than ours, and so it does not capture
Regularized Newton-LESS or uniformly sparsified sketches (LESS-uniform).

Convergence analysis of the Newton Sketch
\cite{pilanci2017newton,lacotte2019faster,lacotte2020limiting} and
other randomized second-order methods
\cite{byrd2011use,byrd2012sample,erdogdu2015convergence,subsampled-newton-math-prog}
has been extensively studied in the machine learning
community, often using techniques from Randomized Numerical Linear
Algebra (RandNLA)~\cite{DM16_CACM,DM21_NoticesAMS}. Some of  
the popular RandNLA methods include
the Subsampled Randomized Hadamard Transform (SRHT, \cite{ailon2009fast}) and several
variants of sparse sketches, such as the CountSketch \cite{cw-sparse,mm-sparse} and OSNAP
\cite{nelson2013osnap,cohen2016nearly}. Also, row 
sampling based on Leverage Scores
\cite{drineas2006sampling,ridge-leverage-scores} and Determinantal
Point Processes \cite{dpp-intermediate, dpp-sublinear,
  DM21_NoticesAMS} has been used for sketching. 
Note that CountSketch and OSNAP sparse
sketches differ from LESS embeddings in several ways, and in
particular, they use a fixed number of non-zeros per column of the
sketching matrix (as opposed to per row), so unlike LESS, their rows
are not independent. While all of the mentioned 
methods, when used in conjunction with the Newton Sketch, exhibit
similar per-iteration complexity as LESS embeddings (see Section \ref{s:preliminaries}), 
their existing convergence analysis is fundamentally limited: The
best known rate is $(C\log d\cdot\frac dm)^T$, which is worse than our
result of $(\frac dm)^T$, by a factor of $C\log d$, where $C>1$ is a
non-negligible constant that arises in the measure
concentration analysis.

In the specific context of least squares regression where
$f(\x)=\frac{1}{2} \|\A \x - \b\|^2$, the Hessian $\A^\top \A$
remains constant, and an alternative strategy is to keep the random
sketch $\S\A$ fixed at every iteration. Many efficient randomized
iterative solvers are based on this precondition-and-solve approach
\cite{rokhlin2008fast, avron2010blendenpik,meng2014lsrn}: form the
sketch $\S \A$, compute an easy-to-invert square-root matrix $\tilde
\H^{\frac{1}{2}}$ of $\A^\top \S^\top \S \A$ and apply an iterative
least squares solver to the preconditioned objective
$\min_{\mathbf{z}} \frac{1}{2} \|\A  \tilde \H^{-\frac{1}{2}}
\mathbf{z} - \b\|^2$, e.g., Chebyshev iterations or the
preconditioned conjugate gradient. In contrast to the Newton Sketch,
these methods do not naturally extend to more generic convex
objectives for which the Hessian matrix changes at every
iteration. Also, similarly as the Newton Sketch, their convergence
guarantees are limited to $(C\log d\cdot\frac dm)^T$ when used in conjunction
with fast sketching methods such as SRHT, OSNAP, or leverage score sampling.

\section{Preliminaries}
\label{s:preliminaries}
\paragraph{Notation.} We let $\|\v\|_{\M}=\sqrt{\v^\top\M\v}$. We
define $a\approx_\epsilon b$ to mean $|a-b|\leq\epsilon b$, whereas
$a=b\pm\epsilon$ means that $|a-b|\leq\epsilon$, and $C$ denotes a
large absolute constant. We use $\E_{\Ec}$ to denote expectation
conditioned on $\Ec$, and for a $\delta\in(0,1)$, we use
$\E_{\delta}$ as a short-hand for: ``There is an event $\Ec$ with
probability at least $1-\delta$ s.t.~$\E_{\Ec}$ ...''. Let
pd and psd mean positive definite and positive
semidefinite. Random variable $X$ is sub-Gaussian if
$\Pr\{|X|\geq t\}\leq \exp(-ct^2)$ for all $t\geq0$ and some $c=\Omega(1)$.

We next introduce some concepts related to LEverage Score Sparsified
(LESS) embeddings. We start with the notion of statistical leverage
scores~\cite{fast-leverage-scores}, which are importance weights assigned to the rows of a matrix
$\A\in\R^{n\times d}$. The definition below for leverage scores is somewhat more general than standard
definitions, because it allows for a \emph{regularized} leverage
score, where the regularization depends on a $d\times d$ matrix
$\C$. When $\C$ is a scaled identity matrix, this matches the
definition of ridge leverage scores \cite{ridge-leverage-scores}.
\begin{definition}[Leverage scores]\label{d:lev}
  For given matrices $\A\in\R^{n\times d}$ and psd $\C\in\R^{d\times d}$,
 we define the $i$th \underline{leverage score} $l_i(\A,\C)$ as the squared norm
  of the $i$th row of $\U=\A\H^{-\frac12}$, where $\H=\A^\top\A+\C$ is
  assumed to be invertible. The \underline{effective
    dimension} of $\A$ (given $\C$) is defined as
  $\deff=\sum_il_i(\A,\C)=\tr(\U^\top\U)$, whereas the
  \underline{coherence} of $\A$ (given $\C$) is $\tau=\frac n\deff
  \max_i l_i(\A,\C)\in[1,\frac n\deff]$.
\end{definition}
Next, we define a class of sparsified sub-Gaussian sketching matrices
which will be used in our results. This captures LESS embeddings, as
well as other sketching matrices that are supported by the
analysis. To that end, we define what we call a \emph{sparsifier},
which is an $n$-dimensional random vector $\xib$ that specifies the sparsification
pattern for one row of the $m\times n$ sketching matrix $\S$.
\begin{definition}[LESS embeddings]\label{d:less}
Let $t_1,...,t_s$ be sampled i.i.d. from a distribution $p=(p_1,...,p_n)$. Then, the random vector
  $\xib^\top=\big(\sqrt{\!\frac{b_1}{sp_1}},...,\sqrt{\!\frac{b_n}{sp_n}}\big)$,
  where $b_i=\sum_{j=1}^s1_{[t_j=i]}$, is a $(p,s)$-sparsifier. 
  A $(p,s)$-sparsified sub-Gaussian sketch is a random matrix $\S$ that
  consists of row vectors $c\cdot(\x\circ\xib)^\top$, where $\circ$ denotes
  the entry-wise product, $\x$ has i.i.d. mean zero, unit variance and
  sub-Gaussian entries, and $c$ is some constant. For given matrices
  $\A\in\R^{n\times d}$ and psd $\C\in\R^{d\times d}$, we focus on
  two variants of Leverage Score Sparsified embeddings:
  \begin{enumerate}
    \item \underline{LESS}. We assume that $p_i\approx_{1/2}l_i(\A,\C)/\deff$,
      and let $s \approx_{1/2} \deff$. For $\C=\zero$ and
      $s=d$, we recover the LESS embeddings proposed by
      \cite{less-embeddings}.
    \item \underline{LESS-uniform}. We simply let $p_i=1/n$ (denoted
      as $p=\mathrm{unif}$). This
      avoids the preprocessing needed for approximating the
      $l_i(\A,\C)$, but we may need larger $s$ to recover the theory.
  \end{enumerate}
\end{definition}
\paragraph{Computational cost.} To implement LESS, we must first approximate (i.e., there is no need to compute exactly~\cite{fast-leverage-scores}) the leverage
scores of~$\A$. This can be done in time $O(\nnz(\A)\log n +d^3\log
d)$ by using standard RandNLA techniques
\cite{fast-leverage-scores,cw-sparse}, where $\nnz(\A)$ is the number of
non-zero entries in $\A$ and it is bounded by~$nd$. Since the prescribed sparsity for LESS
satisfies $s=O(d)$, the sketching cost is at most $O(md^2)$. Thus, the
total cost of constructing the sketched Hessian with LESS is
$O(\nnz(\A)\log n + md^2)$, which up to logarithmic factors matches other sparse
sketching methods such as leverage score sampling (when implemented with approximate leverage scores~\cite{fast-leverage-scores}), CountSketch, and
OSNAP. In comparison, using the SRHT leads to $O(nd\log m + md^2)$
complexity, since this method does not take advantage of data sparsity.
Note that, in practice, the computational trade-offs between sketching
methods are quite different, and significantly hardware-dependent (see Section \ref{s:experiments}).
In particular, the cost of approximating the leverage
scores in LESS embeddings can be entirely avoided by using LESS-uniform. Here, the total cost
of sketching is $O(mds)$, but the sparsity of the sketch that is
needed for the theory depends on $\A$ and $\C$. Yet, in
Section \ref{s:experiments}, we show empirically that this approach
works well even for $s=d$.

\section{Equivalence between LESS and Gaussian Embeddings}
\label{s:equivalence}
In this section, we derive the basic quantities that determine the
convergence properties of the Newton Sketch, namely, the first and
second moments of the normalized sketched Hessian inverse. Our key technical
contribution is a new analysis of the second moment for a wide class
of sketching matrices that includes LESS embeddings and sub-Gaussian sketches.

Consider the Newton Sketch update as in \eqref{eq:regularized-sketch}, and
let $\Deltat_{t}=\xbt_{t}-\x^*$. Denoting $\H_t=\nabla^2f(\xbt_t)$,
$\g_t=\nabla f(\xbt)$, and
using $\p_t=-\mu_t\H_t^{-1}\g_t$ to denote the exact Newton direction with
step size $\mu_t$, a simple calculation shows that:
\begin{align}
  \|\Deltat_{t+1}\|_{\H_t}^2 - \|\Deltat_t\|_{\H_t}^2 =
  2\Deltat_t^\top\H_t^{\frac12}\Qbt\H_t^{\frac12}\p_t +
  \p_t^\top\H_t^{\frac12}\Qbt^2\H_t^{\frac12}\p_t, \label{eq:newton-exact}
\end{align}
where $\Qbt =
\H_t^{\frac12}(\Af{f_0}{\xbt_t}^\top\S_t^\top\S_t\Af{f_0}{\xbt_t}+\nabla^2
g(\xbt_t))^{-1}\H_t^{\frac12}$. 
From this, we have that 
the expected decrease in the optimization error is determined by the
first two moments of the matrix $\Qbt$, i.e., $\E[\Qbt]$ and
$\E[\Qbt^2]$. In the unregularized case, i.e., $g(\x)=0$, these moments can be derived exactly for the Gaussian
embedding. For instance if we let $\S_t$ be an $m\times n$ matrix with i.i.d. standard normal
entries scaled by $\frac1{\sqrt{m-d-1}}$, then  we obtain that:
\begin{align*}
  \E[\Qbt] = \I, \qquad \E[\Qbt^2]=\tfrac{(m-1)(m-d-1)}{(m-d)(m-d-3)}\cdot\I
  \ \approx_{\epsilon}\ \tfrac m{m-d}\cdot\I, 
\end{align*}
for $\epsilon=O(1/d)$. This choice of scaling for the Gaussian Newton Sketch ensures that
each iterate $\xbt_{t+1}$ is an unbiased estimate of the corresponding
exact Newton update with the same step size, i.e., that
$\E[\xbt_{t+1}] = \xbt_t+\p_t$. For most other sketching techniques,
neither of the two moments is analytically tractable because of the
bias coming from matrix inversion. Moreover, if we
allow for regularization, e.g., $g(\x)=\frac\lambda2\|\x\|^2$, then
even the Gaussian embedding does not enjoy tractable formulas for the
moments of $\Qbt$. However, using ideas from asymptotic random matrix
theory, \cite{less-embeddings} showed that in the unregularized case,
the exact Gaussian formula for the first moment holds approximately
for sub-Gaussian sketches and LESS embeddings:
$\E_\delta[\Qbt]\approx_\epsilon\I$. This implies near-unbiasedness of
the unregularized Newton-LESS iterates relative to the exact Newton step, but it
is not sufficient to ensure any convergence guarantees.

In this work, we develop a general characterization of the first and
second moments of $\Qbt$ for a wide class of sketching matrices, both
in the unregularized and in regularized settings. For the sake of
generality, we will simplify the notation here, and analyze the first
and second moment of $\Q=\H^{\frac12}(\A^\top\S^\top\S\A+\C)^{-1}\H^{\frac12}$ for some
matrices $\A\in\R^{n\times d}$ and $\C\in\R^{d\times d}$ such that
$\A^\top\A+\C=\H$. In the context of Newton Sketch \eqref{eq:regularized-sketch}, these quantities
correspond to $\A=\Af{f_0}{\xbt_t}$ and
$\C=\nabla^2 g(\xbt_t)$. Also, as a shorthand, we will define the
normalized version of matrix $\A$ as $\U=\A\H^{-\frac12}$. The following
are the two structural conditions that need to be satisfied by a
sketching matrix to enable our analysis. 

The first condition is
standard in the sketching literature. Essentially, it implies that the
sketching matrix $\S$ produces a useful approximation of the Hessian
with high probability (although this guarantee is still far too coarse
by itself to obtain our results).
\begin{condition}[Property of random matrix $\S$]\label{cond1}
  Given $\U\in\R^{n\times d}$, the $m\times n$ random matrix $\S$
  satisfies $\|\U^\top\S^\top\S\U-\U^\top\U\|\leq \eta$ with probability $1-\delta$.
\end{condition}

\noindent
This property is known as the \emph{subspace embedding property}.
Subspace embeddings were first used by \cite{drineas2006sampling}, where they were used in a data-aware context to obtain relative-error approximations for $\ell_2$ regression and low-rank matrix approximation \cite{cur-decomposition}.
Subsequently, data-oblivious subspace embeddings were used by \cite{sarlos-sketching} and popularized by \cite{woodruff2014sketching}.
Both data-aware and data-oblivious subspace embeddings can be used to derive bounds for the accuracy of various algorithms \cite{DM16_CACM,RandNLA_PCMIchapter_chapter}.

For our analysis, it is important to
assume that $\S$ has i.i.d.\ row vectors $c\s_i^\top$, where $c$ is
an appropriate scaling constant. The second condition is defined as a
property of those row vectors, which makes them sufficiently similar
to Gaussian vectors. This is a relaxation of the Restricted
Bai-Silverstein condition, proposed by \cite{less-embeddings}, which 
leads to significant improvements in the sparsity guarantee
for LESS embeddings when the Newton Sketch is regularized.
\begin{condition}[Property of random vector $\s$]\label{cond2}
 Given $\U\in\R^{n\times d}$,
  the $n$-dimensional random vector $\s$ satisfies
  $\Var\!\big[\s^\top\U\B\U^\top\s\big]\leq \alpha\cdot\tr(\U\B^2\U^\top)$
  for all p.s.d.\ matrices $\B$ and some $\alpha=O(1)$.  
\end{condition}
Given these two conditions, we are ready to derive precise non-asymptotic 
analytic expressions for the first two moments of the regularized sketched 
inverse matrix, which is the main technical contribution of this work
(proof in Appendix \ref{a:structural}).
\begin{theorem}\label{t:structural}
  Fix $\A$ and assume that $\C$ is psd. Define $\H=\A^\top\A+\C$ and $\U=\A\H^{-\frac12}$. Let $\S$
  consist of $m$ i.i.d.~rows distributed as
  $\frac1{\sqrt{m-\deff}}\s^\top$, where $\E[\s\s^\top]=\I_n$ and
  $\deff =\tr(\U^\top\U)$. Also, let $\tdeff=\tr((\U^\top\U)^2)$.
  Suppose that the matrix consisting of the
  first $m/3$ rows of $\S$ scaled by $\sqrt 3$ satisfies
  Condition \ref{cond1}  w.r.t.~$\U$, for $\eta\leq1/2$ and probability
  $1-\delta/3$, where $\delta\leq 1/m^3$. Suppose also that
  $\s$ satisfies Condition~\ref{cond2} w.r.t.~$\U$. If $m\geq
  O(\deff)$, then, conditioned on 
  event $\Ec$ that holds with probability $1-\delta$, matrix
  $\Q=\H^{\frac12}(\A^\top\S^\top\S\A+\C)^{-1}\H^{\frac12}$ satisfies 
  $\|\Q-\I\| \leq O(\eta)$ and:
  \begin{align*}
    \big\|\E_{\Ec}[\Q] - \I\big\|\leq O\big(\tfrac{\sqrt{\deff}}m\big),\qquad
    \big\|\E_{\Ec}[\Q^2]
    - \big(\I + \tfrac{\deff}{m-\tdeff}\U^\top\U\big)\big\| &\leq
                                                                   O\big(\tfrac{\sqrt{\deff}}m\big).
  \end{align*}
\end{theorem}
Theorem \ref{t:structural} shows that, for a wide class of sketching
matrices, we can approximately write $\E[\Q]\approx \I$ and
$\E[\Q^2]\approx \I+\frac{\deff}{m-\tdeff}\U^\top\U$, with the error
term scaling as $O(\frac{\sqrt{\deff}}{m})$. In the
case of the first moment, this is a relatively straightforward
generalization of the unregularized formula for the Gaussian case.
However, for the second moment this expression is considerably more 
complicated, including not one but two notions of effective dimension,
$\tdeff\neq\deff$. To put this in context, in the unregularized case,
i.e., $\C=\zero$, we have $\deff=\tdeff=d$ and $\U^\top\U=\I$, so we
get
$\I+\frac{\deff}{m-\tdeff}\U^\top\U=\frac{m}{m-d}\I$,
which matches the second moment for the Gaussian sketch (up to lower
order terms that get absorbed into the~error).

In the case of the first moment, the proof of Theorem
\ref{t:structural} follows along the same lines as in
\cite{less-embeddings}, using a decomposition of $\E[\Q]$ that is based on
the Sherman-Morrison rank-one update of the inverse. This approach was
originally inspired by the analysis of Stieltjes transforms that are used to establish the
limiting spectral distribution in asymptotic random
matrix theory (e.g., \cite{bai2010spectral, cuillet-book}), and
applied to sketching by \cite{precise-expressions,less-embeddings}. Our key contribution
lies in deriving the bound for the second moment, which requires
a substantially more elaborate decomposition of $\E[\Q^2]$.

In the following lemma, we establish that the assumptions of Theorem
\ref{t:structural} are satisfied not only by Gaussian,
but also sub-Gaussian and LESS embedding matrices (proof in Appendix \ref{a:structural-less}).
\begin{lemma}\label{l:structural-less}
Fix $\A$ and assume that $\C$ is psd. Let $\S$ be a sketching matrix with $m$ i.i.d. rows distributed as
$\frac1{\sqrt{m-\deff}}\s^\top$.
Then, $\S$ satisfies Conditions \ref{cond1} and \ref{cond2}
as long as one of the following holds:
\begin{enumerate}
  \item \underline{Sub-Gaussian}: $\s$ is an i.i.d. sub-Gaussian random vector and
    $m\geq C(\deff+\log(1/\delta))/\eta^2$;
    \item \underline{LESS}: $\s$ is a $(p,\deff)$-sparsified
      i.i.d.~sub-Gaussian random vector with $p_i\approx_{1/2}l_i(\A,\C)/\deff$, and
      $m\geq C\deff\log(\deff/\delta)/\eta^2$;
    \item \underline{LESS-uniform}: 
      $\s$ is a $(\mathrm{unif},\tau\deff)$-sparsified
      i.i.d.~sub-Gaussian random vector, where $\tau$ is the coherence
      of $\A$, i.e., $\frac n\deff\max_il_i(\A,\C)$, and $m\geq C\deff\log(\deff/\delta)/\eta^2$. 
    \end{enumerate}
\end{lemma}

\section{Convergence Analysis for Newton-LESS}
\label{s:convergence}

In this section, we demonstrate how our main technical results can be
used to provide improved convergence guarantees for Newton-LESS (and,
more broadly, any sketching methods that satisfy the conditions of
Theorem \ref{t:structural}). Here, we will focus on the more general
regularized setting \eqref{eq:regularized-sketch}, where we can only show an upper bound on the
convergence rate (Theorem \ref{t:regularized-simple}). The
unregularized result (Theorem \ref{t:main-simple}) with matching
upper/lower bounds follows similarly.

To start, we introduce the standard assumptions on the function $f$,
which are needed to ensure strong local convergence guarantees for the
classical Newton's method \cite{boyd2004convex}.
\begin{assumption}\label{a:lipschitz}
  Function $f:\R^d\rightarrow \R$ has a Lipschitz continuous Hessian
  with constant $L$, i.e., $\|\nabla^2 f(\x)-\nabla^2 f(\x')\|\leq
L\,\|\x-\x'\|$ for all $\x,\x'\in\R^d$.
\end{assumption}
\begin{assumption}\label{a:self-concordant}
  Function $f:\R^d\rightarrow \R$ is self-concordant, i.e., for all
  $\x,\x'\in\R^d$, the function $\phi(t)=f(\x+t\x')$
  satisfies: $|\phi'''(t)|\leq 2(\phi''(t))^{3/2}$.
\end{assumption}
Only one of those two assumptions needs to be satisfied for our
analysis to go through, and the choice of the assumption only affects the size of
the neighborhood around the optimum $\x^*$ for which our local
convergence guarantee is satisfied. To clarify this, below we give an
expanded version of Theorem~\ref{t:regularized-simple} (proof in
Appendix \ref{a:convergence}).
\begin{theorem}[Expanded Theorem \ref{t:regularized-simple}]\label{t:regularized-full}
Let $\H_0=\nabla^2 f_0(\x^*)$ be pd and $\C=\nabla^2 g(\x^*)$ be psd. Define
$\deff=\tr(\H_0\H^{-1})$ and $\tdeff=\tr((\H_0\H^{-1})^2)$ for $\H=\H_0+\C$. Assume one of the
following:
\begin{enumerate}
  \item $f_0$ and $f$ satisfy Assumption \ref{a:lipschitz}, and $U = \{\x :
    \|\x-\x^*\|_{\H}<\frac{\sqrt\deff}{m}(\lambda_{\min})^{3/2}/L\}$,
    where $\lambda_{\min}$ is the smallest eigenvalue of $\H_0$;
  \item $f_0$ and $f$ satisfy Assumption \ref{a:self-concordant}, and
    $U=\{\x:\|\x-\x^*\|_{\H}<\frac{\sqrt\deff}{m}\}$.
  \end{enumerate}
Then, Newton Sketch \eqref{eq:regularized-sketch} starting from
$\xbt_0\in U$, using any $\S_t$ from
Lemma~\ref{l:structural-less} (i.e., sub-Gaussian, LESS, or
LESS-uniform) with $\delta$ replaced by $\delta/T$, and step size 
$\mu_t=1-\frac\deff {m+\deff-\tdeff}$, satisfies:
\vspace{-1mm}
  \begin{align*}
\bigg(\E_\delta\,\frac{\|\xbt_T-\x^*\|_{\H}^2}{\|\xbt_0-\x^*\|_{\H}^2}\bigg)^{1/T}
\leq\ \frac{\deff\cdot (1+\epsilon)}{m+\deff-\tdeff} \leq\ \frac \deff
    m\cdot (1+\epsilon),\qquad\text{for}\quad\epsilon=O\Big(\frac1{\sqrt\deff}\Big).
  \end{align*}
\end{theorem}
Note that, compared to Theorem \ref{t:regularized-simple}, here we present a slightly sharper bound which uses both
types of effective dimension, $\tdeff\leq \deff$, that are present in
Theorem \ref{t:structural}. The statement from
Theorem~\ref{t:regularized-simple} is recovered by replacing $\tdeff$
with $\deff$ in the step size and in the bound.  The key step in the proof
of the  result is the following lemma, 
which uses Theorem \ref{t:structural} to characterize the Newton Sketch iterate
$\xbt_{t+1}$ in terms of the corresponding Newton iterate
$\x_{t+1}$. Note that this result holds globally for arbitrary
$\xbt_t$ and without the smoothness assumptions on $f$. Recall that we
let $\Deltat_t=\xbt_t-\x^*$ denote the error residual at step $t$.
\begin{lemma}\label{l:regularized-decomposition}
Fix $\H_t=\nabla^2 f(\xbt_t)$ and let $\xbt_{t+1}$ be the Newton
Sketch iterate with $\S_t$ as in Lemma~\ref{l:structural-less}. If the exact
Newton step $\x_{t+1}=\xbt_t-\mu_t\H_t^{-1}\g_t$ is a descent direction,
i.e., $\|\Delta_{t+1}\|_{\H_t}\leq\|\Deltat_t\|_{\H_t}$ where
$\Delta_{t+1}=\x_{t+1}-\x^*$, then
\begin{align*}
  \E_{\delta}\,\|\Deltat_{t+1}\|_{\H_t}^2 =
  \|\Delta_{t+1}\|_{\H_t}^2 +
  \tfrac{\deff(\xbt_t)}{m-\tdeff(\xbt_t)}\|\x_{t+1}-\xbt_t\|_{\nabla^2 f_0(\xbt_t)}^2 \pm
  O\big(\tfrac{\sqrt \deff}{m}\big)\|\Deltat_t\|_{\H_t}^2.
\end{align*}
\end{lemma}
Importantly, the second term on the right-hand side uses the norm
$\|\cdot\|_{\nabla^2 f_0(\xbt_t)}$, which is different than the norm
$\|\cdot\|_{\H_t}$ used for the remaining terms. As a result, in the
regularized setting, it is possible that the second term will be much
smaller than the last term (the approximation error). This prevents us
from obtaining a matching lower-bound for the convergence rate of
Regularized Newton-LESS. On the other hand, when $g(\x)=0$, then
$f_0=f$ and we obtain matching upper/lower~bounds.

The remainder of the proof of Theorem \ref{t:regularized-full} is
essentially a variant of the local convergence analysis of the
Newton's method. Here, note that typically, we would set the step size
to $\mu_t=1$ and we would expect superlinear (specifically, quadratic) convergence rate. However,
for Newton Sketch, convergence is a mixture of linear rate and
superlinear rate, where the linear part is due to the approximation
error in sketching the Hessian. Sufficiently close to the optimum, the
linear rate will dominate, and so this is what we focus on in our local
convergence analysis. The key novelty here is that, unlike prior work,
we strive to describe the linear rate precisely, down to lower order
terms. As a key step, we observe that the convergence of exact Newton with step
size $\mu_t<1$, letting $\Delta_{t+1}=\x_{t+1}-\x^*$, is given by:
\begin{align}
  \|\Delta_{t+1}\|_{\H_t}^2 =
  \underbrace{(1-\mu_t)^2\|\Deltat_t\|_{\H_t}^2}_{\text{linear rate}} +
  \underbrace{\mu_t(\Delta_{t+1}+(1-\mu_t)\Deltat_t)^\top(\H_t\Deltat_t-\g_t)}_{\text{superlinear
  rate}},\label{eq:newton-decomposition}
\end{align}
where recall  that $\Deltat_t=\xbt_t-\x^*$ corresponds to the previous
iterate, and $\g_t=\nabla f(\xbt_t)$. Here, $(1-\mu_t)^2$ represents
the linear convergence rate. The superlinear term vanishes  near the
optimum $\x^*$, because of the presence of $\H_t\Deltat_t-\g_t$, which
(under the smoothness assumptions on $f$) vanishes at the same rate as
$\|\Deltat_t\|_{\H_t}^2$. Interestingly, 
with this precise analysis, the superlinear term does not have a
quadratic rate, but rather a $3/2$ rate. Entering
\eqref{eq:newton-decomposition}  into the guarantee from Lemma
\ref{l:regularized-decomposition}, we obtain that the local rate of
Newton Sketch can be expressed as $(1-\mu_t)^2 +
\frac{\deff}{m-\tdeff}\mu_t^2$, and minimizing this expression over
$\mu_t$ we obtain the desired quantities from Theorem
\ref{t:regularized-full}. Finally, we note that while the norms and
effective dimensions in the above exact calculations are stated with
respect to the Hessian $\H_t$ at the current iterate $\xbt_t$, these can all be
approximated by the corresponding quantities computed using the
Hessian at the optimum $\x^*$ (as in Theorem \ref{t:regularized-full}),
relying on smoothness of~$f$.

\begin{figure}[!ht]
	\centering
	\begin{minipage}[t]{0.4\textwidth}
		\centering
		\includegraphics[width=\linewidth]{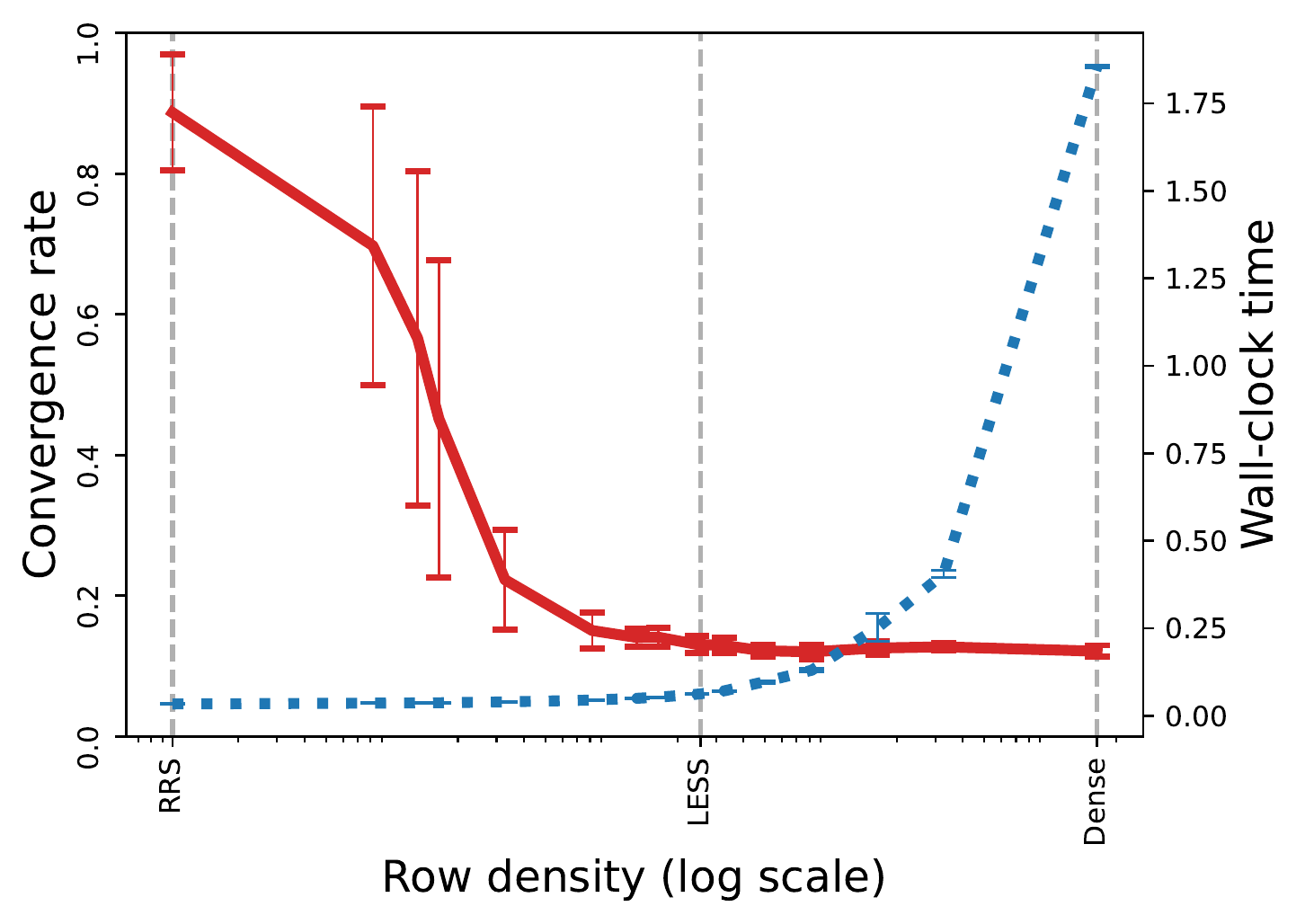}
                \vspace{-7mm}                
		\caption*{\footnotesize (a) High-coherence synthetic matrix}
	\end{minipage}
	\hspace{5mm}
	\begin{minipage}[t]{0.4\textwidth}
		\centering
		\includegraphics[width=\linewidth]{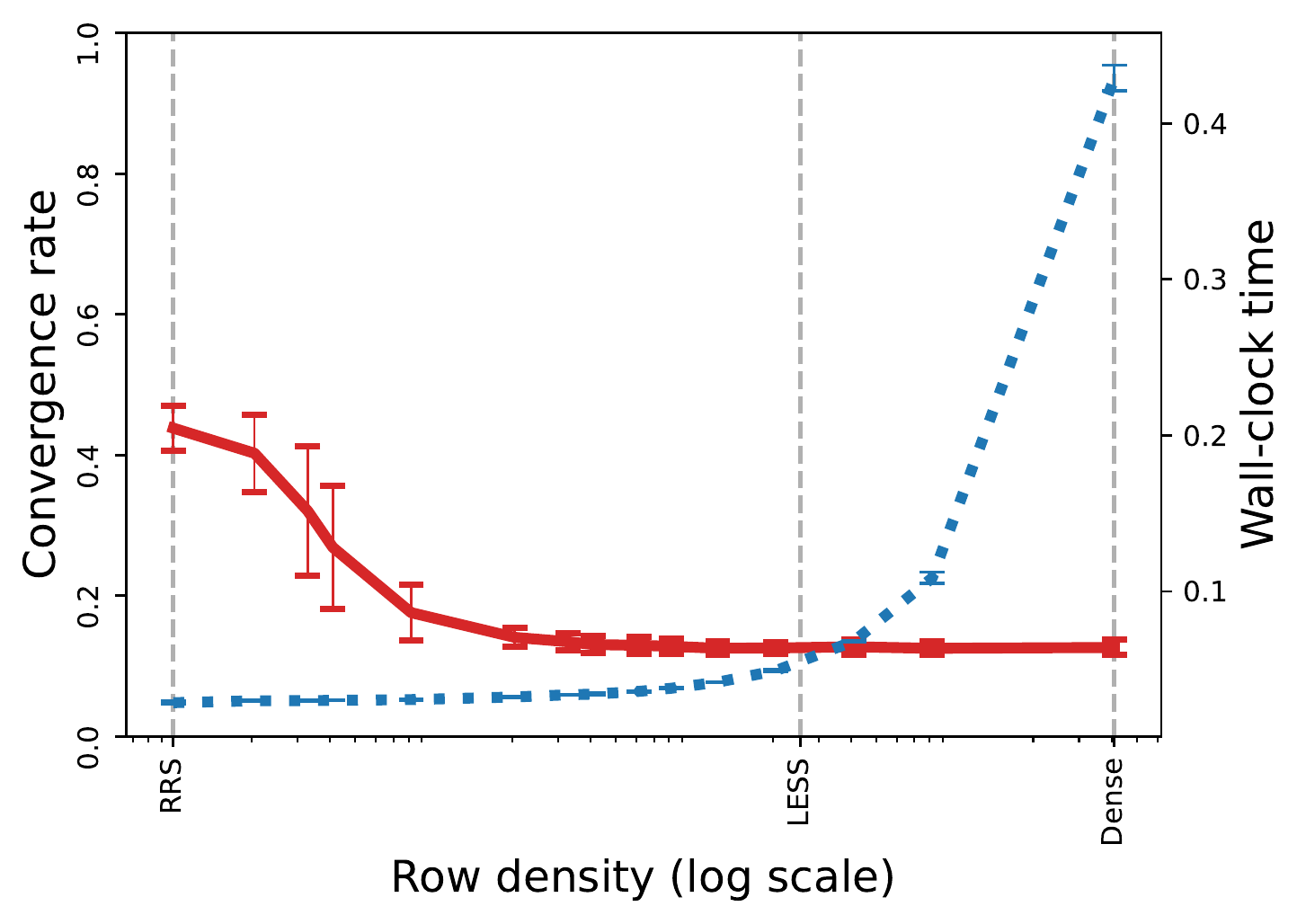}
                \vspace{-7mm}                
		\caption*{\footnotesize (b)   Musk dataset  }
	\end{minipage}
	\begin{minipage}[t]{0.4\textwidth}
          \centering
          \vspace{.3cm}
		\includegraphics[width=\linewidth]{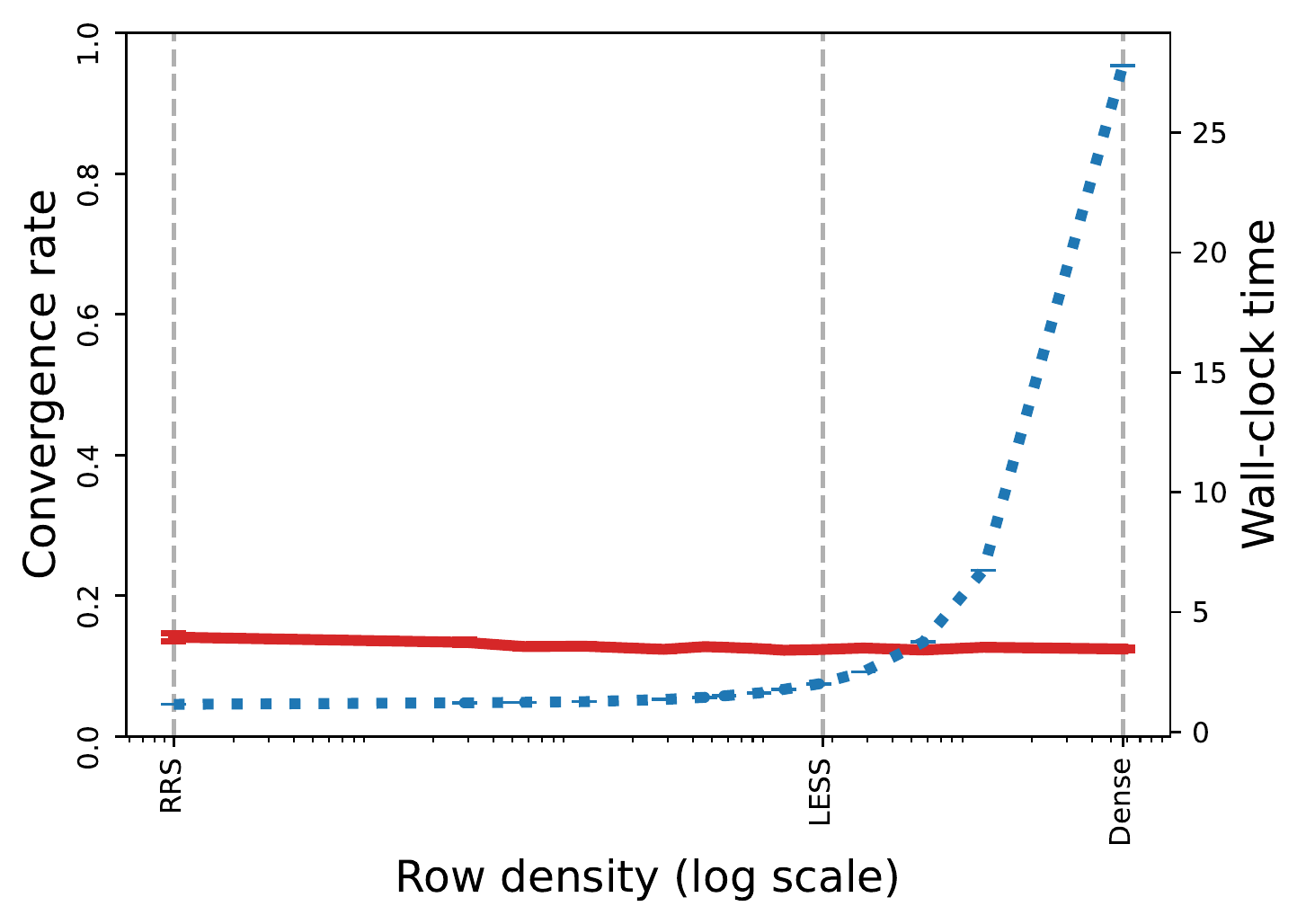}
                \vspace{-7mm}                
		\caption*{\footnotesize(c)  CIFAR-10 dataset }
	\end{minipage}
	\hspace{5mm}
	\begin{minipage}[t]{0.4\textwidth}
          \centering
          \vspace{.3cm}
          \includegraphics[width=\linewidth]{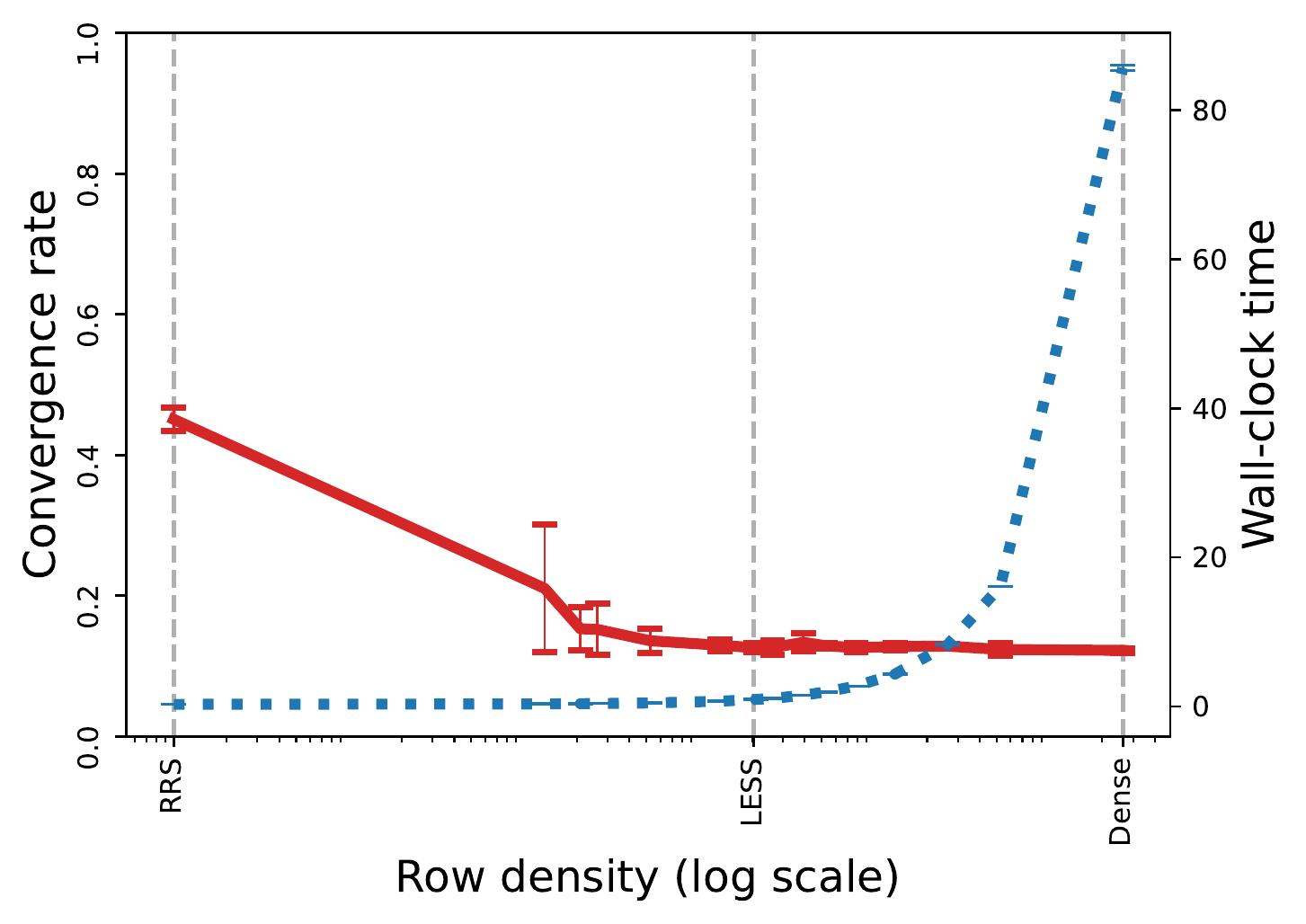}
          \vspace{-7mm}          
          \caption*{\footnotesize (d)   WESAD dataset}
	\end{minipage}
	\caption{LESS-uniform embeddings: convergence rate of the
          Newton Sketch for least squares regression and wall-clock
          time of forming $\S \A$ versus row density, with 
          sketch size $m=8d$. The results were averaged
          over $100$ trials and error bars show  twice the empirical standard deviation.}
	\label{fig:row-sparsity}
\end{figure}

\section{Numerical Experiments}
\label{s:experiments}

We have evaluated our theory on a range of different problems, and we have found that the more precise analysis that our theory provides describes well the convergence behavior for a range of optimization problems. 
In this section, we present numerical simulations illustrating this for regularized logistic regression and least squares regression, with different datasets ranging from medium to large scale: the CIFAR-10 dataset, the Musk dataset, and WESAD \cite{schmidt2018introducing}. 
Data preprocessing and implementation details, as well as additional
numerical results for least squares and regularized least squares, can
be found in Appendix \ref{a:additionalnumericalexperiments}.  

\begin{figure}
	\centering
\begin{minipage}[t]{0.33\textwidth}
	\centering
	\includegraphics[width=\linewidth]{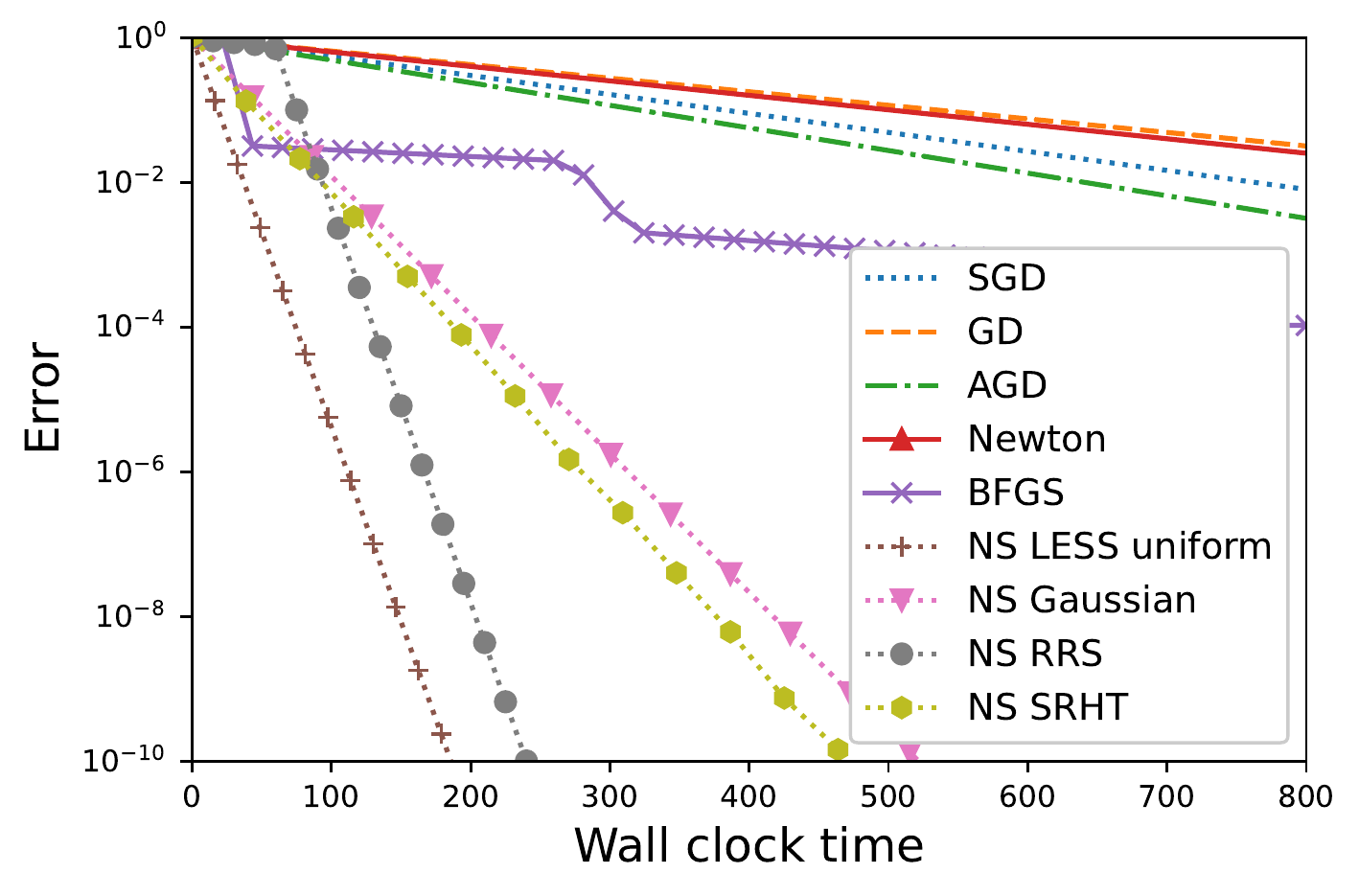} 
        \vspace{-7mm}        
	\caption*{\footnotesize{(a) WESAD dataset }}
\end{minipage}
\nolinebreak%
\begin{minipage}[t]{0.33\textwidth}
	\centering
	\includegraphics[width=\linewidth]{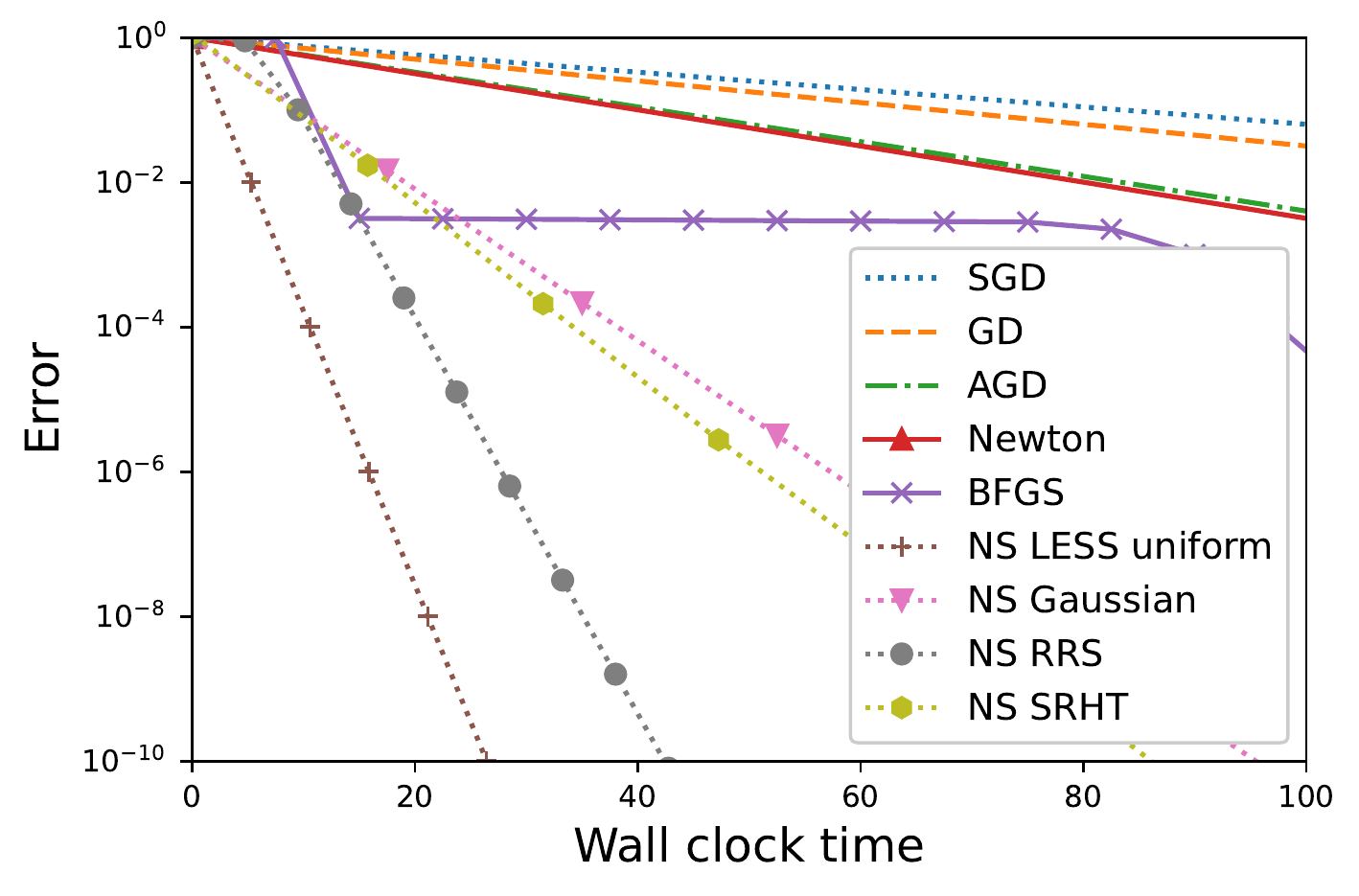} 
        \vspace{-7mm}
	\caption*{\footnotesize{(b) CIFAR-10 dataset }}
\end{minipage}
\nolinebreak%
\begin{minipage}[t]{0.33\textwidth}
	\centering
	\includegraphics[width=\linewidth]{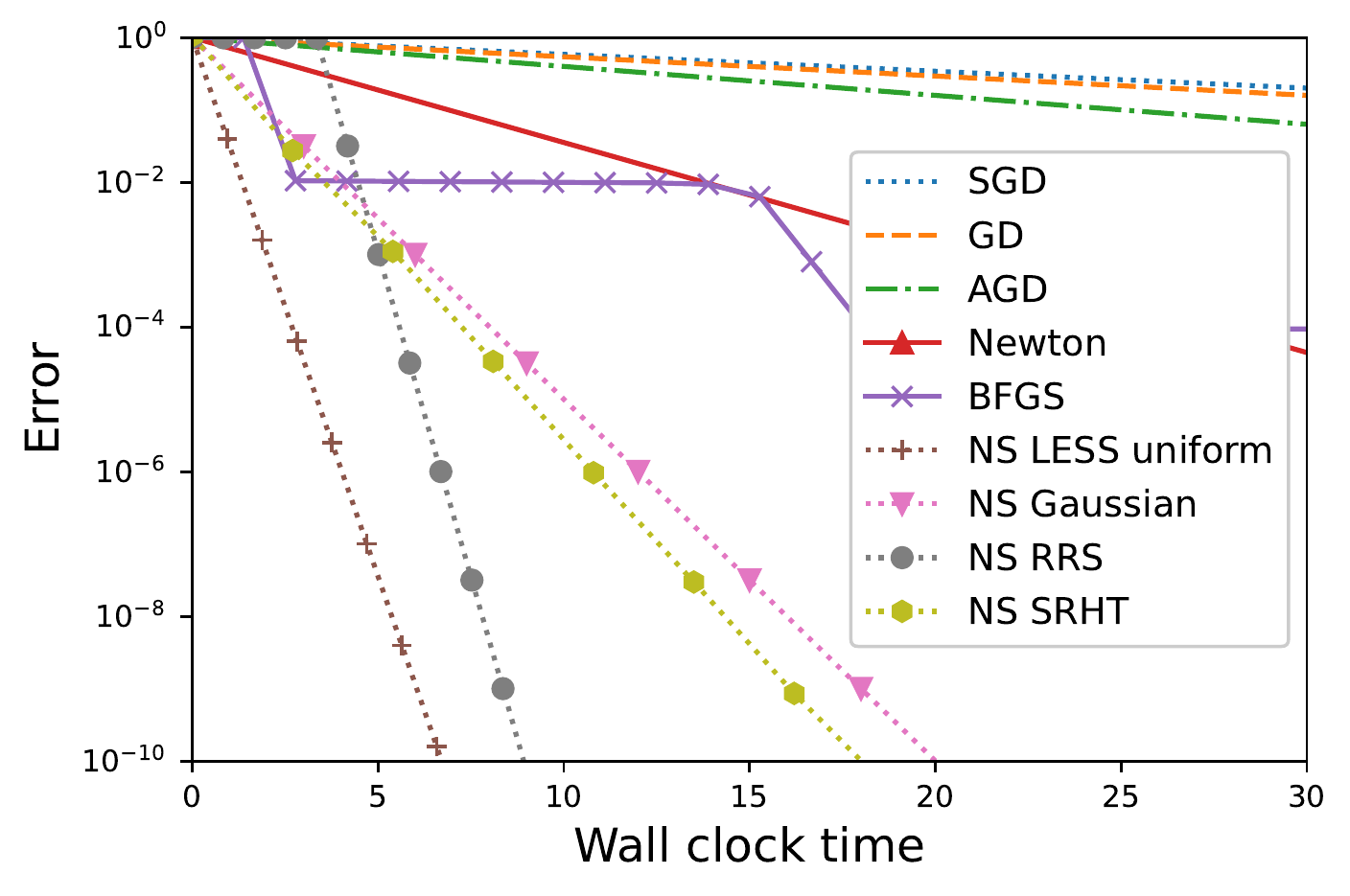} 
        \vspace{-7mm}
        \caption*{\footnotesize{(c)  Musk dataset }}
\end{minipage}
\begin{minipage}[t]{0.33\textwidth}
  \centering
	\includegraphics[width=\linewidth]{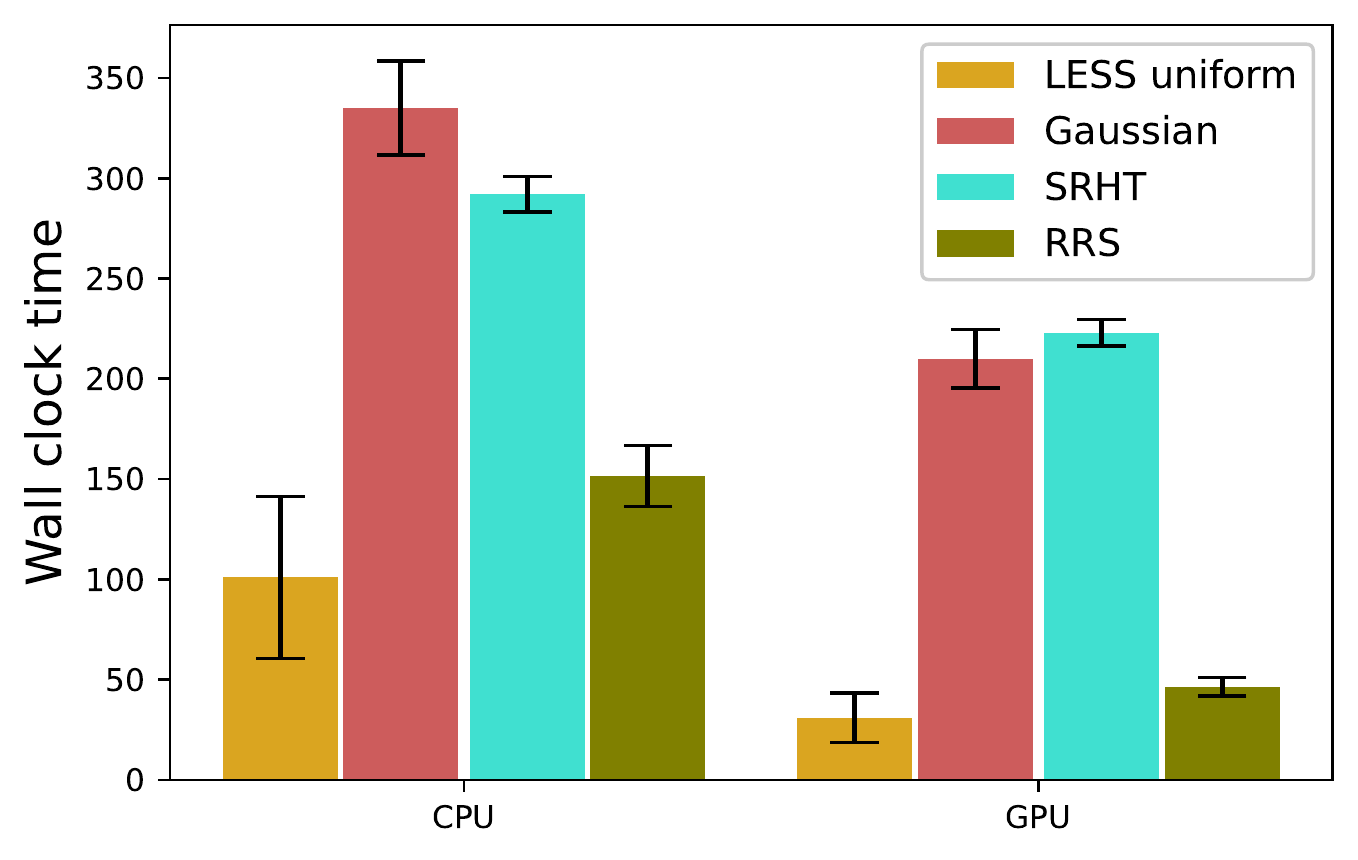}
      \end{minipage}
\nolinebreak%
\begin{minipage}[t]{0.33\textwidth}
  \centering
  \includegraphics[width=\linewidth]{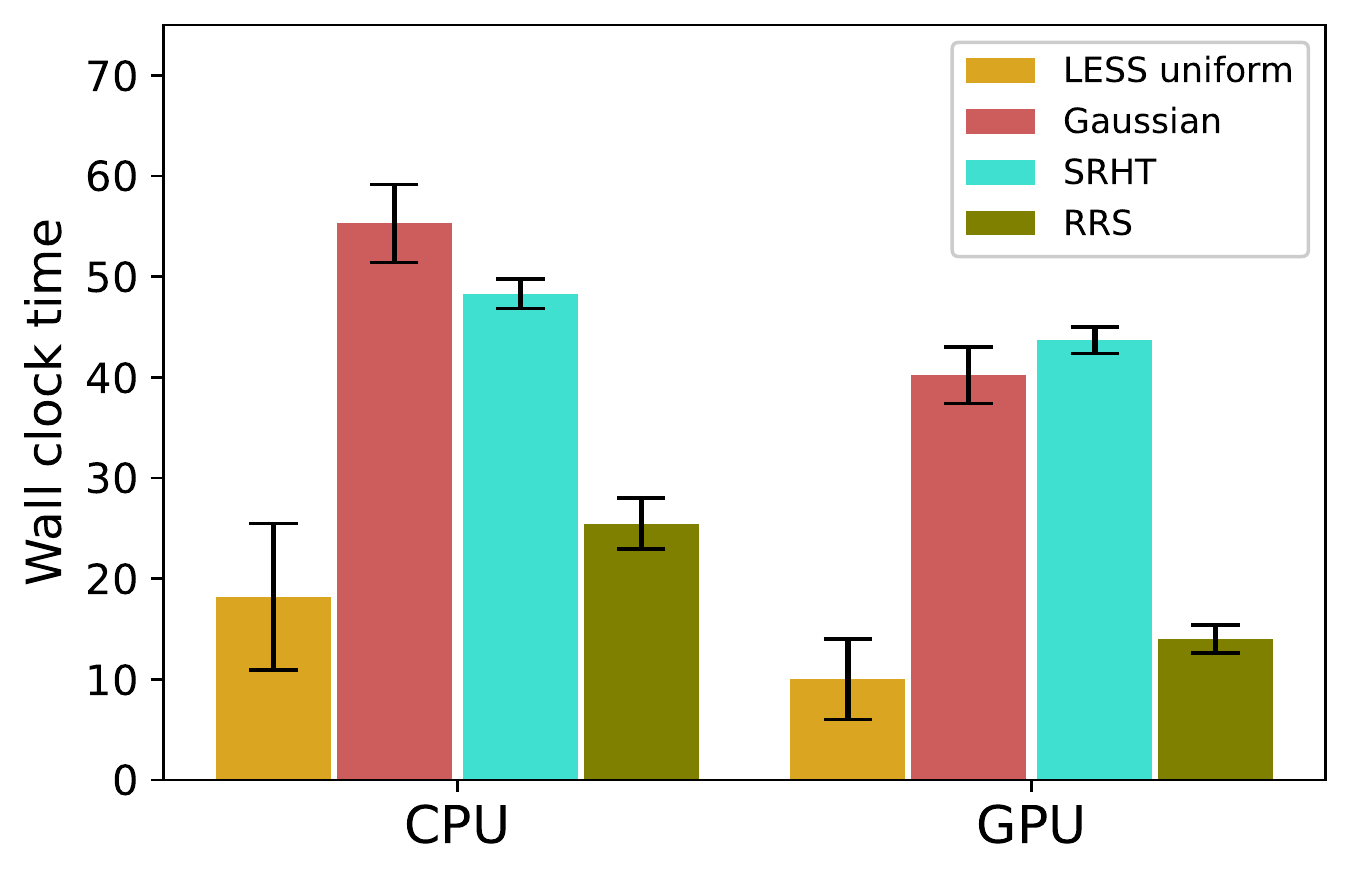}
\end{minipage}
\nolinebreak%
\begin{minipage}[t]{0.33\textwidth}
  \centering
  \includegraphics[width=\linewidth]{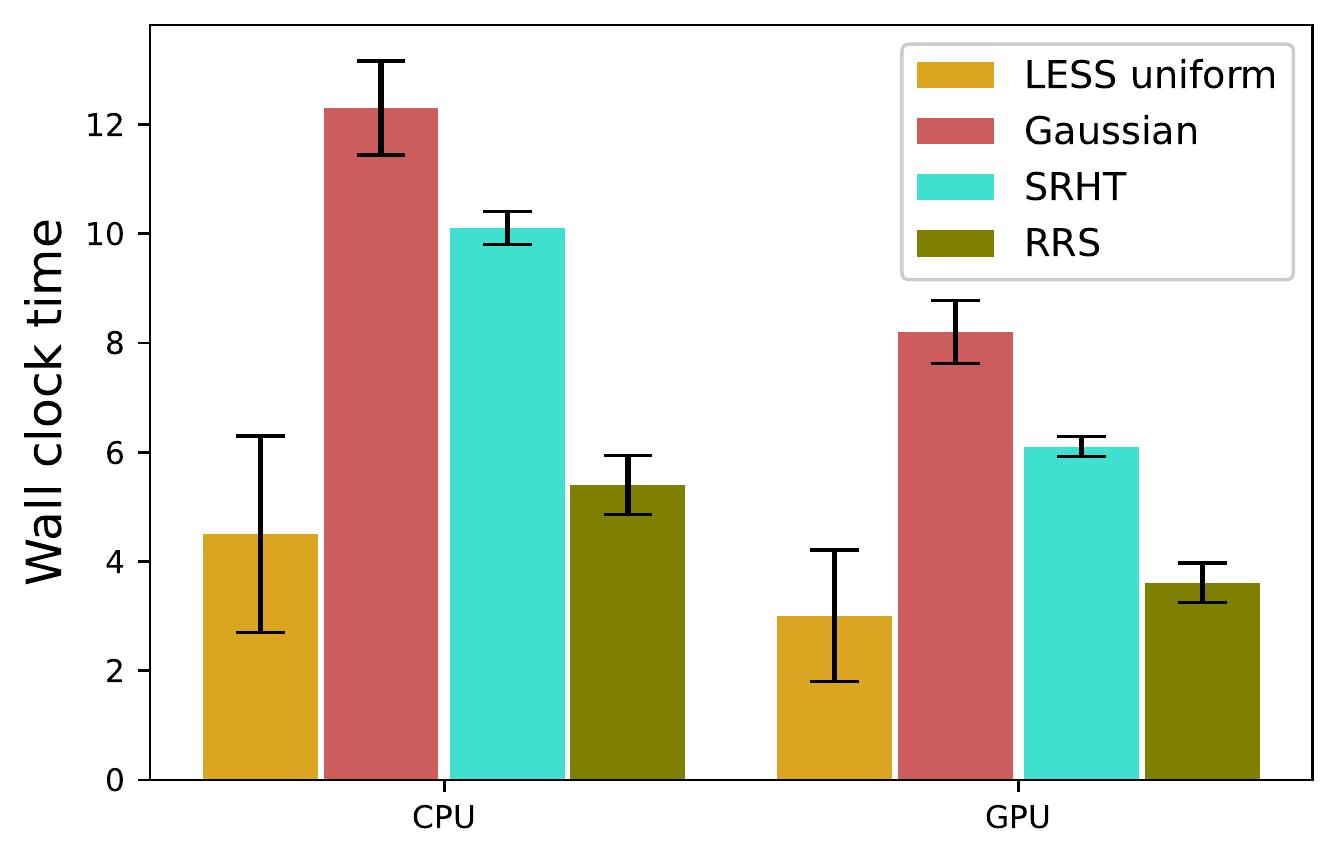}
\end{minipage}
\caption{Top plots show the convergence of Newton Sketch (NS) and
  baselines for logistic regression. We use a sketch size $m=d/2$ for
  NS. In the bottom plots, we report the CPU and GPU wall-clock times to reach a $10^{-6}$ accurate
  solution for NS with different sketching methods.} 
\label{fig:logistic-regression}
\end{figure}

We investigate first the effect of the row density of a LESS-uniform embedding on the Newton Sketch convergence rate and on the time for computing the sketch $\S \A$ for least squares regression. 
In Figure~\ref{fig:row-sparsity}, we report these two performance measures versus the row density.
(This is the empirical analog for real data of Figure~\ref{fig:sparsity-intro}).  
Note that here we also consider a synthetic data matrix with high-coherence, which aims to be more challenging for a uniform sparsifier. 
Remarkably, our prescribed row density of $d$ non-zeros per row offers
an excellent empirical trade-off: except for the CIFAR-10 dataset, for
which random row sampling performs equally well to Gaussian
embeddings, we observe that one can drastically decrease the row
density without significantly impairing the convergence~rate.

Next, in Figure \ref{fig:logistic-regression}, we investigate a
minimization task for the regularized logistic regression
loss. Namely, given a data matrix $\A \in \R^{n \times d}$ with rows denoted by $\mathbf{a}_i$, a target vector of labels $\b \in \{\pm 1\}^n$ and a regularization parameter $\lambda > 0$, the goal is to solve:
\begin{align*}
	\min_{\x \in \R^d} \frac{1}{n} \sum_{i=1}^n \log(1+\exp(-b_i
  \mathbf{a}_i^\top \x)) + \frac{\lambda}{2} \|\x\|_2^2\,. 
\end{align*}
For CIFAR-10 and Musk, we pick $\lambda=10^{-4}$. For WESAD, we pick
$\lambda = 10^{-5}$. We plot the error versus wall-clock time for the Newton
Sketch with LESS-uniform, Gaussian embedings, Subsampled Randomized
Hadamard Transform (SRHT) and Random Row Sampling (RRS) matrices, for
logistic regression, and we compare it with two standard second-order
optimization baselines: Newton's method and BFGS. We also included
three first-order baselines: Gradient Descent (GD), Accelerated GD
(AGD), and Stochastic GD (SGD), observing much worse performance.  
Based on the top plots in Figure \ref{fig:logistic-regression}, we
conclude that Newton-LESS (i.e., NS LESS uniform) offers significant
time speed-ups over all baselines.

Finally, we compare wall-clock time on different hardwares (CPU versus GPU) for the Newton Sketch to
reach a $10^{-6}$-accurate solution (see Appendix
\ref{a:additionalnumericalexperiments} for hardware details). From
Figure~\ref{fig:logistic-regression} (bottom plots), we conclude the following:
first, when switching from CPU to GPU for WESAD and CIFAR-10, Gaussian embeddings become more
efficient than SRHT, despite a worse time complexity, by taking
better advantage of the massively parallel architecture;
second, Random
Row Sampling performs better than either of them, despite
having much weaker theoretical guarantees; and 
third, LESS-uniform is more efficient than all three other
methods, on both CPU and GPU hardware platforms, observing a
significant speed-up when switching to the parallel GPU architecture.

%

\section{Conclusions}
\label{sec:conclusions}

We showed that, when constructing randomized Hessian estimates for
second-order optimization, we can get the best of both worlds: the
efficiency of Sub-Sampling methods and the precision of Gaussian
embeddings, by using sparse sketching matrices known as LEverage Score
Sparsified (LESS)
embeddings. Our algorithm, called Newton-LESS, enjoys both strong
theoretical convergence guarantees and excellent empirical performance
on a number of hardware platforms. An important future direction is to
explain the surprising effectiveness of the simpler LESS-uniform
method, particularly on
high-coherence matrices, which goes beyond the predictions of our current theory.

\subsubsection*{Acknowledgements}
We would like to acknowledge DARPA, IARPA, NSF, and ONR via its BRC on
RandNLA for providing partial support of this work.  Our conclusions
do not necessarily reflect the position or the policy of our sponsors,
and no official endorsement should be inferred.

\bibliography{../pap}
\bibliographystyle{alpha}


\appendix

\section{Characterization of Inverse Moments (Proof of Theorem \ref{t:structural})}
\label{a:structural}

In this section, we prove Theorem \ref{t:structural}. The proof
consists of two parts corresponding to the first and second moment of
$\Q$. The analysis of the first moment bound is nearly the same as in
\cite{less-embeddings}, so we only outline it here, highlighting the
differences coming from the regularization matrix $\C$. The analysis
of the second moment is our main contribution in this proof, and we discuss it
in detail. First, however, we define the high probability event $\Ec$
which is common to both parts.

To simplify the proof, we will let $m$ be divisible by $3$.
Note that we have $\Q=(\U^\top\S^\top\S\U+\D)^{-1}$ for
$\D=\H^{-\frac12}\C\H^{-\frac12}$. Moreover, let  $\S_{-i}$ denote $\S$ 
without the $i$th row, and let
$\Q_{-i}=(\U^\top\S_{-i}^\top\S_{-i}\U+\D)^{-1}$.
Also, define $\Sbt_1$, $\Sbt_2$, $\Sbt_3$ as the matrices consisting
of the first, second and third group of $m/3$ rows in $\S$, all
scaled by $\sqrt 3$, so that $\S^\top\S=\frac13\sum_{j=1}^3\Sbt_j^\top\Sbt_j$.
Next, using $\Sigmab=\U^\top\U$, similarly as in \cite{less-embeddings} we let $t=m/3$ and define three independent events:
\begin{align}
  \Ec_j: \quad
  \big\|\Sigmab-\U^\top\Sbt_j^\top\Sbt_j\U\big\|\leq
  \eta,\quad\text{for}\quad j=1,2,3,\label{eq:event} 
\end{align}
with $\Ec=\bigwedge_{j=1}^3 \Ec_j$ defined as the intersection of the
events. Conditioned on $\Ec$, we have:
\begin{align*}
\|\I-(\U^\top\S^\top\S\U+\D)\|
  &=\Big\|\frac13\sum_{j=1}^3\big(\Sigmab-\U^\top\Sbt_j^\top\Sbt_j\U\big)\Big\|
\leq \frac13\sum_{j=1}^3\|\Sigmab-\U^\top\Sbt_j^\top\Sbt_j\U\|\leq\eta,
\end{align*}
which implies that $\|\Q-\I\|\leq \frac{\eta}{1-\eta}\leq
2\eta$. Furthermore, an important property of the definition of $\Ec$
is that for each $i\in\{1,...,m\}$ there is a $j\in\{1,2,3\}$
such  that $\Ec_j$ is independent of $\x_i$, and after conditioning only on
$\Ec_j$ we get $\|\Q_{-i}\|\leq 6$.
From Condition \ref{cond1} and the union bound
we conclude that $\Pr(\Ec)\geq 1-\delta$.

The analysis of both the first and second moment uses the
Sherman-Morrison formula, to separate one of the rows from the
rest of the sketch. We state this formula in the following lemma.
\begin{lemma}[Sherman-Morrison]\label{lem:rank-one}
For $\A \in \R^{n \times n}$ invertible and $\u,\v \in \R^n$, $\A +
\u \v^\top$ is invertible if and only if $1+\v^\top \A^{-1} \u \neq 0$ and 
\begin{equation*}
    (\A + \u \v^\top)^{-1} = \A^{-1} - \frac{\A^{-1} \u \v^\top \A^{-1} }{1+\v^\top \A^{-1} \u}.
\end{equation*}
From the above formula, it follows that:
\begin{equation*}
    (\A + \u \v^\top)^{-1} \u = \frac{\A^{-1} \u}{1+\v^\top \A^{-1} \u}.
\end{equation*}
\end{lemma}

\vspace{-3mm}
\subsection{Proof of first moment bound}
In this part of the proof we recall the decomposition of
$\E_{\Ec}[\Q]$ used by \cite{less-embeddings}. Most of their analysis is
unaffected by the presence of the regularization matrix $\D$, so we
will focus on the steps that will also be needed for our analysis of the
second moment.  Let the
$i$th row of $\S$ be $\frac1{\sqrt{m-\deff}}\s_i^\top$, and define 
$\x_i=m\U^\top\s_i$, so that $\U^\top\S^\top\S\U=\frac\gamma
m\sum_i\x_i\x_i^\top$, where $\gamma=\frac{m}{m-\deff}$. Using
$\gamma_i=1+\frac\gamma{m}\x_i^\top\Q_{-i}\x_i$, we have: 
\begin{align*}
  \E_{\Ec}[\Q]-\I
  &= \E_{\Ec}[\Q(\Sigmab+\D) - \Q(\U^\top\S^\top\S\U+\D)]
  \\
  &=\E_{\Ec}[\Q\Sigmab] - \E_{\Ec}[\Q\U^\top\S^\top\S\U]
  \\
  &\overset{(*)}{=}\E_{\Ec}[\Q\Sigmab] -
    \E_{\Ec}[\tfrac\gamma{\gamma_i}\Q_{-i}\x_i\x_i^\top]
  \\
  &=\E_{\Ec}[\Q-\Q_{-i}]\Sigmab +
    \E_{\Ec}[\Q_{-i}(\Sigmab-\x_i\x_i^\top)]
    +\E_{\Ec}[(1-\tfrac\gamma{\gamma_i})\Q_{-i}\x_i\x_i^\top],
\end{align*}
where $(*)$ follows from the Sherman-Morrison formula.
From this point, the analysis of \cite{less-embeddings} proceeds to
bound the spectral norm of the first two terms by $O(1/m)$, and the
spectral norm of the last term by $O(\sqrt{\tr(\U^\top\U)}/m)$. In
their setup, $\C=\zero$, which means that $\tr(\U^\top\U)=d$,
whereas in our more general statement, we let
$\deff=\tr(\U^\top\U)$. This does not affect the proofs.
For the sake of our analysis of the second moment, we separate out the
following guarantees obtained by \cite{less-embeddings}, given here in
a slightly more general form than originally.
\begin{lemma}[\cite{less-embeddings}]\label{l:helper}
  The following bounds hold for $k\in\{1,2\}$:
  \begin{align*}
    \|\E_{\Ec}[\Q^k-\Q_{-i}^k]\Sigmab\|
    &= O(1/m),\\
    \|\E_{\Ec}[\Q_{-i}^k(\x_i\x_i^\top-\Sigmab)]\|
    &= O(1/m),\\
    \|\E_{\Ec}[(\tfrac\gamma{\gamma_i}-1)\Q_{-i}^k\x_i\x_i^\top]\|
    &=O(\sqrt \deff/m).
  \end{align*}
\end{lemma}

\vspace{-3mm}
\subsection{Proof of second moment bound}
We next present the analysis of the second moment, $\E_{\Ec}[\Q^2]$,
which requires a considerably more elaborate decomposition. Using
$\rho=\frac{\deff}{m-\tdeff}$ and $\Sigmab=\U^\top\U$, we have:
\begin{align*}
  \E_{\Ec}[\Q^2] -(\I + \rho\Sigmab) =
  \underbrace{\big(\E_{\Ec}[\Q]-\I\big)}_{\T_1}
  + \big(\E_{\Ec}[\Q(\Q-\I)]-\rho\Sigmab\big).
\end{align*}
Recalling that $\Sigmab+\D=\I$, we can rewrite the last term as:
\begin{align*}
  \E_{\Ec}[\Q(\Q-\I)]
  &=\E_{\Ec}[\Q(\Q(\Sigmab+\D) - \Q(\U^\top\S^\top\S\U+\D))]\\
  &= \E_{\Ec}[\Q(\Q\Sigmab-\Q\U^\top\S^\top\S\U )]
  \\
  &\overset{(a)}=\E_{\Ec}[\Q(\Q\Sigmab-\tfrac\gamma{\gamma_i}\Q_{-i}\x_i\x_i^\top)]
  \\
  &\overset{(b)}=
    \E_{\Ec}[\Q^2]\Sigmab
    -  \E_{\Ec}[\tfrac\gamma{\gamma_i}\Q_{-i}^2\x_i\x_i^\top] 
    +\E_{\Ec}\Big[\tfrac{\x_i^\top\Q_{-i}^2\x_i}{m}\,
    \tfrac{\gamma^2}{\gamma_i^2}\Q_{-i}\x_i\x_i^\top\Big]	
  \\
  &=\underbrace{\E_{\Ec}[\Q^2-\Q_{-i}^2]\Sigmab}_{\T_2} +
    \underbrace{\E_{\Ec}[\Q_{-i}^2(\Sigmab-\x_i\x_i^\top)]}_{\T_3}+
    \underbrace{\E_{\Ec}[(1-\tfrac\gamma{\gamma_i})\Q_{-i}^2\x_i\x_i^\top]}_{\T_4} +
    \E_{\Ec}\Big[\tfrac{\x_i^\top\Q_{-i}^2\x_i}{m}\,
    \tfrac{\gamma^2}{\gamma_i^2}\Q_{-i}\x_i\x_i^\top\Big],
\end{align*}
for a fixed $i$, where we denote
$\gamma_i=1+\frac{\gamma}{m}\x_i^\top\Q_{-i}\x_i$. Note that we used
the Sherman-Morrison formula twice, in steps $(a)$ and $(b)$. We can put everything 
together as follows:
\begin{align*}
  \E_{\Ec}[\Q^2] -(\I + \rho\Sigmab) 
  &= \T_1+\T_2+\T_3+\T_4\\
  &+
    \underbrace{\rho(\E_{\Ec}[\Q_{-i}]-\I)\Sigmab}_{\T_5}
    +\underbrace{\rho\,\E_{\Ec}[\Q_{-i}(\x_i\x_i^\top-\Sigmab)]}_{\T_6}
    +\underbrace{\E_{\Ec}\Big[\Big(\tfrac{\x_i^\top\Q_{-i}^2\x_i}m
    \,\tfrac{\gamma^2}{\gamma_i^2}-\rho\Big)\Q_{-i}\x_i\x_i^\top\Big]}_{\T_7}.
\end{align*}
From the bound on the first moment of $\Q$, we conclude that
$\|\T_1\|=O(\sqrt \deff/m)$ and that $\|\T_5\|=O(\sqrt\deff/m)$. Without loss of generality,
assume that events $\Ec_1$ and $\Ec_2$ are both independent of
$\x_i$, and let $\Ec'=\Ec_1\wedge\Ec_2$ as well as
$\delta_3=\Pr(\neg\Ec_3)$.
Next, we will use the fact that for a p.s.d.~random
matrix $\M$ in the probability space
of $\S$, we have $\E_{\Ec}[\M]\preceq\frac1{1-\delta}\E_{\Ec'}[\M]\preceq 2\cdot\E_{\Ec'}[\M]$.

Using $k=2$ 
in Lemma \ref{l:helper}, we can bound
$\|\T_2\|$, $\|\T_3\|$ and $\|\T_4\|$ by $O(\sqrt \deff/m)$, and setting
$k=1$, we can do the same for $\|\T_6\|$.

Thus, it remains  to bound $\|\T_7\|$. Let $\gammat = \frac m{m-\tdeff}$. We first use the Cauchy-Schwartz inequality
twice, obtaining that:
\begin{align}
  \|\T_7\|\leq \frac1m\sqrt{\E_{\Ec}\Big[\Big(\gammat \deff- \x_i^\top\Q_{-i}^2\x_i
  \cdot\tfrac{\gamma^2}{\gamma_i^2}\Big)^2\Big]}\cdot
  \sup_{\|\u\|=1}\sqrt[4]{\E_{\Ec}\big[(\u^\top\Q_{-i}\x_i)^4\big]}\cdot
  \sup_{\|\u\|=1}\sqrt[4]{\E_{\Ec}\big[(\x_i^\top\u)^4\big]}.\label{eq:t7}
\end{align}
The latter two terms can each be bounded  easily  by
$O(\sqrt[4]{\alpha+1})$ using Condition \ref{cond2}. For instance,
considering the middle term, we have:
\begin{align*}
  \E_{\Ec}\big[(\u^\top\Q_{-i}\x_i)^4\big]
  & \leq 2\,\E_{\Ec'}\Big[\E\big[(\x_i^\top\Q_{-i}\u\u^\top\Q_{-i}\x_i)^2\mid\Q_{-i}\big]\Big]
  \\
  &=
    2\,\E_{\Ec'}\Big[\Var\big[\x_i^\top\Q_{-i}\u\u^\top\Q_{-i}\x_i\mid\Q_{-i}\big]
    + \big(\E[\x_i^\top\Q_{-i}\u\u^\top\Q_{-i}\x_i\mid\Q_{-i}]\big)^2\Big]
  \\
  &\leq
    2\,\E_{\Ec'}\Big[\alpha\,\tr\big(\U(\Q_{-i}\u\u^\top\Q_{-i})^2\U^\top\big)
    +2\,\big(\tr(\U\Q_{-i}\u\u^\top\Q_{-i}\U^\top)\big)^2\Big]
    \\
  &\leq
   2\,\E_{\Ec'}\big[ O(\alpha)\,\u_i^\top\Q_{-i}\U^\top\U\Q_{-i}\u_i
    + (\u_i^\top\Q_{-i}\U^\top\U\Q_{-i}\u_i)^2\big]
    \\
  &\leq
    2\,\E_{\Ec'}\big[O(\alpha)\|\Q_{-i}\Sigmab\Q_{-i}\| +
    \|\Q_{-i}\Sigmab\Q_{-i}\|^2\big] = O(\alpha+1),
\end{align*}
where we also used that matrices $\Q_{-i}$, $\Sigmab$ and $\u\u^\top$
have spectral norms bounded by $O(1)$. Similarly, we obtain that
$\E_{\Ec}[(\x_i^\top\u)^4]=O(\alpha+1)$. For the first 
term in \eqref{eq:t7}, we have: 
\begin{align*}
  \E_{\Ec}\Big[\Big(\gammat \deff- \x_i^\top\Q_{-i}^2\x_i
  \cdot\tfrac{\gamma^2}{\gamma_i^2}\Big)^2\Big]
  &\leq 2\cdot \E_{\Ec'}\Big[\Big(\gammat \deff- \x_i^\top\Q_{-i}^2\x_i
    \cdot\tfrac{\gamma^2}{\gamma_i^2}\Big)^2\Big]
  \\
  &\leq 4\cdot \E_{\Ec'}\big[(\gammat\deff-\x_i^\top\Q_{-i}^2\x_i)^2\big]
    +4\cdot\E_{\Ec'}\Big[\big(\x_i^\top\Q_{-i}^2\x_i\big)^2
    \Big(\tfrac{\gamma^2}{\gamma_i^2}-1\Big)^2\Big].
\end{align*}
We can further break down the first term as follows:
\begin{align}
  \E_{\Ec'}\big[(\gammat \deff-\x_i^\top\Q_{-i}^2\x_i)^2\big]
  =(\gammat \deff-\E_{\Ec'}[\tr(\Q_{-i}^2\Sigmab)])^2 +
  \Var_{\Ec'}[\tr(\Q_{-i}^2\Sigmab)]+
  \E_{\Ec'}\big[(\tr(\Q_{-i}^2\Sigmab)-\x_i^\top\Q_{-i}^2\x_i)^2\big]
  \label{eq:three-terms}
\end{align}
The latter term can be bounded immediately using
Condition~\ref{cond2}. The middle term is handled by a separate
lemma, which is an immediate extension of Lemma 25 in \cite{less-embeddings}.
\begin{lemma}[\cite{less-embeddings}]\label{l:applying-burkholder}
  Let $\Var_{\Ec'}[\cdot]$ be the conditional variance
  with respect to event $\Ec'=\Ec_1\wedge\Ec_2$. Then, for $k\in\{1,2\}$, 
  \begin{align*}
    \Var_{\Ec'}\!\big[\tr(\Q_{-i}^k\Sigmab)\big] =O(\deff).
  \end{align*}
\end{lemma}
\noindent
Next, note that
$|\frac{\gamma^2}{\gamma_i^2}-1|=|\gamma-\gamma_i|\cdot\frac{\gamma+\gamma_i}{\gamma_i^2}\leq
|\gamma-\gamma_i|\cdot\frac{\gamma+1}{\gamma_i}$, since $\gamma_i>1$,
so we get:
\begin{align*}
  \E_{\Ec'}\Big[\big(\x_i^\top\Q_{-i}^2\x_i\big)^2
  \Big(\tfrac{\gamma^2}{\gamma_i^2}-1\Big)^2\Big]
  &\leq 6^2(\gamma+1)^2\cdot
    \E_{\Ec'}\Big[(\x_i^\top\Q_{-i}\x_i)^2\frac{(\gamma-\gamma_i)^2}{\gamma_i^2}\Big]
  \\
  &\leq 6^2(\gamma+1)^2\cdot
    \E_{\Ec'}\Big[\frac{(\x_i^\top\Q_{-i}\x_i)^2}{(1+\frac\gamma m
    \x_i^\top\Q_{-i}\x_i)^2}(\gamma-\gamma_i)^2\Big]
  \\
  &\leq O(m^2)\cdot \E_{\Ec'}\big[(\gamma-\gamma_i)^2\big]
  \\
  &\leq O(m^2)\cdot O(\alpha \deff/m^2) = O(\alpha \deff).
\end{align*}
Finally, we analyze the first term in \eqref{eq:three-terms} as
follows:
\begin{align*}
  \big|\gammat \deff-\E_{\Ec'}[\tr(\Q_{-i}^2\Sigmab)]\big|
  &=
    \big|\tr((\E_{\Ec}-\E_{\Ec'})[\Q_{-i}^2\Sigmab]) - \tr(\T_2) +
    \tr(\gammat\Sigmab-\E_{\Ec}[\Q^2]\Sigmab)\big|
  \\
  &=
    \big|\tr((\E_{\Ec}-\E_{\Ec'})[\Q_{-i}^2\Sigmab]) - \tr(\T_2) +
    \tr((\I+\rho\Sigmab-\E_{\Ec}[\Q^2])\Sigmab)\big|
  \\
  &\leq O(\deff/m^3)
    + \deff\cdot O(\alpha\sqrt \deff/m) + |\tr(\T_7\Sigmab)|,
\end{align*}
where to bound the first term we used  the fact that
$\Pr(\neg \Ec_3\mid \Ec')\leq 1/m^3$ and for the last term, recall
that $\deff=\tr(\Sigmab)$ and $\tdeff=\tr(\Sigmab^2)$, which leads to the
following identity:
\begin{align*}
  \tr\big(\gammat\Sigmab - (\I+\rho\Sigmab)\Sigmab\big)
  &=\tr\big(\tfrac{\tdeff}{m-\tdeff}\Sigmab -
    \tfrac{\deff}{m-\tdeff}\Sigmab^2\big) = 0.
\end{align*}
Further, note that from the analysis of $\T_7$ we have:  
\begin{align*}
  |\tr(\T_7)|
  &\leq \frac \deff m \sqrt{
    4\big(\gammat \deff-\E_{\Ec'}[\tr(\Q_{-i}^2)]\big)^2 + O(\alpha \deff)}
    \cdot O(\sqrt \alpha)
  \\
  &\leq O(\alpha \deff/m)\cdot \big(|\gammat \deff-\E_{\Ec'}[\tr(\Q_{-i}^2)]
    + \sqrt \deff\big)  .
\end{align*}
Putting this together with the previous inequality, we conclude that
for sufficiently large $m$: 
\begin{align*}
  \big|\gammat \deff-\E_{\Ec'}[\tr(\Q_{-i}^2)]\big|\leq \frac{O(\alpha\sqrt
  \deff)}{1-O(\alpha \deff/m)} = O(\alpha\sqrt \deff).
\end{align*}
Plugging this back into the analysis of $\|\T_7\|$, we can bound it by
$O(\alpha\sqrt d/m)$,  which concludes the proof.


\vspace{-3mm}
\section{Local Convergence Rate of Newton-LESS}
\label{a:convergence}
In this section, we present the convergence analysis of Newton
Sketch for sketching matrices satisfying the structural conditions of
Theorem \ref{t:structural}. We start by proving Lemma
\ref{l:regularized-decomposition}, then we show how it can be used to
establish the guarantee from Theorem
\ref{t:regularized-full}. Finally, we discuss how the analysis needs
to be adjusted to obtain the two-sided bound from Theorem
\ref{t:main-simple}.

\vspace{-3mm}
\subsection{Proof of Lemma \ref{l:regularized-decomposition}}
  Let $\Deltat_t=\xbt_t-\x^*$, $\Delta_{t+1}=\x_{t+1}-\x^*$, and
  $\p_t=\x_{t+1}-\x_t$. Also, define $\rho=\frac{\deff}{m-\tdeff}$ as
  well as the matrices
$\Qbt =
\H_t^{\frac12}(\Af{f_0}{\xbt_t}^\top\S_t^\top\S_t\Af{f_0}{\xbt_t}+\nabla^2
g(\xbt_t))^{-1}\H_t^{\frac12}$ and
$\U_t=\Af{f_0}{\xbt}\H_t^{-\frac12}$. We have:
  \begin{align*}
    \E_{\Ec}\,\|\Deltat_{t+1}\|_{\H_t}^2-\|\Delta_{t+1}\|_{\H_t}^2
    &=    
    2\Delta_{t+1}^\top\H_t\E_{\Ec}[\xbt_{t+1}-\x_{t+1}] +
    \E_{\Ec}\,\|\xbt_{t+1}-\x_{t+1}\|_{\H_t}^2
    \\
    &=
      2\Delta_{t+1}^\top\H_t^{\frac12}(\I-\E_{\Ec}\Qbt)\H_t^{\frac12}\p_t
      +\p_t^\top\H_t^{\frac12}\E_{\Ec}(\I-\Qbt)^2\H_t^{\frac12}\p_t
    \\
&\leq       \rho\cdot\p_t^\top\H_t^{\frac12}\U_t^\top\U_t\H_t^{\frac12}\p_t
                +
     2 \|\I-\E_{\Ec}\Qbt\|\cdot\big(\|\Delta_{t+1}\|_{\H_t}\|\p_t\|_{\H_t}+\|\p_t\|_{\H_t}^2\big)
    \\
    &\quad+\|\I+\rho\U_t^\top\U_t-
      \E_{\Ec}\Qbt^2\|\cdot\|\p_t\|_{\H_t}^2
    \\
&\leq \rho\,\|\p_t\|_{\nabla^2 f_0(\xbt_t)}^2 + O\big(\tfrac{\sqrt
         \deff}{m}\big) \|\Deltat_t\|_{\H_t}^2,
  \end{align*}
  where the last step follows by applying Theorem \ref{t:structural} and observing that
  $\|\p_t\|_{\H_t}\leq\|\Deltat_t\|_{\H_t}+\|\Delta_{t+1}\|_{\H_t}\leq
  2\|\Deltat_t\|_{\H_t}$ and
  $\|\Delta_{t+1}\|_{\H_t}\|\p_t\|_{\H_t}\leq
  \frac12(\|\Delta_{t+1}\|_{\H_t}^2+\|\p_t\|_{\H_t}^2)\leq
  3\|\Deltat_t\|_{\H_t}^2$. The matching lower-bound follows identically.

\vspace{-3mm}
  \subsection{Proof of Theorem \ref{t:regularized-full}}
We start by analyzing the exact Newton step 
  $\x_{t+1}=\xbt_t-\mu_t\H_t^{-1}\g_t$ wih step size $\mu_t$, gradient
  $\g_t=\nabla 
  f(\xbt_t)$, and Hessian $\H_t=\nabla^2 f(\xbt_t)$. Letting
$\Deltat_t=\xbt_t-\x^*$ and $\Delta_{t+1}=\x_{t+1}-\x^*$, we have:
  \begin{align*}
    \|\Delta_{t+1}\|_{\H_t}^2
    &=  (1-\mu_t)\Delta_{t+1}^\top\g_t
      +\Delta_{t+1}^\top(\H_t\Deltat_t-\g_t)
    \\
    &=(1-\mu_t)\Delta_{t+1}^\top\H_t\Deltat_t -
      (1-\mu_t)\Delta_{t+1}^\top(\H_t\Deltat_t-\g_t)
      +\Delta_{t+1}^\top(\H_t\Deltat_t-\g_t)
    \\
    &
 =(1-\mu_t)\big(\Deltat_t^\top\H_t\Deltat_t -
      \mu_t\g_t^\top\Deltat_t\big) + \mu_t \Delta_{t+1}^\top(\H_t\Deltat_t-\g_t)
     \\
    &= (1-\mu_t)^2\|\Deltat_t\|_{\H_t}^2
+\mu_t\big(\Delta_{t+1}+(1-\mu_t)\Deltat_t\big)^\top(\H_t\Deltat_t-\g_t).
  \end{align*}
Before we proceed, we make the following assumptions,
  which will be addressed later.
  \begin{align}
    \text{Assume:}\qquad \|\H_t\Deltat_t-\g_t\|_{\H_t^{-1}}\leq
    \epsilon\beta\|\Deltat_t\|_{\H_t},
    \qquad \H_t\approx_{\epsilon} \H,\label{eq:assumptions}
  \end{align}
  where $\epsilon=O(\frac1{\sqrt\deff})$ and $\beta = \frac\rho{1+\rho}$ will become the
  convergence rate of Newton-LESS, and recall that
  $\rho=\frac{\deff}{m-\tdeff}$. Now, using the Cauchy-Schwartz
  inequality we obtain that:
  \begin{align*}
    \|\Delta_{t+1}\|_{\H_t}^2
    &\leq
       (1-\mu_t)^2\|\Deltat_t\|_{\H_t}^2
      +\mu_t\|\Delta_{t+1}+(1-\mu_t)\Deltat_t\|_{\H_t}\|\H_t\Deltat_t-\g_t\|_{\H_t^{-1}}
    \\
    &\leq (1-\mu_t)^2\|\Deltat_t\|_{\H_t}^2
      +\epsilon\beta\mu_t\|\Delta_{t+1}\|_{\H_t}\|\Deltat_t\|_{\H_t} +
      \epsilon\beta\mu_t(1-\mu_t)\|\Deltat_t\|_{\H_t}^2. 
  \end{align*}
  Solving for $\|\Delta_{t+1}\|_{\H_t}$, we obtain the following
  upper bound:
  \begin{align*}
    \|\Delta_{t+1}\|_{\H_t}^2
    &\leq
    2(1-\mu_t)^2\|\Deltat_t\|_{\H_t}^2+2\epsilon\beta\mu_t(1-\mu_t)\|\Deltat_t\|_{\H_t}^2
      +\epsilon^2\beta^2\mu_t^2\|\Deltat_t\|_{\H_t}^2
    \\
    &\leq2\big((1-\mu_t)^2 +\epsilon\beta\mu_t\big)\|\Deltat_t\|_{\H_t}^2.
  \end{align*}
  Setting $\mu_t=\frac1{1+\rho}$, we conclude that
$\|\Delta_{t+1}\|_{\H_t}^2\leq
\frac{2\rho}{1+\rho}\frac{\rho+\epsilon}{1+\rho}\|\Deltat_t\|_{\H_t}^2\leq\beta
\|\Deltat_t\|_{\H_t}^2$ when $m\geq 4\deff+2$ and $\epsilon<1/4$.
  Next, we return to the Newton Sketch. Recall that using
  Lemma~\ref{l:regularized-decomposition} with the event $\Ec$ having
  failure probability $\delta/T$, we have: 
  \begin{align}
    \E_{\Ec}\,\|\Deltat_{t+1}\|_{\H_t}^2
    -\|\Delta_{t+1}\|_{\H_t}^2
    &=
\rho\|\p_t\|_{\nabla^2 f_0(\xbt_t)}^2 \pm O\big(\tfrac{\sqrt
      \deff}{m}\big)\|\Deltat_t\|_{\H_t}^2\nonumber
    \\
    &\leq \rho\|\p_t\|_{\H_t}^2 + O\big(\tfrac{\sqrt
      \deff}{m}\big)\|\Deltat_t\|_{\H_t}^2,\label{eq:upper-not-lower}
  \end{align}
  where we also used the fact that $\nabla^2f_0(\xbt_t)\preceq\H_t$.
  The leading term in the above decomposition can be written as follows:
  \begin{align*}
\rho\|\p_t\|_{\H_t}^2  &=\rho\mu_t^2\,(\g_t\Deltat_t
      -\g_t^\top\H_t^{-1}(\H_t\Deltat_t-\g_t))
    \\
    &=\rho\mu_t^2\,(\Deltat_t^\top\H_t\Deltat_t -
      \Deltat_t^\top(\H_t\Deltat_t-\g_t)-\g_t^\top\H_t^{-1}(\H_t\Deltat_t-\g_t))
    \\
    &=\rho\mu_t^2\|\Deltat_t\|_{\H_t}^2 -
      \rho\mu_t^2(\Deltat_t+\H_t^{-1}\g_t)^\top(\H_t\Deltat_t-\g_t).
  \end{align*}
  Putting everything together, and then setting $\mu_t=\frac1{1+\rho}$, we obtain:
  \begin{align*}
    \E_{\Ec}\,\|\Deltat_{t+1}\|_{\H_t}^2
         & \leq\Big((1-\mu_t)^2 +
           \rho\mu_t^2 + O\big(\tfrac{\sqrt\deff}m\big)\Big)\|\Deltat_t\|_{\H_t}^2
    \\
          &\quad+\big((1+\rho)\mu_t(\Delta_{t+1}-\Deltat_t)+(1-\rho\mu_t^2)\Deltat_t
            \big)^\top(\H_t\Deltat_t-\g_t)
    \\
         &=\Big(\beta+O\big(\tfrac{\sqrt\deff}m\big)\Big)\|\Deltat_t\|_{\H_t}^2
           + \big(\Delta_{t+1} - \beta\mu_t\Deltat_t\big)^\top (\H_t\Deltat_t-\g_t).
  \end{align*}
We can bound the second term by using Cauchy-Schwartz, the first
assumption in \eqref{eq:assumptions} and $\mu_t\leq 1$:
\begin{align*}
\big| \big(\Delta_{t+1} - \beta\mu_t\Deltat_t\big)^\top
  (\H_t\Deltat_t-\g_t)\big|
  &\leq \epsilon\beta\|\Delta_{t+1}\|_{\H_t}\|\Deltat_t\|_{\H_t}+
    \epsilon\beta^2\|\Deltat_t\|_{\H_t}^2
    \leq 2\epsilon\beta\|\Deltat_t\|_{\H_t}^2.
\end{align*}
Combining this with the assumption $\H_t\approx_{\epsilon} \H$,
which implies that $\|\v\|_{\H_t}^2\approx_\epsilon\|\v\|_{\H}^2$, we obtain:
\begin{align}
  \E_{\delta/T}\,\frac{\|\widetilde\Delta_{t+1}\|_{\H}^2}{\|\Deltat_t\|_{\H}^2}
 \leq\beta\cdot\Big(1+O\big(\tfrac1{\sqrt\deff}\big)\Big).\label{eq:step-t}
\end{align}
Note that since $\rho=\frac{\deff}{m-\tdeff}$, we have $\beta= \frac
\deff{m+\deff-\tdeff}$ and $\mu_t=\frac{\deff}{m+\deff-\tdeff}$. Alternatively, if throughout the analysis we
use $\rho=\frac\deff{m-\deff}\leq \frac{\deff}{m-\tdeff}$, then we
obtain the simpler (and slightly weaker) convergence rate $\beta =
\frac{\deff}m$ with step size $\mu_t=1-\frac\deff m$, 
as in Theorem~\ref{t:regularized-simple}.

It remains to
address the assumptions from \eqref{eq:assumptions}, and then carefully chain
the expectations together. Next, we define the 
neighborhood $U$ in which we can establish our convergence guarantee,
and show that when the iterate lies in the neighborhood, then
\eqref{eq:assumptions} is satisfied. This part of the proof
will depend on what type of function $f(\x)$ we are minimizing.

\paragraph{Lipschitz Hessian.} Suppose that function $f(\x)$ has a
Lipschitz continuous Hessian with constant~$L$ (Assumption \ref{a:lipschitz}). We define the
neighborhood $U$ through the following condition:
\[\|\Deltat_t\|_{\H}<\frac{\sqrt\deff}{m}\frac{(\lambda_{\min})^{3/2}}{L},\]
where $\lambda_{\min}$
denotes the smallest eigenvalue of $\H$. Suppose that the
condition holds for some $t$. Then, we have:
\begin{align*}
\|\H^{-\frac12}(\H_t-\H)\H^{-\frac12}\|\leq
  \frac1{\lambda_{\min}}\|\H_t-\H\|\leq
  \frac L{\lambda_{\min}}\|\Deltat_t\|\leq
    \frac L{(\lambda_{\min})^{3/2}}\|\Deltat_t\|_{\H}\leq
\frac{\sqrt\deff}{m}\leq \epsilon,
\end{align*}
for $\epsilon=O(\frac1{\sqrt\deff})$, showing that $\H_t\approx_\epsilon \H$. In particular, this implies
that $\|\H_t^{-1}\|\geq \frac1{\lambda_{\min} (1-\epsilon)}$. To get
the second assumption in \eqref{eq:assumptions}, we first follow 
standard analysis of the Newton's method \cite{boyd2004convex}:
\begin{align*}
  \|\H_t\Deltat_t-\g_t\|
  &=
    \bigg\|\H_t\Deltat_t-\Big(\int_0^1\nabla^2f(\x^*+\tau\Deltat_t)
d\tau\Big)\Deltat_t\bigg\|
    \\
  &\leq
    \|\Deltat_t\|\cdot
    \int_0^1\big\|\nabla^2 f(\xbt_t)-\nabla^2f(\x^*+\tau\Deltat_t)\big\| d\tau
  \\
  &\leq  \|\Deltat_t\|\cdot \int_0^1(1-\tau)L\|\Deltat_t\| d\tau
    \leq \frac L2\|\Deltat_t\|^2.
\end{align*}
Then, we simply use the fact that $\|\H_t^{-1}\|\geq
\frac1{\lambda_{\min} (1-\epsilon)}$ to conclude:
\begin{align*}
  \|\H_t\Deltat_t-\g_t\|_{\H_t^{-1}}
  &\leq
  \frac1{\sqrt{\lambda_{\min}(1-\epsilon)}} \|\H_t\Deltat_t-\g_t\|
  \\
  &\leq \frac1{\sqrt{\lambda_{\min}(1-\epsilon)}}\,\frac
    L2\|\Deltat_t\|^2
  \\
  &\leq\frac{L}{(\lambda_{\min})^{3/2}}\|\Deltat_t\|_{\H}^2
  \leq  \epsilon\beta\|\Deltat_t\|_{\H_t},
\end{align*}
since $\beta = O(\frac{\deff}{m})$,
thus establishing the assumptions from \eqref{eq:assumptions}. 

\paragraph{Self-concordant function.} Suppose that function $f(\x)$ is
self-concordant (Assumption \ref{a:self-concordant}). We will define the neighborhood $U$ through the
following condition:
\begin{align*}
  \|\Deltat_t\|_{\H}<\frac{\sqrt\deff}{m}.
\end{align*}
Now, using a standard property of self-concordant functions
\cite[Chapter 9]{boyd2004convex}, we have:
\begin{align*}
  (1-\|\Deltat_t\|_{\H})^2\,\H\preceq\H_t\preceq\frac1{(1-\|\Deltat_t\|_{\H})^2}\,\H,
\end{align*}
and note that $\frac1{(1-\|\Deltat_t\|_{\H})^2} \leq 1+\epsilon$ for
$\epsilon=O(\frac1{\sqrt\deff})$, so it
follows that $\H_t\approx_{\epsilon}\H$. Furthermore, for
self-concordant functions, it follows that:
\begin{align*}
  \|\H_t\Deltat_t-\g_t\|_{\H_t^{-1}}
  &=\bigg\|\H_t^{\frac12}\Deltat_t -
    \H_t^{-\frac12}\Big(\int_0^1\nabla^2f(\x^*+\tau\Deltat_t)d\tau\Big)
    \Deltat_t\bigg\|
  \\
  &=\bigg\|\Big(\int_0^1\big(\I-\H_t^{-\frac12}\nabla^2f(\x^*+\tau\Deltat_t)
    \H_t^{-\frac12}\big)d\tau\Big)\H_t^{\frac12}\Deltat_t\bigg\|
  \\
  &\leq
    \|\Deltat_t\|_{\H_t}\cdot\int_0^1\big\|\I-\H_t^{-\frac12}\nabla^2f(\x^*+\tau\Deltat_t)
    \H_t^{-\frac12}\big\|d\tau
  \\
  &\leq
    \|\Deltat_t\|_{\H_t}\cdot
    \int_0^1\frac{1}{(1-\tau\|\Deltat_t\|_{\H_t})^2} d\tau
= \frac{\|\Deltat_t\|_{\H_t}^2}{1-\|\Deltat_t\|_{\H_t}}.
\end{align*}
Using the neighborhood condition, we conclude that
$\frac{\|\Deltat_t\|_{\H_t}^2}{1-\|\Deltat_t\|_{\H_t}}
\leq O(\frac{\sqrt\deff}m)\|\Deltat_t\|_{\H_t}\leq \epsilon\beta \|\Deltat_t\|_{\H_t}$.

\paragraph{Chaining the expectations.}
 Let $\Ec_t$ denote the high-probability event corresponding to the
conditional expectation in \eqref{eq:step-t} for the iteration $t$. It remains to show that
after conditioning on event $\Ec=\bigwedge_{t=0}^{T-1}\Ec_t$,
we maintain that $\xbt_t\in U$ for all $t$. Assume that this holds for $t=0$. Then,
it suffices to show that $\|\Deltat_{t+1}\|_{\H}\leq \|\Deltat_{t}\|_{\H}$
for every $t$ almost surely (conditioned on $\Ec$). Recall that
Theorem \ref{t:structural} implies that conditioned on $\Ec_t$ we have
\begin{align*}
  \|\I-\Qbt\|\leq \eta,
\end{align*}
where Lemma~\ref{l:structural-less} ensures that $\eta$ is small.
We use this to show the following coarse convergence guarantee that
holds almost surely conditioned on $\Ec_t$, but is substantially
weaker than $\beta$. First, note that using
the derivation as in the proof of Lemma
\ref{l:regularized-decomposition} and the analysis of the exact Newton
step, 
\begin{align*}
  \|\Deltat_{t+1}\|_{\H_t}^2 
  &\leq \|\Delta_{t+1}\|_{\H_t}^2+ O\big(\|\I-\Qbt\|\big)\cdot
    \|\Deltat_t\|_{\H_t}^2
  \\
  &\leq \big(\beta + O(\eta)\big)\cdot\|\Deltat_t\|_{\H_t}^2.
\end{align*}
Using a sufficiently large constant $C$ in
Lemma~\ref{l:structural-less} so that $\beta+O(\eta)$ is small enough,
and given the assumption $\H_t\approx_\epsilon\H$, we obtain
 $\|\Deltat_{t+1}\|_{\H}^2\leq \|\Deltat_t\|_{\H}^2$.
Thus, we conclude that
all of the iterates will lie in the neighborhood $U$, and so
\eqref{eq:step-t} will hold for all $t=0,1,...,T-1$. Finally, note
that by the union bound, event $\Ec$ holds with probability
$1-\delta$, which completes the proof.

\paragraph{Lower-bound from Theorem \ref{t:main-simple}}
The matching lower-bound from Theorem \ref{t:main-simple} holds only
in the unregularized setting. In this case, we have
$\nabla^2f_0(\xbt_t)=\H_t$, so instead of an inequality in
\eqref{eq:upper-not-lower}, we can obtain a two-sided approximation. The rest
of the proof proceeds identically.

\vspace{-3mm}
\section{Sketches Satisfying Structural Conditions
  (Proof of Lemma \ref{l:structural-less})} 
\label{a:structural-less}

In this section, we prove Lemma \ref{l:structural-less}, showing that sub-Gaussian, LESS, and LESS-uniform
embeddings all satisfy the assumptions of Theorem \ref{t:structural},
which are derived from Conditions \ref{cond1} and \ref{cond2}. This
analysis follows along similar lines as in \cite{less-embeddings},
except for extending LESS embeddings to LESS-uniform, and allowing for
the presence of regularization.

\vspace{-3mm}
\subsection{Sub-Gaussian embeddings}

For sub-Gaussian embeddings, both conditions follow from existing
results. To establish Condition \ref{cond1}, we rely on a covariance
estimation result of \cite{koltchinskii2017concentration}, stated as
in \cite{mendelson2020robust}. Here, we say that a random vector $\x$
is sub-Gaussian if $\v^\top\x$ is a sub-Gaussian variable for all unit
vectors $\v$.
\begin{lemma}[{\cite[Theorem~1.4]{mendelson2020robust}}]\label{lem:cov-concentration}
For $i \in \{1,\ldots, m\}$, let $\x_i \in \R^{n}$ be independent
sub-Gaussian random vectors such that $\E[\x_i] = 0$ and
$\E[\x_i \x_i^\top] = \Sigmab$. Then, it holds with
probability at least $1 - 2\exp(-t^2)$ that 
\[
	\left\| \frac1m \sum_{i=1}^m \x_i \x_i^\top - \Sigmab \right\|
        \le C \| \Sigmab \| \left( \sqrt{ \frac{\tr \Sigmab / \|
              \Sigmab \|}{ m} } + \frac{\tr \Sigmab / \| \Sigmab
            \|}{m}  + \frac{t}{\sqrt m} + \frac{t^2}m \right). 
\]
\end{lemma}
Now, to establish Condition \ref{cond1} for the sub-Gaussian sketching
matrix
$\S$, i.e., where the $m$ rows are distributed as
$\frac1{\sqrt{m-\deff}}\s_i^\top$ for $\s_i$ having
i.i.d. zero mean, unit variance and sub-Gaussian entries, we let $\x_i=\U^\top\s_i$. Recall
that $\U=\A\H^{-\frac12}$ for $\H=\A^\top\A+\C$, so $\|\U\|\leq 1$, so
it follows that $\x_i$ is a sub-Gaussian random vector. Therefore,
letting $\gamma = \frac{m}{m-\deff}$ and $\Sigmab = \U^\top\U$, we
have $\E[\x_i\x_i^\top]=\Sigmab$ and with probability $1-\delta$: 
\begin{align*}
  \|\U^\top\S^\top\S\U-\Sigmab\|
  &\leq
  \gamma\cdot\Big\|\frac1m\sum_{i=1}^m\x_i\x_i^\top - \Sigmab\Big\|
  + (\gamma-1)\cdot\|\Sigmab\|
  \\
  &\leq C\bigg(\sqrt{\frac{\deff}{m}} +
    \sqrt{\frac{\log(1/\delta)}{m}}\bigg) + \frac{\deff}{m-\deff},
\end{align*}
thus setting $m\geq O(1)\cdot(\deff+\log(1/\delta))/\eta^2$, we can
bound the above by $\eta$, obtaining Condition \ref{cond1}.

To show Condition \ref{cond2} for sub-Gaussian embeddings, we can
again rely on a more general moment bound for quadratic
forms, which is a special case of Lemma B.26 in \cite{bai2010spectral}.
\begin{lemma}[\cite{bai2010spectral}]\label{l:bai-silverstein}
Let $\M$ be a $n\times n$ matrix, and let $\x$ be an $n$-dimensional
random vector with independent, mean zero, unit variance entries such
that $\E[x_i^4] = O(1)$. Then,
\begin{align*}
\Var[\x^\top\M\x] \leq O(1)\cdot\tr(\M\M^\top).
\end{align*}
\end{lemma}
To obtain Condition \ref{cond2}, we simply set $\x=\s$ and
$\M=\U\B\U^\top$. Note that since $\|\U\|\leq1$, we have
$\tr(\M\M^\top) = \tr(\U\B\U^\top\U\B\U^\top)\leq \tr(\U\B^2\U^\top)$.

\vspace{-3mm}
\subsection{LESS embeddings: Condtion \ref{cond1}}

Now, we demonstrate that Condition \ref{cond1} also
holds for LESS embeddings. We will use 
the following matrix concentration inequality which is a
straightforward combination of two standard results.
\begin{lemma}[{\cite[Theorem~6.2]{matrix-tail-bounds} and \cite[Theorem 7.7.1]{tropp2015introduction}}]\label{l:bernstein}
For $i=1,2,...$, consider a finite sequence $\X_i$ of $d\times d$
independent symmetric random matrices such that
$\E[\X_i]=\mathbf{0}$, and one of the following holds for all $i$:
\begin{enumerate}
\item $\E[\X_i^p] \preceq \frac{p!}2 \cdot R^{p-2}
  \A_i^2$\quad for\quad $p=2,3,...$;
  \item $\|\X_i\|\leq R$\quad and\quad $\E[\X_i^2]\preceq \A_i^2$.
\end{enumerate}
Then, defining the variance matrix $\V=\sum_i\A_i^2$,
parameter $\sigma^2 = \|\V\|$ and $\deff=\tr(\V)/\|\V\|$, for any $t\geq\sigma+R$ we have:
\begin{align*}
	\Pr \bigg\{ \lambda_{\max}\Big( \sum\nolimits_i \X_i \Big) \geq t
  \bigg\}& \leq 4\deff \cdot \exp \left( \frac{ -t^2/2 }{ \sigma^2 +
            R t } \right).
\end{align*}
\end{lemma}
Before we can use matrix concentration, we must first establish
high-probability concentration of the quadratic form
$\s^\top\U\U^\top\s$, for a leverage score sparsified sub-Gaussian
random vector $\s$. This is an analog of the Hanson-Wright inequality,
which holds for non-sparsified sub-Gaussian random vectors, as given
below. 
\begin{lemma}[Hanson-Wright inequality,
  {\cite[Theorem~1.1]{rudelson2013hanson}}]\label{l:hanson-wright} 
Let $\x$ have independent sub-Gaussian entries with
mean zero and unit variance. Then, there is $c=\Omega(1)$ such that
for any $n\times n$ matrix $\B$ and $t\geq 0$,
\begin{align*}
  \Pr\Big\{|\x^\top\B\x-\tr(\B)|\geq t\Big\}\leq
  2\exp\bigg(-c\min\Big\{\frac{t^2}{\|\B\|_F^2},\frac{t}{\|\B\|}\Big\}\bigg).
\end{align*}
\end{lemma}
Our version of this result for sparsified sub-Gaussian vectors is an
extension of Lemma 31 of \cite{less-embeddings}, introducing the
effective dimension $\deff$ as opposed to the regular dimension $d$,
and allowing a broader class of sparsifiers, so that we can cover the
results for LESS-uniform embeddings.
\begin{lemma}\label{l:restricted-hanson-wright}
Let $\U=\A\H^{-\frac12}$ for $\H=\A^\top\A+\C$, and let
$\deff=\tr(\U^\top\U)$. Let $\xib$ be a $(p,s)$-sparsifier and $\x$ have
indepedent sub-Gaussian entries with mean zero and unit
variance. If $p_i= \Omega(l_i(\A,\C)/s)$ for all $i$, then for any
$t\geq C\deff$, vector $\s=\x\circ\xib$ satisfies: 
\begin{align*}
\Pr\Big\{\s^\top\U\U^\top\s\geq t\Big\}\leq
  \exp\Big(-c\,\big(\sqrt t + t/\deff\big)\Big).
\end{align*}
\end{lemma}
\begin{proof}
  The analysis follows along the same lines as the proof of Lemma 31
  in \cite{less-embeddings}. First, we define the shorthand
  $\Ubbar=\diag(\xib)\U$, and use Lemma \ref{l:bernstein} to bound the
  spectral norm $\|\Ubbar\|$. Observe that from the definition of the
  sparsifier $\xib$ we have the following decomposition:
  $\Ubbar^\top\Ubbar =
  \sum_{i=1}^s\frac1{sp_{t_i}}\u_{t_i}\u_{t_i}^\top$, where $\u_i^\top$
  denotes the $i$th row of $\U$ and
  $t_1,...,t_s$ are the independently sampled indices from $p$.  Note
  that since $l_i(\A,\C)=\|\u_i\|^2$, we have that
  $p_i=\Omega(\|\u_i\|^2/s)$, so
  $\X_i=\frac1{sp_{t_i}}\u_{t_i}\u_{t_i}^\top-\frac1s\U^\top\U$
  satisfies $\E[\X_i]=\mathbf{0}$, 
  $\|\X_i\|=O(1)$, and $\E[\X_i^2] = O(1/s)\cdot \I$. So, using
  Lemma~\ref{l:bernstein} with $\sigma^2=R=O(1)$, for any $t\geq
  C\deff$ we have $\Pr\{\|\Ubbar\|^2\geq \sqrt t\}\leq \exp(-c\sqrt
  t)$, with $c=\Omega(1)$. Using the fact that $\|\Ubbar\|^2\leq
  \tr(\Ubbar^\top\Ubbar)\leq C\deff$ almost surely, it follows that
  the event $\Ec:\|\Ubbar\|^2\leq\min\{\sqrt t,C\deff\}$ has
  probability $1-\exp(-c(\sqrt t + t/\deff))$. Now, it suffces to
condition on $\xib$ and  apply the Hanson-Wright inequality
(Lemma~\ref{l:hanson-wright}), concluding that:
\begin{align*}
  \Pr\big\{\x^\top\Ubbar\Ubbar^\top\x\geq C\deff+t\mid\xib,\Ec\big\}
  &\leq 2\exp\bigg(-c\min\Big\{\frac{t^2}{\|\Ubbar\Ubbar\|_F^2},
    \frac{t}{\|\Ubbar\Ubbar^\top\|}\Big\}\bigg)
  \\
  &\leq 2\exp(-\Omega(\sqrt t + t/\deff)),
\end{align*}
which completes the proof.
\end{proof}
By appropriately integrating out the concentration inequality from
Lemma~\ref{l:restricted-hanson-wright}, as in Lemma~30 of
\cite{less-embeddings} but replacing $d$ with $\deff$, we can show the following matrix moment bound.
\begin{lemma}
  Under the assumptions of Lemma \ref{l:restricted-hanson-wright}, for
  all $p=2,3,...$ we have:
\begin{align*}
  \bigg\|\E\bigg[\Big(\U^\top\s\s^\top\U-\U^\top\U\Big)^p\bigg]\bigg\|\leq \frac{p!}{2}\cdot(C\deff)^{p-1}.
\end{align*}
\end{lemma}
\begin{proof}
  First, we bound the expression in terms of the quadratic form
  $\s^\top\U\U^\top\s$, so that we can use the concentration
  inequality from Lemma~\ref{l:restricted-hanson-wright}. To that end,
  we have:
  \begin{align*}
    \E\bigg[\Big(\U^\top\s\s^\top\U-\U^\top\U\Big)^p\bigg]
    &\preceq
      \E\bigg[\Big\|\U^\top\s\s^\top\U-\U^\top\U\Big\|^{p-2}\Big(\U^\top\s\s^\top\U-\U^\top\U\Big)^2\bigg]
    \\
    &\overset{(*)}\preceq \E\bigg[\Big(\s^\top\U\U^\top\s +
      \deff\Big)^{p-2}\Big(2(\U^\top\s\s^\top\U)^2 +
      2(\U^\top\U)^2\Big)\bigg]
    \\
    &\preceq
      2\,\E\bigg[\Big(\s^\top\U\U^\top\s +
      \deff\Big)^{p-1}\U^\top\s\s^\top\U\bigg] + 2\,\E\bigg[\Big(\s^\top\U\U^\top\s +
      \deff\Big)^{p-2}\bigg]\cdot \I,
  \end{align*}
where in $(*)$ we used the fact that function $f(x)=x^2$ is operator
convex.  Now, integrating out the concentration inequality from 
  Lemma~\ref{l:restricted-hanson-wright} for each of the two terms
  (following the steps of \cite[Appendix D.2]{less-embeddings}), we obtain the desired bound. 
\end{proof}

We can now apply Lemma \ref{l:bernstein} with
$\X_i=\frac1m\U^\top\s_i\s_i^\top\U-\frac1m\U^\top\U$, and
$\sigma^2=R=O(\deff/m)$, obtaining that:
\begin{align*}
  \Pr\big\{\|\gamma^{-1}\U^\top\S^\top\S\U-\U^\top\U\|\geq\eta\big\}\leq \deff\cdot\exp\big(-\Omega(\eta^2m/\deff)\big).
\end{align*}
Setting $m\geq C\deff\log(\deff/\delta)/\eta^2$, we obtain the desired
bound. Note that we must account again for the scaling
$\gamma=\frac{m}{m-\deff}$, which gets absorbed into the error
$\eta$.

Finally, observe that the conditions imposed on the sparsifier
$\xib$ in Lemma \ref{l:restricted-hanson-wright} encompass both LESS
and LESS-uniform embeddings. In the case of LESS, we can simply let
$s\approx_{1/2}\deff$, and then the condition on sparsifying
distribution is $p_i=\Omega(l_i(\A,\C)/\deff)$. On the other hand, for
LESS-uniform, as long as $s=\Omega(\tau\deff)$ where $\tau=\frac
n\deff\max_i l_i(\A,\C)$ is the coherence of $\A$, it follows that
$\frac1n=\Omega(l_i(\A,\C)/s)$ for all $i$'s, so a uniformly sparsifying
distribution suffices. 

\vspace{-3mm}
\subsection{LESS embeddings: Condition \ref{cond2}}
Here, we prove a result that is similar to the so-called Restricted
Bai-Silverstein inequality from \cite[Lemma 28]{less-embeddings}. Our
assumptions on the sparsifier are somewhat weaker, to account for
LESS-uniform embeddings and for the presence of regularization. 
\begin{lemma}
Let $\U=\A\H^{-\frac12}$ for $\H=\A^\top\A+\C$. Let $\xib$ be a $(p,s)$-sparsifier and $\x$ have
indepedent sub-Gaussian entries with mean zero and unit
variance. If $p_i= \Omega(l_i(\A,\C)/s)$ for all $i$, then for all
$d\times d$ psd
matrices $\B$, vector $\s=\x\circ\xib$ satisfies:
\begin{align*}
  \Var[\s^\top\U\B\U^\top\s] \leq O(1)\cdot \tr(\U\B^2\U^\top).
\end{align*}
\end{lemma}
\begin{proof}
  Let $\Ubbar = \diag(\xib)\U$. We start with a decomposition of the variance:
    \begin{align*}
    \Var[\s^\top\U\B\U^\top\s]
    &=\E\big[(\x^\top\Ubbar\B\Ubbar^\top\x
    -\tr(\Ubbar\B\Ubbar^\top)
    +    \tr(\Ubbar\B\Ubbar^\top) - \tr(\B))^2\big]\\
    &=\E\big[\Var[\x^\top\Ubbar\B\Ubbar^\top\x\mid\Ubbar]\big] +
      \Var[\tr(\Ubbar\B\Ubbar^\top)].
  \end{align*}
Recall that
  $\Ubbar^\top\Ubbar=\sum_{i=1}^s\frac1{sp_{t_i}}\u_{t_i}\u_{t_i}^\top$,
where $\u_i^\top$ is the $i$th row of $\U$ and  $p_i= \Omega(\|\u_i\|^2/s)$.
Then
  \begin{align*}
    \Var\big[\tr(\Ubbar\B\Ubbar^\top)\big]
    & =s\,\Var\bigg[\frac{\u_{t_1}^\top\B\u_{t_1}}{sp_{t_1}}\bigg]
        \leq\E\bigg[\frac{\|\u_{t_1}\|^2}{sp_{t_1}}\,\frac{\u_{t_1}^\top\B^2\u_{t_1}}{p_{t_1}}\bigg]
= O(1)\,\tr(\U\B^2\U^\top) ,
  \end{align*}
where we use that $\U^\top\U\preceq\I$.
    Next, we use the classical Bai-Silverstein inequality (Lemma
    \ref{l:bai-silverstein}):
    \begin{align*}
      \E\big[\Var[\x^\top\Ubbar\B\Ubbar^\top\x\mid\Ubbar]\big]
      &\leq O(1)\cdot \E\big[\tr\big((\Ubbar\B\Ubbar^\top)^2\big)\big]
 =O(1)\cdot\E\Big[\tr\Big(\Big(\sum_{i=1}^s\frac{1}{sp_{t_i}}\B\u_{t_i}\u_{t_i}^\top\Big)^2\Big)\Big]
      \\
      &=O(1)\,\tr(\U\B^2\U^\top) + O(1)\,\tr\big((\U\B\U^\top)^2\big),
    \end{align*}
    where the last step follows by breaking down the expanded square
    into the diagonal part and the cross-terms. Since
    $\tr\big((\U\B\U^\top)^2\big)\leq\tr(\U\B^2\U^\top)$, this
    completes the proof.
\end{proof}

\vspace{-3mm}
\section{Distributed Averaging for Newton-LESS}
\label{a:distributed}

An important property of the Gaussian Newton Sketch is that it
produces unbiased estimates of the exact Newton step. This is useful
in distributed settings, where we can construct multiple independent
estimates in parallel, and then produce an improved estimate by
averaging them together. Newton-LESS retains this unbiasedness
property, up to a small error, which also makes it amenable to
distributed averaging. This near-unbiasedness of LESS embeddings
follows from the characterization of the first inverse moment of the
sketched Hessian (see \cite{less-embeddings} and the first part of
Theorem \ref{t:structural}).

In this section, we show that the near-unbiasedness of LESS embeddings
can be combined with our new convergence analysis to provide improved
convergence rates for Distributed Newton-LESS:
\begin{align}
  \xbt_{t+1} = \xbt_t - \frac{\mu_t}q\sum_{i=1}^q
  \big(\Af{f_0}{\xbt_t}^\top\S_{t,i}^\top\S_{t,i}
  \Af{f_0}{\xbt_t} + \nabla^2g(\xbt_t)\big)^{-1}\nabla
  f(\xbt_t),\label{eq:distributed-sketch} 
\end{align}
where $\S_{t,i}$ are independently drawn LESS embedding matrices. To
adapt our analysis for this algorithm, we extend the characterization
from Lemma \ref{l:regularized-decomposition}.
\begin{lemma}\label{l:distributed-decomposition}
Fix $\H_t=\nabla^2 f(\xbt_t)$ and let $\xbt_{t+1}$ be as in
\eqref{eq:distributed-sketch} with $\S_{t,i}$ as in Lemma
\ref{l:structural-less} (i.e., sub-Gaussian, LESS or
LESS-uniform). Also, suppose that the exact Newton step $\x_{t+1}=\xbt_t-\mu_t\H_t^{-1}\g_t$ is a descent direction,
i.e., $\|\Delta_{t+1}\|_{\H_t}\leq\|\Deltat_t\|_{\H_t}$ where
$\Delta_{t+1}=\x_{t+1}-\x^*$ and $\Deltat_t=\xbt_t-\x^*$. Then, letting $\rho
=\tfrac{\deff(\xbt_t)}{m-\tdeff(\xbt_t)}$, we have: 
\begin{align*}
  \E_{q\delta}\,\|\Deltat_{t+1}\|_{\H_t}^2 =
  \|\Delta_{t+1}\|_{\H_t}^2 +
  \tfrac\rho q\,\|\x_{t+1}-\xbt_t\|_{\nabla^2 f_0(\xbt_t)}^2 \pm
  O\big(\tfrac{\sqrt \deff}{m}\big)\|\Deltat_t\|_{\H_t}^2.
\end{align*}
\end{lemma}
\begin{proof}
The proof is analogous to the proof of Lemma
\ref{l:regularized-decomposition}, except we must replace $\Qbt$ with
\begin{align*}
  \Qbar=\frac1q\sum_{i=1}^q\Qbt_i,\qquad\text{for}\qquad\Qbt_i =
  \H_t^{\frac12}(\Af{f_0}{\xbt_t}^\top\S_{t,i}^\top\S_{t,i}\Af{f_0}{\xbt_t}+
  \nabla^2 g(\xbt_t))^{-1}\H_t^{\frac12}.
\end{align*}
Each $\Qbt_i$ satisfies the first and second moment
characterizations from Theorem \ref{t:structural}. Let
$\Ec=\bigwedge_{i=1}^q\Ec_i$ denote the intersection of the
corresponding $1-\delta$ probability events. Then,
$\|\E_{\Ec}[\Qbar] - \I\| \leq O\big(\tfrac{\sqrt \deff}{m}\big)$ and also:
\begin{align*}
\E_{\Ec}[\Qbar^2] - \I
  &=\frac1{q^2}\sum_{i=1}^q\E_{\Ec}[\Qbt_i^2] - \sum_{i\neq j}
    \E_{\Ec}[\Qbt_i]\E_{\Ec}[\Qbt_j] - \I 
  \\
  &=\frac1q\big(\E_{\Ec}[\Qbt_1^2]-\I) + \frac{q(q-1)}{q^2}\big(\E_{\Ec}[\Qbt_1]^2-\I\big),
\end{align*}
so using that $\|\E_{\Ec}[\Qbt_1^2]-(\I+\rho\,\U^\top\U)\|\leq
O\big(\frac{\sqrt \deff}{m}\big)$, where
$\U=\Af{f_0}{\xbt}\H_t^{-\frac12}$, we get: 
\begin{align*}
  \|\E_{\Ec}[\Qbar^2] - (\I+\tfrac\rho q\U^\top\U)\|\leq O\big(\tfrac{\sqrt \deff}{m}\big).
\end{align*}
The rest of the proof follows identically as in
Lemma~\ref{l:regularized-decomposition}, using $\Qbar$ in place of
$\Qbt$. Note that, using the union bound,  we can show that
the probability of $\Ec$ is at least $1-q\delta$.
\end{proof}
From this lemma, repeating the local convergence analysis of
Theorem~\ref{t:regularized-full},  we obtain that in the neighborhood
of $\x^*$, setting $\mu_t=\frac{q(m-\tdeff)}{\deff+q(m-\tdeff)}$, Distributed Newton-LESS achieves:
  \begin{align*}
\bigg(\E_{qT\delta}\,\frac{\|\xbt_T-\x^*\|_{\H}^2}{\|\xbt_0-\x^*\|_{\H}^2}\bigg)^{1/T}
\leq\ \frac{\deff}{\deff+q(m-\tdeff)}+ O\Big(\frac{\sqrt\deff}{m}\Big),
  \end{align*}
  and for the unregularized case, where $\deff=\tdeff=d$, we can
  obtain a matching lower bound on the convergence rate. This shows
  that the convergence rate of Newton-LESS can be substantially improved
via distributed averaging.

\vspace{-3mm}
\section{Additional Numerical Experiments and Implemention Details}
\label{a:additionalnumericalexperiments}

Experiments are implemented in Python using the Pytorch module on
Amazon Sagemaker instances with CPUs with $256$ gigabytes of memory
and GPUs NVIDIA Tesla V100. Code is publicly available at
\url{https://github.com/lessketching/newtonsketch}.

\vspace{-3mm}
\subsection{Sketching matrices}
\label{a:sketchingmatrices}

Given a data matrix $\A \in \real^{n \times d}$, we follow the
procedure described in \cite{drineas2012fast} (see Algorithm~1
therein) for fast approximation of the leverage scores. We use these
approximate leverage scores to compute the RSS-lev-score embedding. 

For LESS embeddings, we report the performance of the computationally
most efficient method between using the approximate leverage scores,
or, pre-processing the data matrix $\A$ by a Hadamard matrix $\H$ and
then using a uniformly sparsified sketching matrix. Preprocessing with
a Hadamard matrix uniformizes the leverage scores, so this second
option is a valid implementation of a LESS embedding (see
\cite{less-embeddings} for a detailed discussion). We found this
second option to be the fastest method in practice. 

For a LESS-uniform embedding $\S \in \R^{m \times n}$, we fix a number
of non-zero entries per row to be $s$. For each row $\s_i^\top$, we
sample $s$ indices $\{i_1, \dots, i_s\}$ in $\{1,\dots,n\}$ uniformly
at random with replacement. Each entry $S_{i i_j}$ is then chosen
uniformly at random in $\{\pm (n/m s)^{1/2}\}$. We choose to sample
with replacement for maximal computational efficiency. In our
experiments, the number of non-zero entries $s$ is small in comparison
to the sample size $n$, so the probability of sampling twice the
same index remains very small.

\vspace{-3mm}
\subsection{Datasets}
\label{a:datasets}

The high-coherence synthetic data matrix $\A$ is generated as
follows. We construct a covariance matrix $\mathbf{\Sigma} \in \R^{d
  \times d}$ with entries $\mathbf{\Sigma}_{ij} = 2 \cdot
0.5^{|i-j|}$. The rows $\mathbf{a}_i$ of $\A$ are then sampled
independently as $\mathbf{a}_i \sim \mathbf{g}_i / \sqrt{z_i}$ where
$\mathbf{g}_i \sim \mathcal{N}(0,\Sigma)$ and $z_i$ follows a Gamma
distribution with shape $1/2$ and scale $2$. We use $n=16384$ and
$d=256$. 

We downloaded the Musk and CIFAR-10 datasets from
\url{https://www.openml.org/}. The Musk data matrix has size $n=4096$
and $d=256$. The sample size of the CIFAR-10 dataset is $n=50000$. We
transform each image using a random features map that approximates the
Gaussian kernel $\exp(-\gamma x^2)$ with bandwith $\gamma=0.02$, and
we use $d=2000$ random cosine components. Regarding the labels, we
partition the ten classes of CIFAR-10 into two groups with
corresponding labels $0$ and $1$.  

For the  WESAD  dataset \cite{schmidt2018introducing},  we  used  the  data  obtained  from  the  E4  Empatica  device  and we filtered the data over windows of one second.%
\footnote{We refer to the public repository for implementation details about subsampling the signal, \url{https://github.com/WJMatthew/WESAD/blob/master/data_wrangling.py}.}  
This results in a sample size $n=262144$. Then we applied a random features map that approximates the Gaussian kernel $\exp(-\gamma x^2)$ with $\gamma=0.01$ and we use $d=2000$ components.

\vspace{-3mm}
\subsection{Least squares regression}

We consider first least squares regression. On Figure \ref{fig:ls-all},
we report the relative error versus number of
iterations, as well as the relative error versus wall-clock time for
the Newton Sketch. We compare to Gaussian embeddings, the SRHT,
uniformly random row sampling matrices (RRS) and random row sampling
based on approximate leverage scores (RRS-lev-scores). As predicted by
our theory, LESS embeddings have convergence rate scaling as
$d/m$. This is similar to the convergence rate of the Newton sketch
with Gaussian embeddings \cite{lacotte2019faster}. We also observe
similar convergence for the SRHT, which is not explained by
existing worst-case theory \cite{tropp2011improved}, but it matches the
predictions based on high-dimensional asymptotic analysis of the SRHT
\cite{lacotte2020limiting}. Except for CIFAR-10, RSS and 
RSS-lev-scores have weaker convergence rates. This suggests that the
CIFAR-10 data matrix has low coherence. Except for the high-coherence
synthetic data matrix for which the convergence rate is slightly worse
than $d/m$, using LESS with a uniformly random sparsifier does not
affect the convergence rate. Here, we implement LESS-uniform with $d$
non-zero entries per row subsampled uniformly at random. Importantly,
LESS-uniform offers significant speed-ups over other sketching
matrices.  

Note that some curves stop earlier than
others (e.g., RRS) on the wall-clock time versus error plots, because we
run the Newton sketch for each embedding for 
a fixed number of iterations. 

\begin{figure}[!ht]
	\centering
	\begin{minipage}[t]{\textwidth}
		\centering
		\includegraphics[width=.4\linewidth]{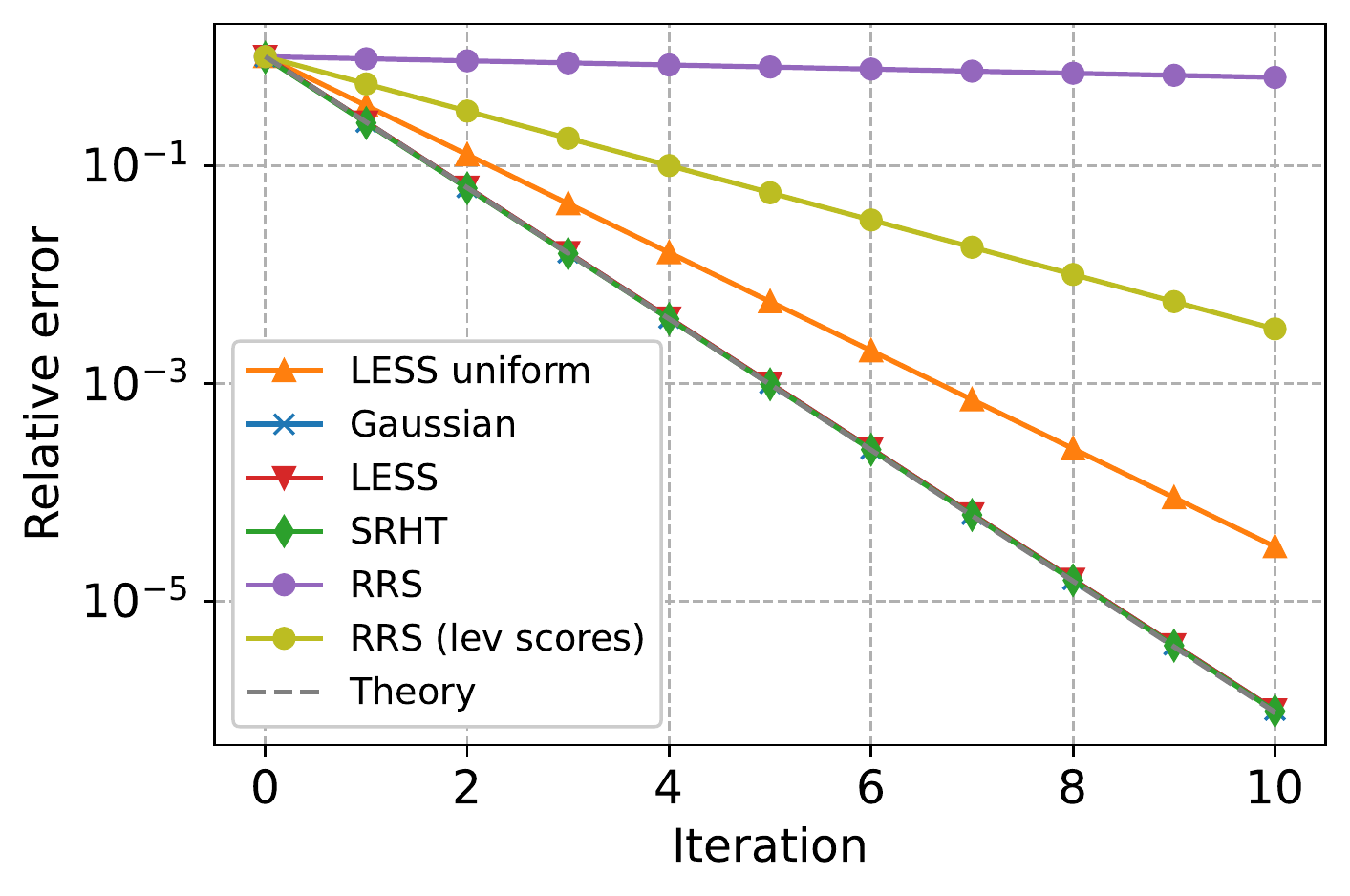}
		\includegraphics[width=.4\linewidth]{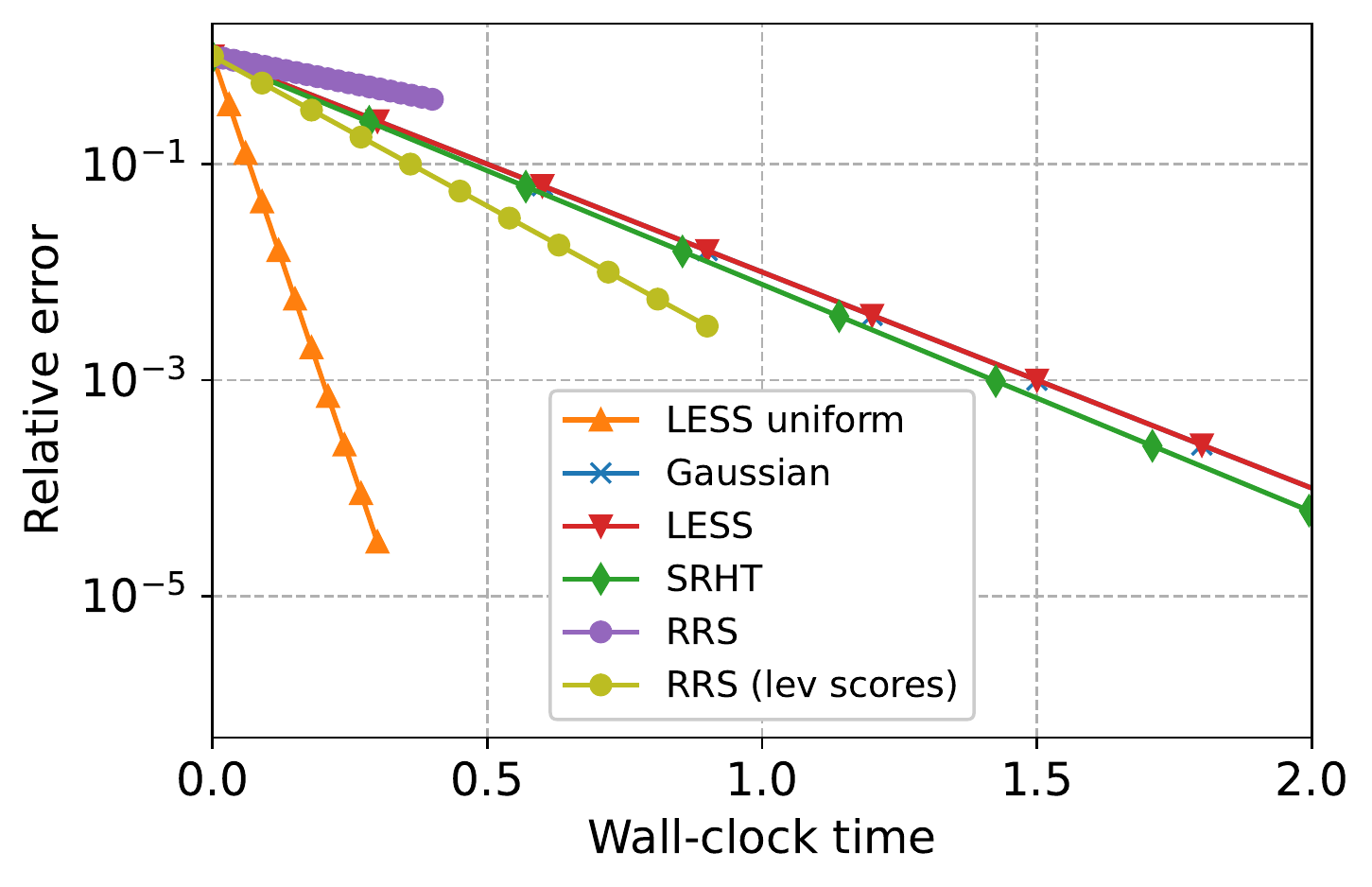}
                \vspace{-5mm}
	\caption*{\footnotesize{~\quad(a) High-coherence synthetic matrix}}
	\end{minipage}
        \vspace{2mm}
        
	\begin{minipage}[t]{\textwidth}
		\centering
		\includegraphics[width=.4\linewidth]{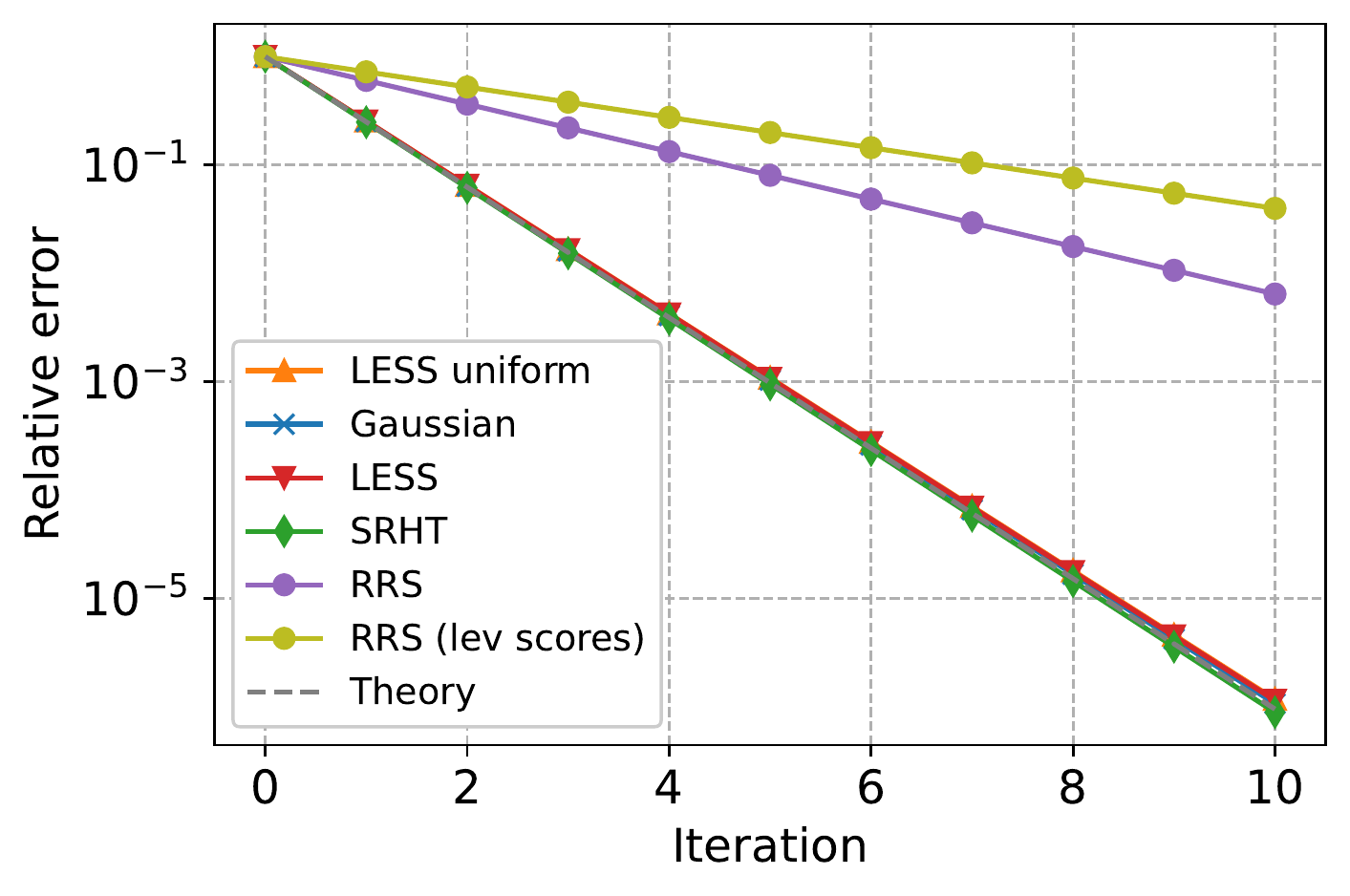}
		\includegraphics[width=.4\linewidth]{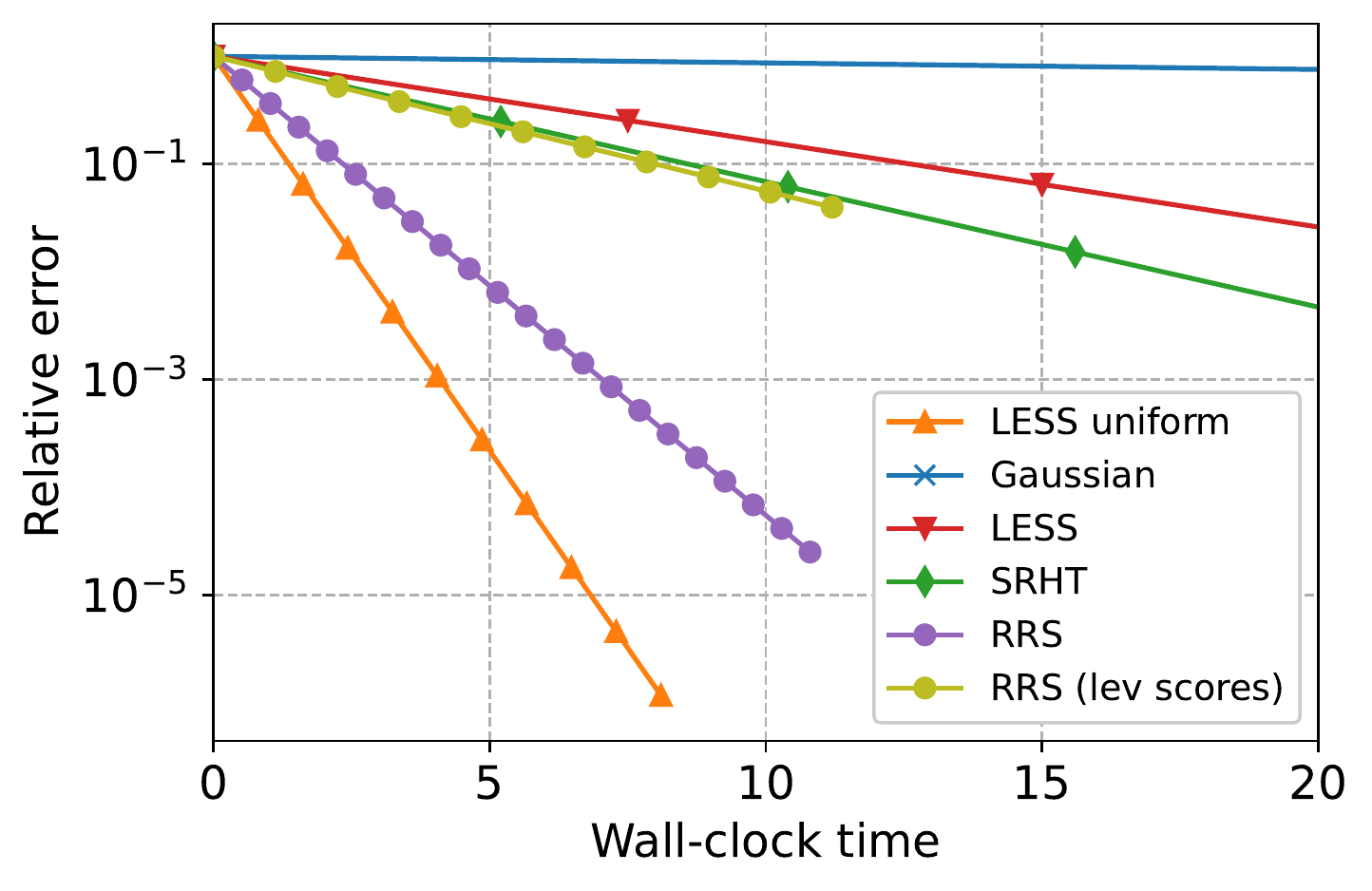}
                \vspace{-5mm}
                \caption*{\footnotesize{(b) WESAD dataset}}
	\end{minipage}
\vspace{2mm}

\begin{minipage}[t]{\textwidth}
		\centering
		\includegraphics[width=.4\linewidth]{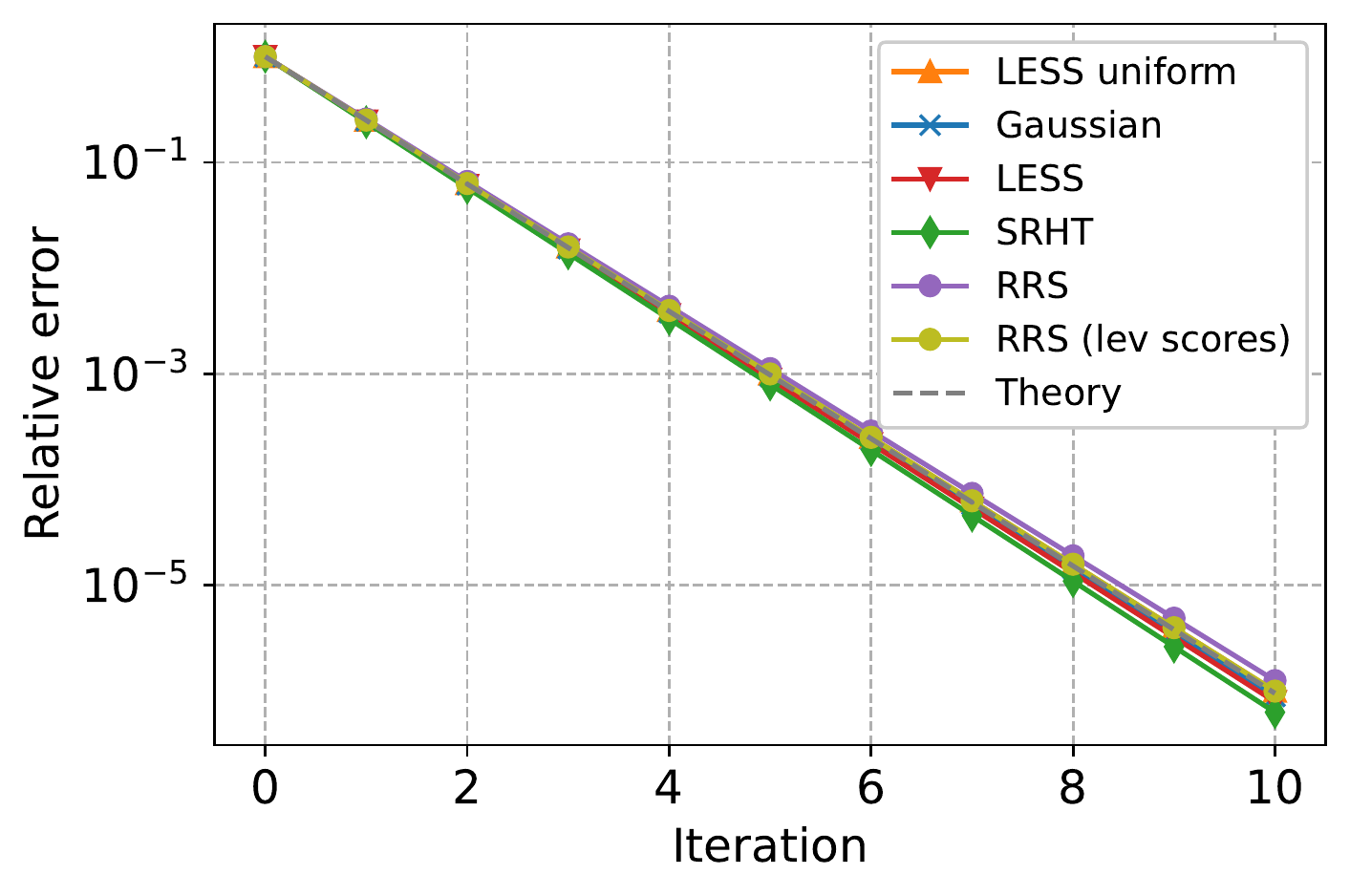}
		\includegraphics[width=.4\linewidth]{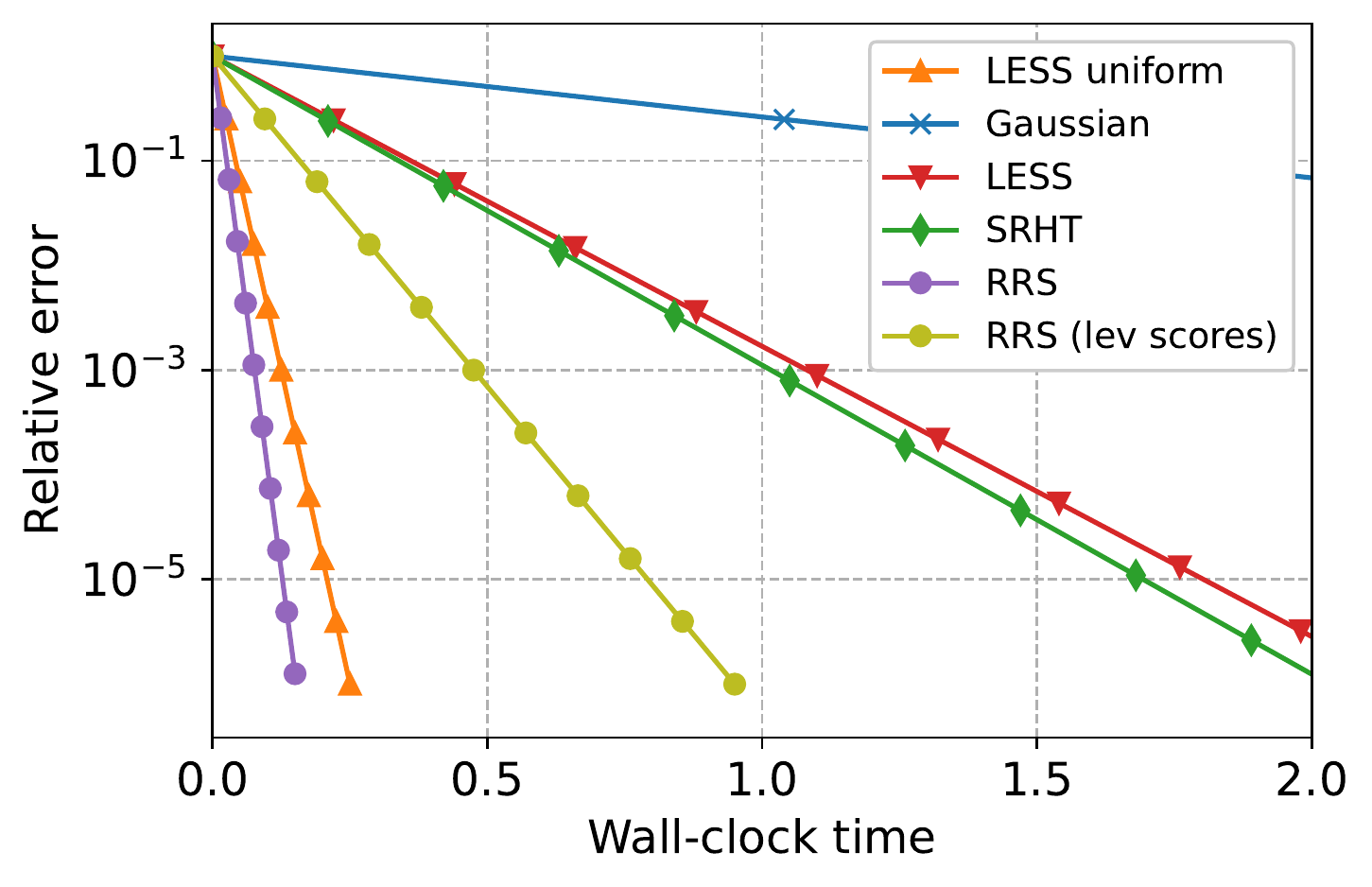}
                \vspace{-5mm}
                \caption*{\footnotesize{(c) CIFAR-10 dataset}}
	\end{minipage}
        \vspace{2mm}
        
	\begin{minipage}[t]{\textwidth}
		\centering
		\includegraphics[width=.4\linewidth]{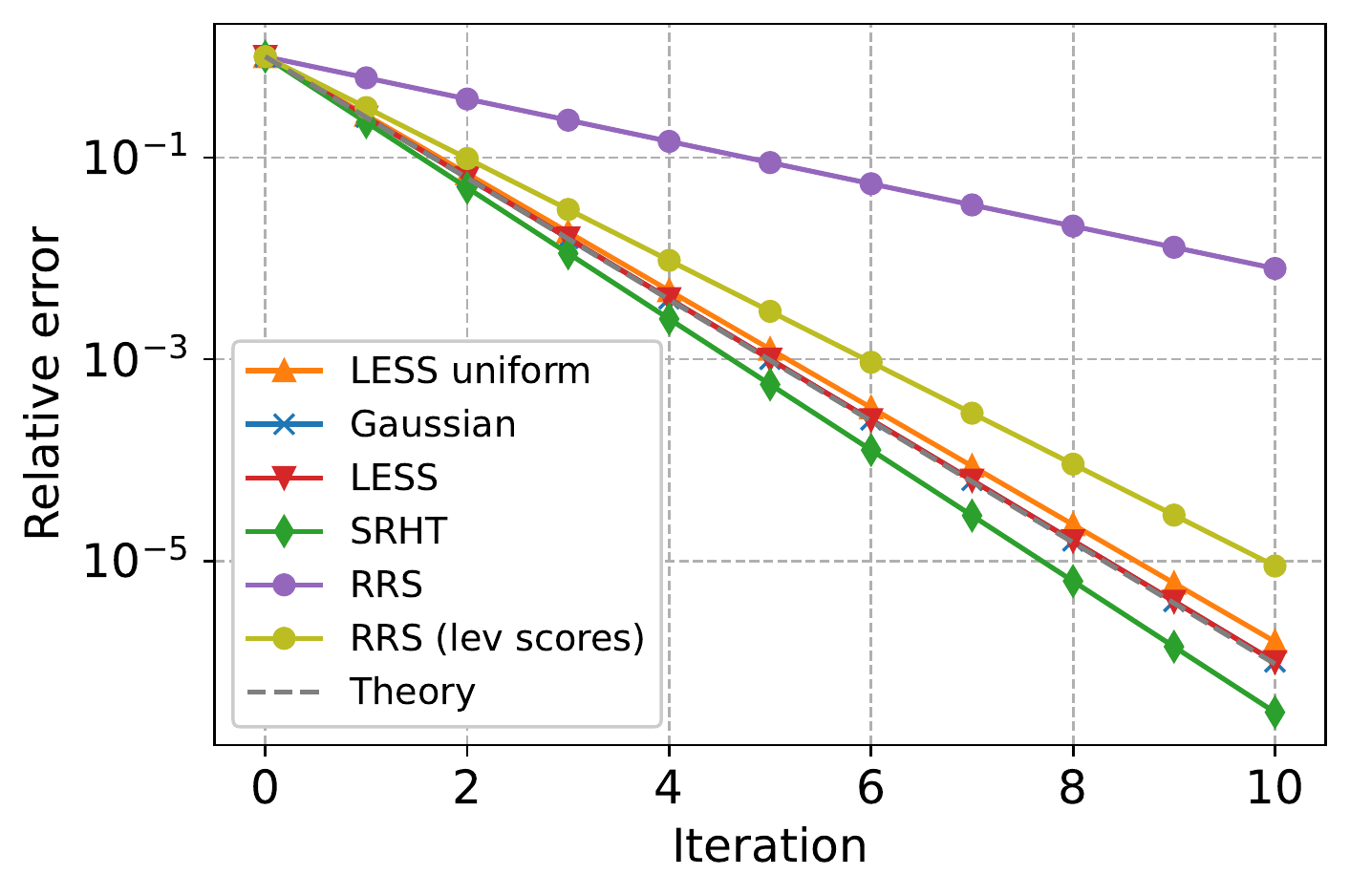}
		\includegraphics[width=.4\linewidth]{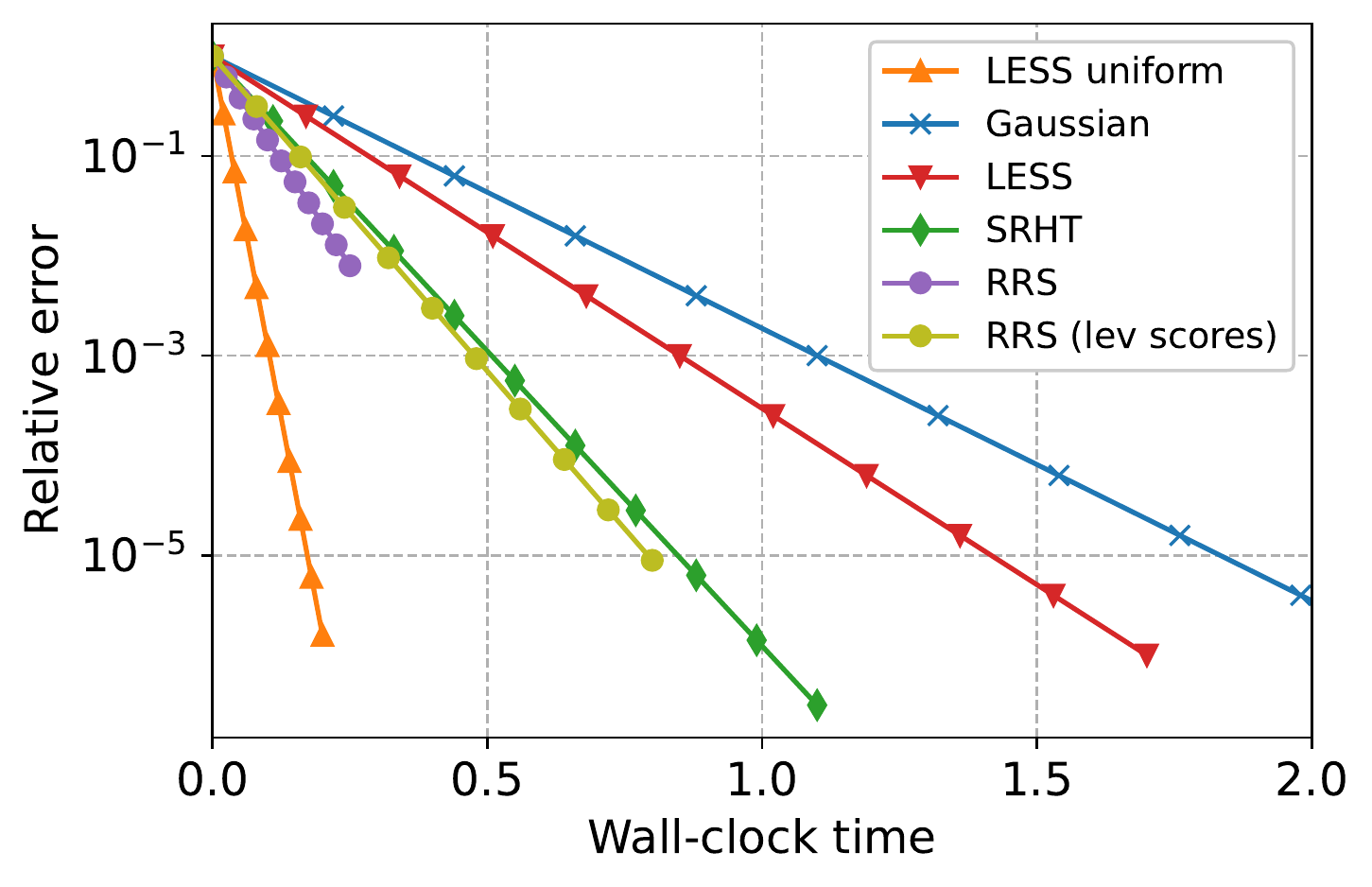}
                \vspace{-5mm}
                \caption*{\footnotesize{(d) Musk dataset}}
	\end{minipage}
        \caption{Newton sketch for least squares regression. We use the sketch size $m=4d$ for all experiments. Results are averaged over $10$ trials.}
        \label{fig:ls-all}
\end{figure}

\vspace{-3mm}
\subsection{Regularized least squares and effective dimension}
\label{a:regularizedls}

On Figure \ref{fig:regularized-ls}, we report the error versus number of iterations of the Newton Sketch for regularized least squares regression. These results illustrate in particular our theoretical predictions: the convergence rate of Newton-LESS is upper bounded by $d_\text{eff}/m$. In fact, Newton-LESS has the same convergence rate as the Newton Sketch with dense Gaussian embeddings.

\begin{figure}[!ht]
	\centering
	\begin{minipage}[t]{0.4\textwidth}
		\centering
		\includegraphics[width=\linewidth]{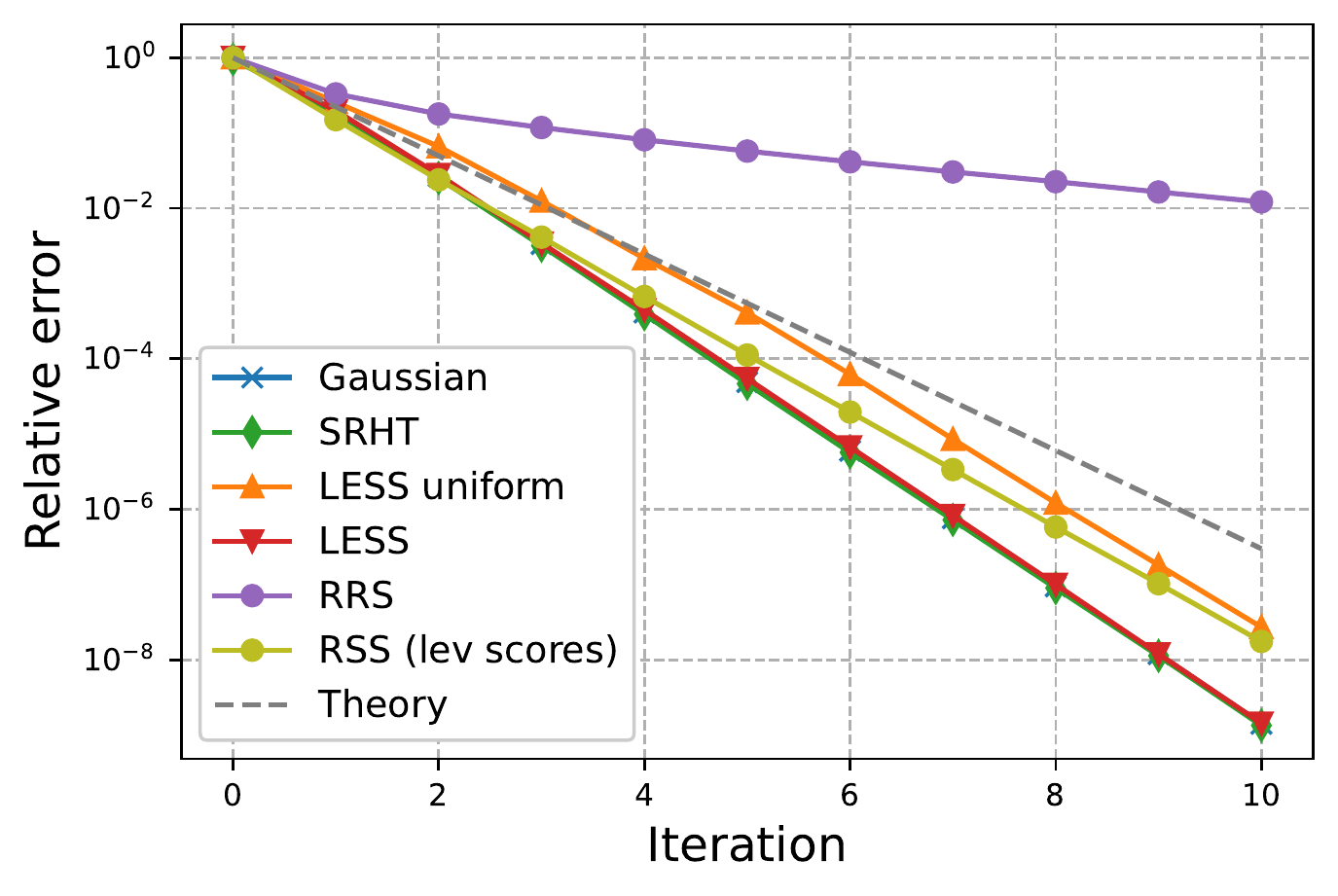}
                \vspace{-7mm}
		\caption*{\footnotesize{(a)  High-coherence synthetic matrix }}
	\end{minipage}
	\begin{minipage}[t]{0.4\textwidth}
		\centering
		\includegraphics[width=\linewidth]{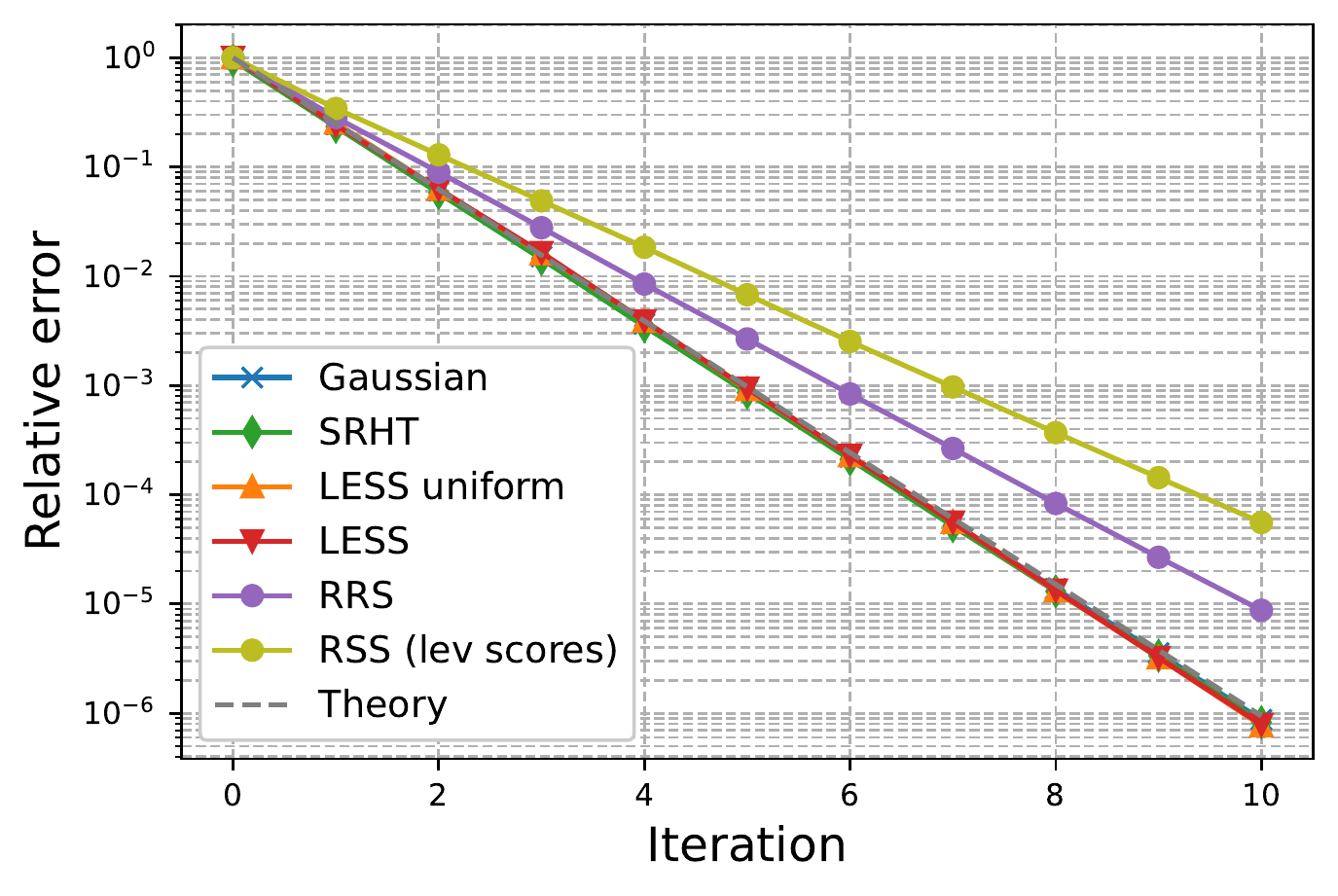}
                \vspace{-7mm}
		\caption*{\footnotesize{(b) CIFAR-10 dataset }}
	\end{minipage}
        \vspace{5mm}
        
	\begin{minipage}[t]{0.4\textwidth}
		\centering
		\includegraphics[width=\linewidth]{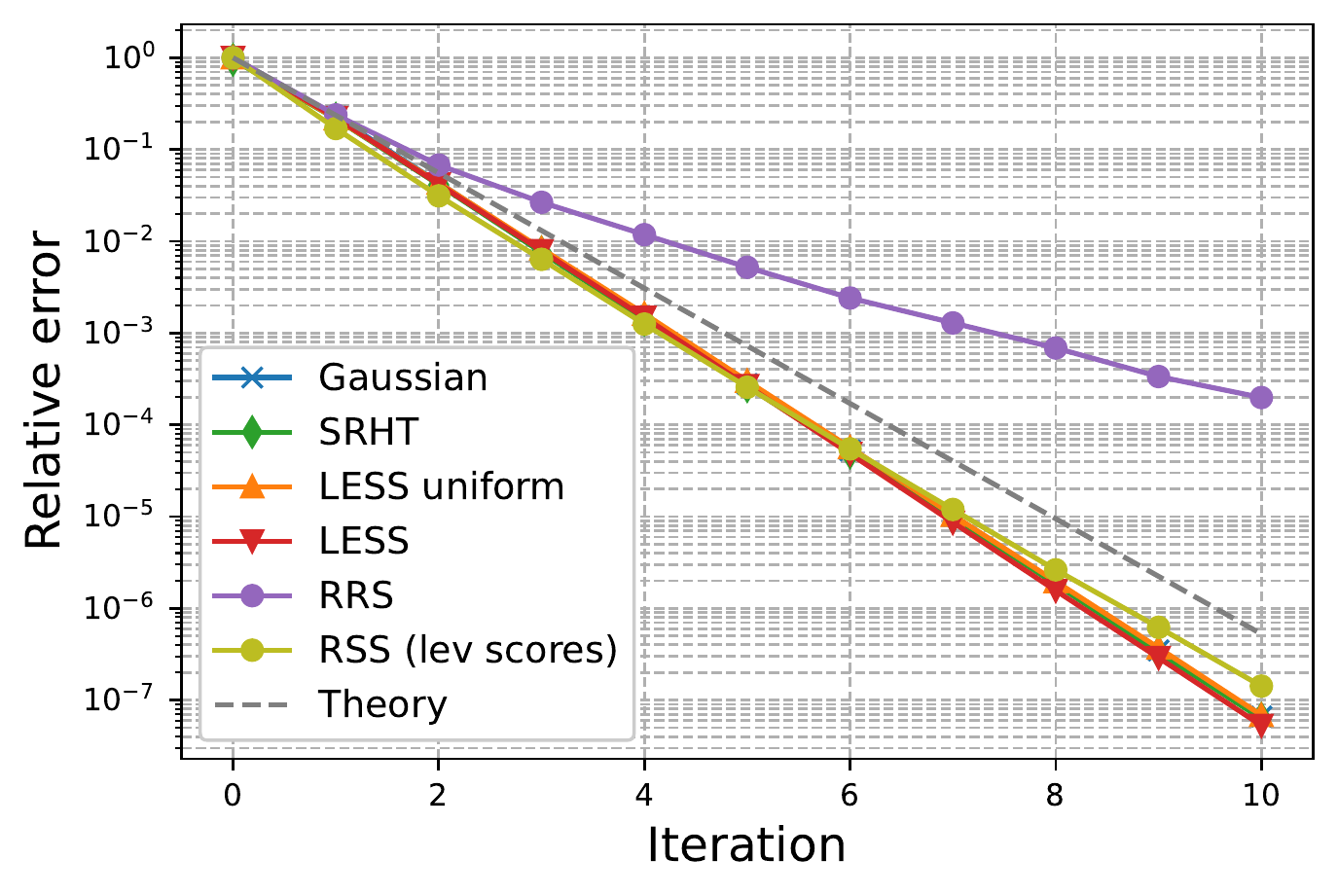}
                \vspace{-7mm}
		\caption*{\footnotesize{(c)  Musk dataset }}
	\end{minipage}
	\begin{minipage}[t]{0.4\textwidth}
		\centering
		\includegraphics[width=\linewidth]{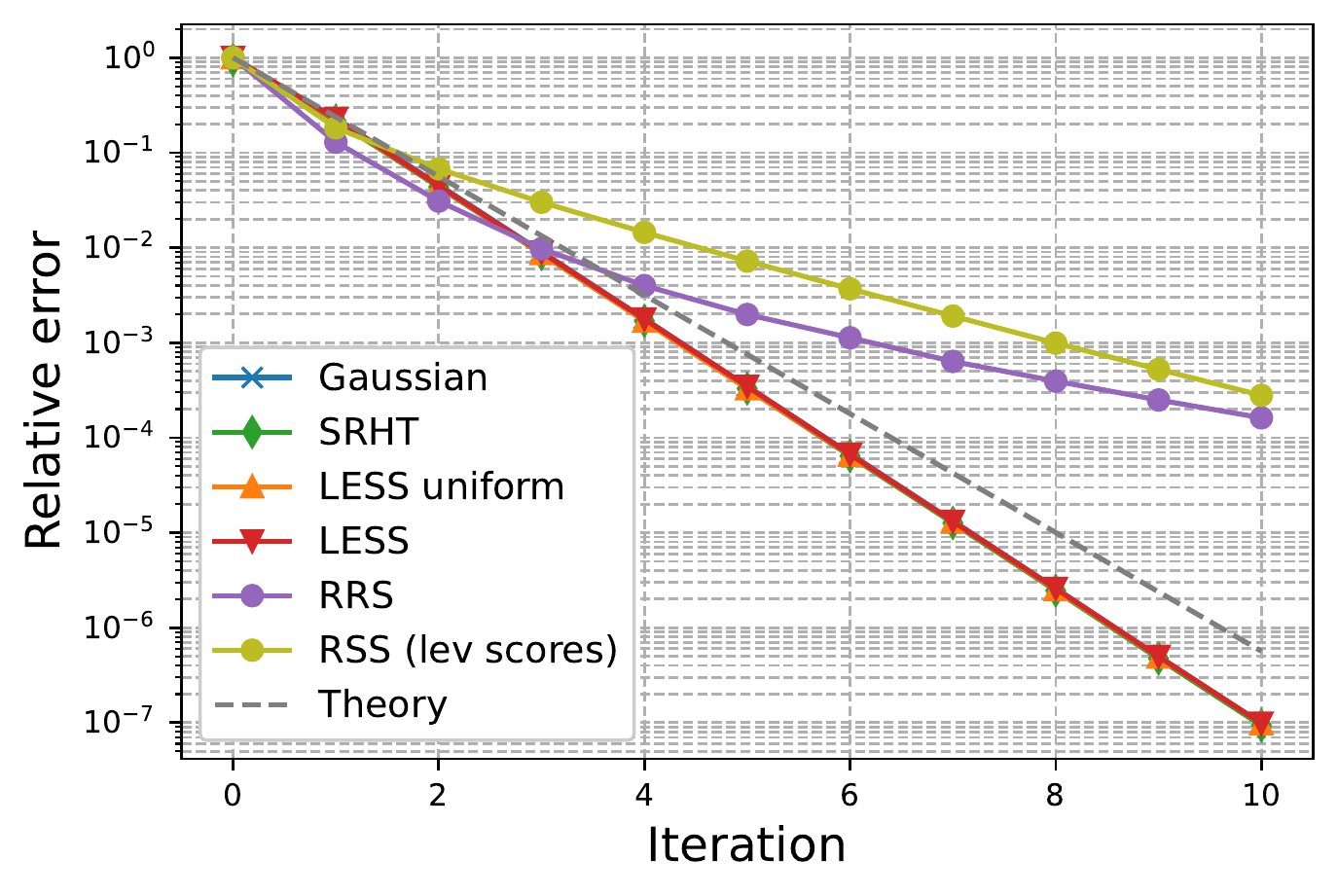}
                \vspace{-7mm}
		\caption*{\footnotesize{(d)  WESAD dataset }}
	\end{minipage}
	\caption{Newton Sketch for regularized least squares regression. We use the sketch size $m=4 d_\text{eff}$ for all experiments. Results are averaged over $10$ trials.}
	\label{fig:regularized-ls}
\end{figure}

\end{document}